\documentclass{amsart}
\usepackage{amsmath}
\usepackage{amsfonts}
\usepackage{amsthm}
\usepackage{amssymb}
\usepackage{enumitem}
\usepackage{xypic}

\usepackage{tikz}
\usetikzlibrary{arrows,positioning,fit,matrix,shapes.geometric,external}

\newcommand*\pgfdeclareanchoralias[3]{%
  \expandafter\def\csname pgf@anchor@#1@#3\expandafter\endcsname
     \expandafter{\csname pgf@anchor@#1@#2\endcsname}}

\tikzset{
circnode/.style={
  circle, draw=red, very thin, outer sep=0.025em, minimum size=2em,
  fill=red, text centered},
integral/.style={
  regular polygon, regular polygon sides=3, shape border rotate=180, draw=black, very thick,
  outer sep=0.025em, inner sep=0, minimum size=2em, fill=blue!5, text centered},
multiply/.style={
  regular polygon, regular polygon sides=3, shape border rotate=180, draw=black, very thick,
  outer sep=0.025em, inner sep=0, minimum size=2em, fill=blue!5, text centered},
upmultiply/.style={
  regular polygon, regular polygon sides=3, draw=black, very thick,
  outer sep=0.025em, inner sep=0, minimum size=2em, fill=blue!5, text centered},
zero/.style={
  circle, draw=black, very thick, minimum size=0.15cm, fill=black,
  inner sep=0, outer sep=0},
hole/.style={
  circle, draw=white, very thick, minimum size=0.25cm, fill=white,
  inner sep=0, outer sep=0},
bang/.style={
  circle, draw=black, very thick, minimum size=0.15cm, fill=green!10,
  inner sep=0, outer sep=0},
delta/.style={
  regular polygon, regular polygon sides=3, minimum size=0.4cm, inner
  sep=0, outer sep=0.025em, draw=black, very thick, fill=green!10},
codelta/.style={
  regular polygon, regular polygon sides=3, shape border rotate=180, minimum size=0.4cm,
  inner sep=0, outer sep=0.025em, draw=black, very thick, fill=green!10},
plus/.style={
  regular polygon, regular polygon sides=3, shape border rotate=180, minimum size=0.4cm,
  inner sep = 0, outer sep=0.025em, draw=black, very thick, fill=black},
coplus/.style={
  regular polygon, regular polygon sides=3, minimum size=0.4cm,
  inner sep = 0, outer sep=0.025em, draw=black, very thick, fill=black},
sqnode/.style={
  regular polygon, regular polygon sides=4, minimum size=2.6em,
  draw=black, very thick, inner sep=0.2em, outer sep=0.025em,
  fill=yellow!10, text centered},
blackbox/.style={
  regular polygon, regular polygon sides=4, minimum size=2.6em,
  draw=black, very thick, inner sep=0.2em, outer sep=0.025em, fill=black},
bigcirc/.style={
  circle, draw=black, very thick, text width=1.6em, outer sep=0.025em,
  minimum height=1.6em, fill=blue!5, text centered}
 }
\pgfdeclareanchoralias{regular polygon}{corner 1}{io}
\pgfdeclareanchoralias{regular polygon}{corner 2}{left out}
\pgfdeclareanchoralias{regular polygon}{corner 3}{right out}
\pgfdeclareanchoralias{regular polygon}{corner 3}{left in}
\pgfdeclareanchoralias{regular polygon}{corner 2}{right in}

\usepackage{color}

\usepackage[hang,small,bf]{caption}
\definecolor{myurlcolor}{rgb}{0.6,0,0}
\definecolor{mycitecolor}{rgb}{0,0,0.8}
\definecolor{myrefcolor}{rgb}{0,0,0.8}
\usepackage[pagebackref]{hyperref}
\hypersetup{colorlinks, linkcolor=myrefcolor, citecolor=mycitecolor, urlcolor=myurlcolor}

\newcommand{\Cob}{\mathrm{Cob}}
\newcommand{\Hilb}{\mathrm{Hilb}}

\newcommand{\Fin}{\mathrm{Fin}}

\newcommand{\R}{{\mathbb R}}
\newcommand{\C}{{\mathbb C}}

\newcommand{\N}{{\mathbb N}}

\newcommand{\B}{{\mathbb B}}
\newcommand{\Z}{{\mathbb Z}}

\newcommand{\maps}{\colon} 
\newcommand{\iso}{\cong}

\newcommand{\asrelto}{\nrightarrow}
\newcommand{\of}{\circ}
\newcommand{\ev}{\mathrm{ev}}
\newcommand{\End}{\mathrm{End}}

\newcommand{\Relk}{\mathrm{FinRel}_k}
\newcommand{\Vectk}{\mathrm{FinVect}_k}
\newcommand{\Mat}{\mathrm{Mat}}
\newcommand{\SV}{{\mathbb {SV}}}

\newtheorem{thm}{Theorem}

\newtheorem{lemma}[thm]{Lemma}

\newcommand{\Define}[1]{{\bf \boldmath{#1}}}

\title{Categories in Control }
\author{John C.~Baez}
\address{Department of Mathematics\\ 
University of California\\ 
Riverside CA 92521\\
USA \\
and Centre for Quantum Technologies\\ 
National University of Singapore\\ 
Singapore 117543}
\email{baez@math.ucr.edu}
\author{Jason Erbele}
\address{Department of Mathematics\\ 
University of California\\ 
Riverside CA 92521\\
USA}
\email{erbele@math.ucr.edu}
\date{May 27, 2014 (revised May 20, 2015)}

\begin{document}

\begin{abstract}
Control theory uses `signal-flow diagrams' to describe processes where real-valued functions of 
time are added, multiplied by scalars, differentiated and integrated, duplicated and deleted.  
These diagrams can be seen as string diagrams for the symmetric monoidal category \(\Vectk\) of 
finite-dimensional vector spaces over the field of rational functions \(k = \R(s)\), where the 
variable \(s\) acts as differentiation and the monoidal structure is direct sum rather than the 
usual tensor product of vector spaces.  For any field \(k\) we give a presentation of \(\Vectk\) in 
terms of the generators used in signal-flow diagrams.  A broader class of signal-flow diagrams also 
includes `caps' and `cups' to model feedback.  We show these diagrams can be seen as string 
diagrams for the symmetric monoidal category \(\Relk\), where objects are still finite-dimensional 
vector spaces but the morphisms are linear relations.  We also give a presentation for \(\Relk\).  
The relations say, among other things, that the 1-dimensional vector space \(k\) has two special 
commutative \(\dagger\)-Frobenius structures, such that the multiplication and unit of either one 
and the comultiplication and counit of the other fit together to form a bimonoid.  This sort of 
structure, but with tensor product replacing direct sum, is familiar from the `ZX-calculus' obeyed 
by a finite-dimensional Hilbert space with two mutually unbiased bases.
\end{abstract}

\maketitle

\section{Introduction}
\label{intro}

Control theory is the branch of engineering that focuses on manipulating `open systems'---systems
with inputs and outputs---to achieve desired goals.  In control theory, `signal-flow diagrams' are
used to describe linear ways of manipulating signals, which we will take here to be smooth
real-valued functions of time \cite{Friedland}.  For a category theorist, at least, it is natural to
treat signal-flow diagrams as string diagrams in a symmetric monoidal category \cite{JS1,JS2}.  This
forces some small changes of perspective, which we discuss below, but more important is the
question: \emph{which symmetric monoidal category?}   

We shall argue that the answer is: the category \(\Relk\) of finite-dimensional vector spaces over a
certain field \(k\), but with \emph{linear relations} rather than linear maps as morphisms, and
\emph{direct sum} rather than tensor product providing the symmetric monoidal structure.  We use the
field \(k = \R(s)\) consisting of rational functions in one real variable \(s\).  This variable has 
the meaning of differentation.  A linear relation from \(k^m\) to \(k^n\) is thus a system of linear
constant-coefficient ordinary differential equations relating \(m\) `input' signals and \(n\)
`output' signals.

Our main goal is to provide a complete `generators and relations' picture of this symmetric monoidal
category, with the generators being familiar components of signal-flow diagrams.  It turns out that
the answer has an intriguing but mysterious connection to ideas that are familiar in the
diagrammatic approach to quantum theory.   Quantum theory also involves linear algebra, but it uses
linear maps between Hilbert spaces as morphisms, and the tensor product of Hilbert spaces provides
the symmetric monoidal structure.

We hope that the category-theoretic viewpoint on signal-flow diagrams will shed new light on control
theory.  However, in this paper we only lay the groundwork.  In Section \ref{sigflow} we introduce 
signal-flow diagrams and summarize our main results.  In Section \ref{finvect} we use 
signal-flow diagrams to give a presentation of \(\Vectk\), the symmetric monoidal
category of finite-dimensional vector spaces and linear maps.    In Section \ref{finrel} we use them
to give a presentation of \(\Relk\).  In Section \ref{example} we discuss a well-known example from 
control theory: an inverted pendulum on a cart.   Finally, in Section \ref{conclusions} we compare 
our results to subsequent work of Bonchi--Soboci\'nski--Zanasi \cite{BSZ1,BSZ2} and Wadsley--Woods 
\cite{WW}.

\section{Signal-flow diagrams}
\label{sigflow}

There are several basic operations that one wants to perform when manipulating signals.  The
simplest is multiplying a signal by a scalar.  A signal can be amplified by a constant factor:  
\[  f \mapsto cf  \]
where \(c \in \R\).  We can write this as a string diagram:
  
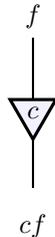
\begin{center}
   \begin{tikzpicture}[thick]
   \node[coordinate] (in) [label=\(f\)] {};
   \node [multiply] (mult) [below of=in] {\(c\)};
   \node[coordinate] (out) [below of=mult, label={[shift={(0,-0.8)}]\(cf\)}] {};

   \draw (in) -- (mult) -- (out);
   \end{tikzpicture}
\end{center}
Here the labels \(f\) and \(cf\) on top and bottom are just for explanatory purposes and not really
part of the diagram.  Control theorists often draw arrows on the wires, but this is unnecessary from
the string diagram perspective.  Arrows on wires are useful to distinguish objects from their
duals, but ultimately we will obtain a compact closed category where each object is its own dual, so
the arrows can be dropped.  What we really need is for the box denoting scalar multiplication to
have a clearly defined input and output.  This is why we draw it as a triangle.  Control theorists
often use a rectangle or circle, using arrows on wires to indicate which carries the input \(f\) and
which the output \(c f\).

A signal can also be integrated with respect to the time variable:
\[ f \mapsto \int f .\]
Mathematicians typically take differentiation as fundamental, but engineers sometimes prefer 
integration, because it is more robust against small perturbations.  In the end it will not matter 
much here.  We can again draw integration as a string diagram:
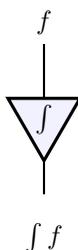
\begin{center}
   \begin{tikzpicture}[thick]
   \node[coordinate] (in) [label=\(f\)] {};
   \node [integral] (int) [below of=in] {\(\int\)};
   \node[coordinate] (out) [below of=int, label={[shift={(0.0,-0.9)}]\(\int f\)}] {};

   \draw (in) -- (int) -- (out);
   \end{tikzpicture}
  \end{center}
Since this looks like the diagram for scalar multiplication, it is natural to extend \(\R\) to
\(\R(s)\), the field of rational functions of a variable \(s\) which stands for differentiation.
Then differentiation becomes a special case of scalar multiplication, namely multiplication by
\(s\), and integration becomes multiplication by \(1/s\).  Engineers accomplish the same effect with
Laplace transforms, since differentiating a signal \(f\) is equivalent to multiplying its Laplace
transform 
\[   (\mathcal{L}f)(s) = \int_0^\infty f(t) e^{-st} \,dt  \]
by the variable \(s\).   Another option is to use the Fourier transform: differentiating \(f\) is
equivalent to multiplying its Fourier transform 
\[   (\mathcal{F}f)(\omega) = \int_{-\infty}^\infty f(t) e^{-i\omega t}\, dt  \]
by \(-i\omega\).   Of course, the function \(f\) needs to be sufficiently well-behaved to justify
calculations involving its Laplace or Fourier transform.  At a more basic level, it also requires
some work to treat integration as the two-sided inverse of differentiation.  Engineers do this by
considering signals that vanish for \(t < 0\), and choosing the antiderivative that vanishes under
the same condition.  Luckily all these issues can be side-stepped in a formal treatment of
signal-flow diagrams: we can simply treat signals as living in an unspecified vector space over the
field \(\R(s)\).  The field \(\C(s)\) would work just as well, and control theory relies heavily on
complex analysis.  In most of this paper we work over an arbitrary field \(k\).

The simplest possible signal processor is a rock, which takes the `input' given by the force \(F\)
on the rock and produces as `output' the rock's position \(q\).  Thanks to Newton's second law
\(F=ma\), we can describe this using a signal-flow diagram:
  \begin{center}
   \scalebox{0.8}
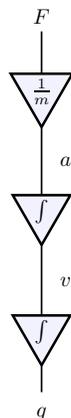
{
   \begin{tikzpicture}[thick]
   \node[coordinate] (q) [label={[shift={(0,-0.6)}]\(q\)}] {};
   \node [integral] (diff) [above of=q] {\(\int\)};
   \node (v) [above of=diff, label={[shift={(0.4,-0.5)}]\(v\)}] {};
   \node [integral] (dot) [above of=v] {\(\int\)};
   \node (a) [above of=dot, label={[shift={(0.4,-0.5)}]\(a\)}] {};
   \node [multiply] (m) [above of=a] {\(\frac{1}{m}\)};
   \node[coordinate] (F) [above of=m, label={[shift={(0,0)}]\(F\)}] {};

   \draw (F) -- (m) -- (dot) -- (diff) -- (q);
   \end{tikzpicture}
}
  \end{center}
Here composition of morphisms is drawn in the usual way, by attaching the output wire of one
morphism to the input wire of the next.

To build more interesting machines we need more building blocks, such as addition:
\[ + \maps (f,g) \mapsto f + g   \]
and duplication:
\[ \Delta \maps  f \mapsto (f,f)  \]
When these linear maps are written as matrices, their matrices are transposes of each other.  This
is reflected in the string diagrams for addition and duplication:
  
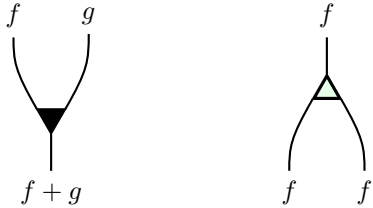
\begin{center}
   \begin{tikzpicture}[thick]
   \node[plus] (adder) {};
   \node (f) at (-0.5,1.35) {\(f\)};
   \node (g) at (0.5,1.35) {\(g\)};
   \node (out) [below of=adder] {\(f+g\)};

   \draw (f) .. controls +(-90:0.6) and +(120:0.6) .. (adder.left in);
   \draw (g) .. controls +(-90:0.6) and +(60:0.6) .. (adder.right in);
   \draw (adder) -- (out);
   \end{tikzpicture}
     \hspace{2cm}
   \begin{tikzpicture}[thick]
   \node[delta] (dupe){};
   \node (o1) at (-0.5,-1.35) {\(f\)};
   \node (o2) at (0.5,-1.35) {\(f\)};
   \node (in) [above of=dupe] {\(f\)};

   \draw (o1) .. controls +(90:0.6) and +(-120:0.6) .. (dupe.left out);
   \draw (o2) .. controls +(90:0.6) and +(-60:0.6) .. (dupe.right out);
   \draw (in) -- (dupe);
   \end{tikzpicture}
  \end{center}
The second is essentially an upside-down version of the first.  However, we draw addition as a dark
triangle and duplication as a light one because we will later want another way to `turn addition
upside-down' that does \emph{not} give duplication.  As an added bonus, a light upside-down triangle
resembles the Greek letter \(\Delta\), the usual symbol for duplication.  

While they are typically not considered worthy of mention in control theory, for completeness we
must include two other building blocks.  One is the zero map from \(\{0\}\) to our field \(k\),
which we denote as \(0\) and draw as follows:
  \begin{center}
  \begin{tikzpicture}[thick]
   \node (out1) {\(0\)};
   \node [zero] (ins1) at (0,1) {};

   \draw (out1) -- (ins1);
   \end{tikzpicture}
  \end{center}
The other is the zero map from \(k\) to \(\{0\}\), sometimes called `deletion', which we denote as
\(!\) and draw thus:
  \begin{center}
   \begin{tikzpicture}[thick]
   \node (in1) {\(f\)};
   \node [bang] (del1) at (0,-1) {};

   \draw (in1) -- (del1);
   \end{tikzpicture}
  \end{center}

Just as the matrices for addition and duplication are transposes of each other, so are the matrices
for zero and deletion, though they are rather degenerate, being \(1 \times 0\) and \(0 \times 1\)
matrices, respectively.  Addition and zero make \(k\) into a commutative monoid, meaning that the
following relations hold:
 \begin{center}
    \scalebox{0.80}
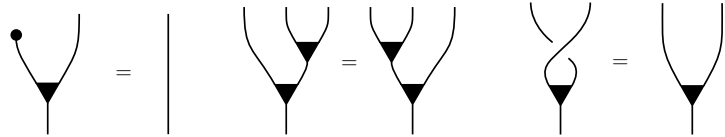
{
   \begin{tikzpicture}[-, thick, node distance=0.74cm]
   \node [plus] (summer) {};
   \node [coordinate] (sum) [below of=summer] {};
   \node [coordinate] (Lsum) [above left of=summer] {};
   \node [zero] (insert) [above of=Lsum, shift={(0,-0.35)}] {};
   \node [coordinate] (Rsum) [above right of=summer] {};
   \node [coordinate] (sumin) [above of=Rsum] {};
   \node (equal) [right of=Rsum, shift={(0,-0.26)}] {\(=\)};
   \node [coordinate] (in) [right of=equal, shift={(0,1)}] {};
   \node [coordinate] (out) [right of=equal, shift={(0,-1)}] {};

   \draw (insert) .. controls +(270:0.3) and +(120:0.3) .. (summer.left in)
         (summer.right in) .. controls +(60:0.6) and +(270:0.6) .. (sumin)
         (summer) -- (sum)    (in) -- (out);
   \end{tikzpicture}
        \hspace{1.0cm}
   \begin{tikzpicture}[-, thick, node distance=0.7cm]
   \node [plus] (uradder) {};
   \node [plus] (adder) [below of=uradder, shift={(-0.35,0)}] {};
   \node [coordinate] (urm) [above of=uradder, shift={(-0.35,0)}] {};
   \node [coordinate] (urr) [above of=uradder, shift={(0.35,0)}] {};
   \node [coordinate] (left) [left of=urm] {};

   \draw (adder.right in) .. controls +(60:0.2) and +(270:0.1) .. (uradder.io)
         (uradder.right in) .. controls +(60:0.35) and +(270:0.3) .. (urr)
         (uradder.left in) .. controls +(120:0.35) and +(270:0.3) .. (urm)
         (adder.left in) .. controls +(120:0.75) and +(270:0.75) .. (left)
         (adder.io) -- +(270:0.5);

   \node (eq) [right of=uradder, shift={(0,-0.25)}] {\(=\)};

   \node [plus] (ulsummer) [right of=eq, shift={(0,0.25)}] {};
   \node [plus] (summer) [below of=ulsummer, shift={(0.35,0)}] {};
   \node [coordinate] (ulm) [above of=ulsummer, shift={(0.35,0)}] {};
   \node [coordinate] (ull) [above of=ulsummer, shift={(-0.35,0)}] {};
   \node [coordinate] (right) [right of=ulm] {};

   \draw (summer.left in) .. controls +(120:0.2) and +(270:0.1) .. (ulsummer.io)
         (ulsummer.left in) .. controls +(120:0.35) and +(270:0.3) .. (ull)
         (ulsummer.right in) .. controls +(60:0.35) and +(270:0.3) .. (ulm)
         (summer.right in) .. controls +(60:0.75) and +(270:0.75) .. (right)
         (summer.io) -- +(270:0.5);
   \end{tikzpicture}
        \hspace{1.0cm}
   \begin{tikzpicture}[-, thick, node distance=0.7cm]
   \node [plus] (twadder) {};
   \node [coordinate] (twout) [below of=twadder] {};
   \node [coordinate] (twR) [above right of=twadder, shift={(-0.2,0)}] {};
   \node (cross) [above of=twadder] {};
   \node [coordinate] (twRIn) [above left of=cross, shift={(0,0.3)}] {};
   \node [coordinate] (twLIn) [above right of=cross, shift={(0,0.3)}] {};

   \draw (twadder.right in) .. controls +(60:0.35) and +(-45:0.25) .. (cross)
                            .. controls +(135:0.2) and +(270:0.4) .. (twRIn);
   \draw (twadder.left in) .. controls +(120:0.35) and +(-135:0.25) .. (cross.center)
                           .. controls +(45:0.2) and +(270:0.4) .. (twLIn);
   \draw (twout) -- (twadder);

   \node (eq) [right of=twR] {\(=\)};

   \node [coordinate] (L) [right of=eq] {};
   \node [plus] (adder) [below right of=L] {};
   \node [coordinate] (out) [below of=adder] {};
   \node [coordinate] (R) [above right of=adder] {};
   \node (cross) [above left of=R] {};
   \node [coordinate] (LIn) [above left of=cross] {};
   \node [coordinate] (RIn) [above right of=cross] {};

   \draw (adder.left in) .. controls +(120:0.7) and +(270:0.7) .. (LIn)
         (adder.right in) .. controls +(60:0.7) and +(270:0.7) .. (RIn)
         (out) -- (adder);
   \end{tikzpicture}
    }
\end{center}
The equation at right is the commutative law, and the crossing of strands is the `braiding'
\[       B \maps (f,g) \mapsto (g,f)  \]
by which we switch two signals.   In fact this braiding is a `symmetry', so it does not matter which
strand goes over which:
 \begin{center}
   \begin{tikzpicture}[thick, node distance=0.5cm]
   \node (fstart) {\(f\)};
   \node [coordinate] (ftop) [below of=fstart] {};
   \node (center) [below right of=ftop] {};
   \node [coordinate] (fout) [below right of=center] {};
   \node (fend) [below of=fout] {\(f\)};
   \node [coordinate] (gtop) [above right of=center] {};
   \node (gstart) [above of=gtop] {\(g\)};
   \node [coordinate] (gout) [below left of=center] {};
   \node (gend) [below of=gout] {\(g\)};

   \draw [rounded corners] (fstart) -- (ftop) -- (center) --
   (fout) -- (fend) (gstart) -- (gtop) -- (gout) -- (gend);
   \end{tikzpicture}
\hspace{1em} 
\raisebox{3em}{=} 
\hspace{1em}
   \begin{tikzpicture}[thick, node distance=0.5cm]
   \node (fstart) {\(f\)};
   \node [coordinate] (ftop) [below of=fstart] {};
   \node (center) [below right of=ftop] {};
   \node [coordinate] (fout) [below right of=center] {};
   \node (fend) [below of=fout] {\(f\)};
   \node [coordinate] (gtop) [above right of=center] {};
   \node (gstart) [above of=gtop] {\(g\)};
   \node [coordinate] (gout) [below left of=center] {};
   \node (gend) [below of=gout] {\(g\)};

   \draw [rounded corners] (fstart) -- (ftop) -- (fout) -- (fend)
   (gstart) -- (gtop) -- (center) -- (gout) -- (gend);
   \end{tikzpicture}
  \end{center}

Dually, duplication and deletion make \(k\) into a cocommutative comonoid.  This means that if we
reflect the equations obeyed by addition and zero across the horizontal axis and turn dark
operations into light ones, we obtain another set of valid equations:
\begin{center}
    \scalebox{0.80}{
   \begin{tikzpicture}[-, thick, node distance=0.74cm]
   \node [delta] (dupe) {};
   \node [coordinate] (top) [above of=dupe] {};
   \node [coordinate] (Ldub) [below left of=dupe] {};
   \node [bang] (delete) [below of=Ldub, shift={(0,0.35)}] {};
   \node [coordinate] (Rdub) [below right of=dupe] {};
   \node [coordinate] (dubout) [below of=Rdub] {};
   \node (equal) [right of=Rdub, shift={(0,0.26)}] {\(=\)};
   \node [coordinate] (in) [right of=equal, shift={(0,1)}] {};
   \node [coordinate] (out) [right of=equal, shift={(0,-1)}] {};

   \draw (delete) .. controls +(90:0.3) and +(240:0.3) .. (dupe.left out)
         (dupe.right out) .. controls +(300:0.6) and +(90:0.6) .. (dubout)
         (dupe) -- (top)    (in) -- (out);
   \end{tikzpicture}
       \hspace{1.0cm}
   \begin{tikzpicture}[-, thick, node distance=0.7cm]
   \node [delta] (lrduper) {};
   \node [delta] (duper) [above of=lrduper, shift={(-0.35,0)}] {};
   \node [coordinate](lrm) [below of=lrduper, shift={(-0.35,0)}] {};
   \node [coordinate](lrr) [below of=lrduper, shift={(0.35,0)}] {};
   \node [coordinate](left) [left of=lrm] {};

   \draw (duper.right out) .. controls +(300:0.2) and +(90:0.1) .. (lrduper.io)
         (lrduper.right out) .. controls +(300:0.35) and +(90:0.3) .. (lrr)
         (lrduper.left out) .. controls +(240:0.35) and +(90:0.3) .. (lrm)
         (duper.left out) .. controls +(240:0.75) and +(90:0.75) .. (left)
         (duper.io) -- +(90:0.5);

   \node (eq) [right of=lrduper, shift={(0,0.25)}] {\(=\)};

   \node [delta] (lldubber) [right of=eq, shift={(0,-0.25)}] {};
   \node [delta] (dubber) [above of=lldubber, shift={(0.35,0)}] {};
   \node [coordinate] (llm) [below of=lldubber, shift={(0.35,0)}] {};
   \node [coordinate] (lll) [below of=lldubber, shift={(-0.35,0)}] {};
   \node [coordinate] (right) [right of=llm] {};

   \draw (dubber.left out) .. controls +(240:0.2) and +(90:0.1) .. (lldubber.io)
         (lldubber.left out) .. controls +(240:0.35) and +(90:0.3) .. (lll)
         (lldubber.right out) .. controls +(300:0.35) and +(90:0.3) .. (llm)
         (dubber.right out) .. controls +(300:0.75) and +(90:0.75) .. (right)
         (dubber.io) -- +(90:0.5);
   \end{tikzpicture}
       \hspace{1.0cm}
   \begin{tikzpicture}[-, thick, node distance=0.7cm]
   \node [coordinate] (twtop) {};
   \node [delta] (twdupe) [below of=twtop] {};
   \node [coordinate] (twR) [below right of=twdupe, shift={(-0.2,0)}] {};
   \node (cross) [below of=twdupe] {};
   \node [coordinate] (twROut) [below left of=cross, shift={(0,-0.3)}] {};
   \node [coordinate] (twLOut) [below right of=cross, shift={(0,-0.3)}] {};

   \draw (twdupe.left out) .. controls +(240:0.35) and +(135:0.25) .. (cross)
                           .. controls +(-45:0.2) and +(90:0.4) .. (twLOut)
         (twdupe.right out) .. controls +(300:0.35) and +(45:0.25) .. (cross.center)
                            .. controls +(-135:0.2) and +(90:0.4) .. (twROut)
         (twtop) -- (twdupe);

   \node (eq) [right of=twR] {\(=\)};

   \node [coordinate] (L) [right of=eq] {};
   \node [delta] (dupe) [above right of=L] {};
   \node [coordinate] (top) [above of=dupe] {};
   \node [coordinate] (R) [below right of=dupe] {};
   \node (uncross) [below left of=R] {};
   \node [coordinate] (LOut) [below left of=uncross] {};
   \node [coordinate] (ROut) [below right of=uncross] {};

   \draw (dupe.left out) .. controls +(240:0.7) and +(90:0.7) .. (LOut)
         (dupe.right out) .. controls +(300:0.7) and +(90:0.7) .. (ROut)
         (top) -- (dupe);
   \end{tikzpicture}
    }
\end{center}
There are also relations between the monoid and comonoid operations.  For example, adding two
signals and then duplicating the result gives the same output as duplicating each signal and then
adding the results:
  \begin{center}
   \begin{tikzpicture}[thick]
   \node[plus] (adder) {};
   \node [coordinate] (f) [above of=adder, shift={(-0.4,-0.325)}, label={\(f\)}] {\(f\)};
   \node [coordinate] (g) [above of=adder, shift={(0.4,-0.325)}, label={\(g\)}] {\(g\)};
   \node[delta] (dupe) [below of=adder, shift={(0,0.25)}] {};
   \node [coordinate] (outL) [below of=dupe, shift={(-0.4,0.325)}, label={[shift={(-0.2,-0.6)}]\(f+g\)}] {};
   \node [coordinate] (outR) [below of=dupe, shift={(0.4,0.325)}, label={[shift={(0.15,-0.6)}]\(f+g\)}] {};

   \draw (adder.io) -- (dupe.io)
         (f) .. controls +(270:0.4) and +(120:0.25) .. (adder.left in)
         (adder.right in) .. controls +(60:0.25) and +(270:0.4) .. (g)
         (dupe.left out) .. controls +(240:0.25) and +(90:0.4) .. (outL)
         (dupe.right out) .. controls +(300:0.25) and +(90:0.4) .. (outR);
   \end{tikzpicture}
      \raisebox{4em}{=}
      \hspace{1em}
   \begin{tikzpicture}[-, thick, node distance=0.7cm]
   \node [plus] (addL) {};
   \node (cross) [above right of=addL, shift={(-0.1,-0.0435)}] {};
   \node [plus] (addR) [below right of=cross, shift={(-0.1,0.0435)}] {};
   \node [delta] (dupeL) [above left of=cross, shift={(0.1,-0.0435)}] {};
   \node [delta] (dupeR) [above right of=cross, shift={(-0.1,-0.0435)}] {};
   \node [coordinate] (f) [above of=dupeL, label={\(f\)}] {};
   \node [coordinate] (g) [above of=dupeR, label={\(g\)}] {};
   \node [coordinate] (sum1) [below of=addL, shift={(0,0.2)}, label={[shift={(-0.2,-0.6)}]\(f+g\)}] {};
   \node [coordinate] (sum2) [below of=addR, shift={(0,0.2)}, label={[shift={(0.15,-0.6)}]\(f+g\)}] {};

   \path
   (addL) edge (sum1) (addL.right in) edge (dupeR.left out) (addL.left in) edge [bend left=30] (dupeL.left out)
   (addR) edge (sum2) (addR.left in) edge (cross) (addR.right in) edge [bend right=30] (dupeR.right out)
   (dupeL) edge (f)
   (dupeL.right out) edge (cross)
   (dupeR) edge (g);
   \end{tikzpicture}
  \end{center} 
This diagram is familiar in the theory of Hopf algebras, or more generally bialgebras.  Here it is
an example of the fact that the monoid operations on \(k\) are comonoid homomorphisms---or
equivalently, the comonoid operations are monoid homomorphisms.  We summarize this situation by
saying that \(k\) is a \Define{bimonoid}.

So far all our string diagrams denote linear maps.  We can treat these as morphisms in the category
\( \Vectk \), where objects are finite-dimensional vector spaces over a field \(k\) and morphisms
are linear maps.  This category is equivalent to a skeleton where the only objects are vector spaces
\(k^n\) for \(n \ge 0\), and then morphisms can be seen as \(n \times m\) matrices.  The space of
signals is a vector space \(V\) over \(k\) which may not be finite-dimensional, but this does not
cause a problem: an \(n \times m\) matrix with entries in \(k\) still defines a linear map from
\(V^n\) to \(V^m\) in a functorial way.

In applications of string diagrams to quantum theory \cite{BS,CP}, we make \(\Vectk\) into a
symmetric monoidal category using the tensor product of vector spaces.  In control theory, we
instead make \(\Vectk\) into a symmetric monoidal category using the \emph{direct sum} of vector
spaces.  In Lemma~\ref{gensvk} we prove that for any field \(k\), \(\Vectk\) with direct sum
is generated as a symmetric monoidal category by the one object \(k\) together with these morphisms:
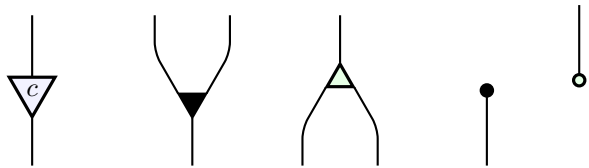
\begin{center}
\scalebox{1}{
 \begin{tikzpicture}[thick]
   \node[coordinate] (in) at (0,2) {};
   \node [multiply] (mult) at (0,1) {\(c\)};
   \node[coordinate] (out) at (0,0) {};
   \draw (in) -- (mult) -- (out);
   \end{tikzpicture}
\hspace{3 em}
\begin{tikzpicture}[thick]
   \node[plus] (adder) at (0,0.85) {};
   \node[coordinate] (f) at (-0.5,1.5) {}; 
   \node[coordinate] (g) at (0.5,1.5) {};
   \node[coordinate] (out) at (0,0) {};
   \node [coordinate] (pref) at (-0.5,2) {};
   \node [coordinate] (preg) at (0.5,2) {};
   \draw[rounded corners] (pref) -- (f) -- (adder.left in);
   \draw[rounded corners] (preg) -- (g) -- (adder.right in);
   \draw (adder) -- (out);
   \end{tikzpicture} 
\hspace{2em}
  \begin{tikzpicture}[thick]
   \node[delta] (dupe) at (0,1.15) {};
   \node[coordinate] (o1) at (-0.5,0.5) {};
   \node[coordinate] (o2) at (0.5,0.5) {};
   \node[coordinate] (in) at (0,2) {};
   \node [coordinate] (posto1) at (-0.5,0) {};
   \node [coordinate] (posto2) at (0.5,0) {};

   \draw[rounded corners] (posto1) -- (o1) -- (dupe.left out);
   \draw[rounded corners] (posto2) -- (o2) -- (dupe.right out);
   \draw (in) -- (dupe);
\end{tikzpicture}
\hspace{3em}
   \begin{tikzpicture}[thick]
   \node[hole] (heightHolder) at (0,2) {};
   \node [coordinate] (out) at (0,0) {};
   \node [zero] (del) at (0,1) {};
   \draw (del) -- (out);
   \end{tikzpicture}
\hspace{2em}
  \begin{tikzpicture}[thick]
   \node[coordinate] (in) at (0,2) {};
   \node [bang] (mult) at (0,1) {};
   \node [hole] (heightHolder) at (0,0) {};
   \draw (in) -- (mult);
   \end{tikzpicture}
}
\end{center}
where \(c \in k\) is arbitrary.  

However, these generating morphisms obey some unexpected relations!  For example, we have:
  
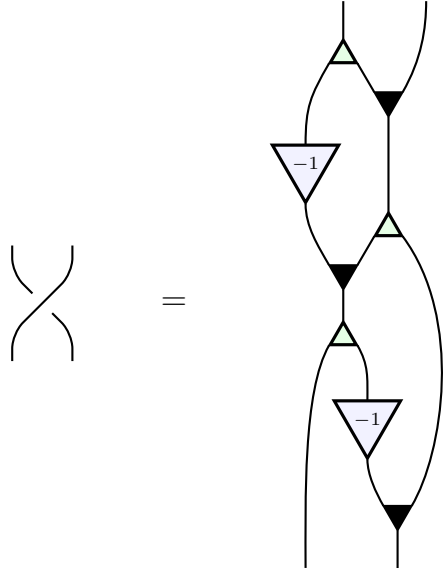
\begin{center}
   \begin{tikzpicture}[-, thick, node distance=0.85cm]
   \node (UpUpLeft) at (-0.4,-0.1) {};
   \node [coordinate] (UpLeft) at (-0.4,-0.6) {};
   \node (mid) at (0,-1) {};
   \node [coordinate] (DownRight) at (0.4,-1.4) {};
   \node (DownDownRight) at (0.4,-1.9) {};
   \node [coordinate] (UpRight) at (0.4,-0.6) {};
   \node (UpUpRight) at (0.4,-0.1) {};
   \node [coordinate] (DownLeft) at (-0.4,-1.4) {};
   \node (DownDownLeft) at (-0.4,-1.9) {};

   \draw [rounded corners=2mm] (UpUpLeft) -- (UpLeft) -- (mid) --
   (DownRight) -- (DownDownRight) (UpUpRight) -- (UpRight) -- (DownLeft) -- (DownDownLeft);

\begin{scope}[font=\fontsize{20}{20}\selectfont]
   \node (equals) at (1.75,-1) {\scalebox{0.65}{\(=\)}};
\end{scope}

   \node [coordinate] (sum2L) at (3.5,0) {};
   \node [multiply] (neg) [above of=sum2L] {\(\scriptstyle{-1}\)};
   \node [coordinate] (dupe1L) [above of=neg] {};
   \node [delta] (dupe1) [above right of=dupe1L, shift={(-0.1,0)}] {};
   \node (in1) [above of=dupe1] {};
   \node [plus] (sum1) [below right of=dupe1] {};
   \node [coordinate] (sum1R) [above right of=sum1, shift={(-0.1,0)}] {};
   \node (in2) [above of=sum1R] {};
   \node [delta] (dupe2) [below of=sum1, shift={(0,-0.85)}] {};
   \node [plus] (sum2) [below right of=sum2L, shift={(-0.1,0)}] {};
   \node [coordinate] (dupe2R) [below right of=dupe2, shift={(0.4,-0.6)}] {};
   \node [delta] (dupe3) [below of=sum2] {};
   \node [coordinate] (dupe3L) [below left of=dupe3, shift={(0.1,0)}] {};
   \node [multiply] (neg1) [below right of=dupe3, shift={(-0.28,-0.5)}] {\(\scriptstyle{-1}\)};
   \node [plus] (sum3) [below right of=neg1, shift={(-0.2,-0.65)}] {};
   \node [coordinate] (sum3R) [above right of=sum3, shift={(0,0.2)}] {};
   \node (out2) [below of=sum3] {};
   \node (out1) [below of=dupe3L, shift={(0,-1.75)}] {};

   \draw (in1) -- (dupe1);
   \draw (dupe1.left out) .. controls +(240:0.5) and +(90:0.5) .. (neg.90);
   \draw (dupe1.right out) -- (sum1.left in);
   \draw (in2) .. controls +(270:0.5) and +(60:0.5) .. (sum1.right in);
   \draw (sum1) -- (dupe2);
   \draw (dupe2.right out) .. controls (dupe2R) and (sum3R) .. (sum3.right in);
   \draw (dupe2.left out) -- (sum2.right in);
   \draw (neg.io) .. controls +(270:0.3) and +(120:0.3) .. (sum2.left in);
   \draw (sum2) -- (dupe3);
   \draw (dupe3.left out) .. controls +(240:0.7) and +(90:1) .. (out1);
   \draw (dupe3.right out) .. controls +(300:0.3) and +(90:0.3) .. (neg1.90);
   \draw (neg1.io) .. controls +(270:0.2) and +(120:0.2) .. (sum3.left in);
   \draw (sum3) -- (out2);
   \end{tikzpicture}
  \end{center}
Thus, it is important to find a complete set of relations obeyed by these generating
morphisms, thus obtaining a presentation of \(\Vectk\) as a symmetric monoidal 
category.  We do this in Theorem~\ref{presvk}.  In brief, these relations say:
\begin{enumerate}
\item \( (k, +, 0, \Delta, !) \) is a bicommutative bimonoid;
\item the rig operations of \(k\) can be recovered from the generating morphisms;
\item all the generating morphisms commute with scalar multiplication.
\end{enumerate}
Here item (2) means that \(+\), \(\cdot\), \(0\) and \(1\) in the field \(k\) can be expressed in 
terms of signal-flow diagrams as follows:
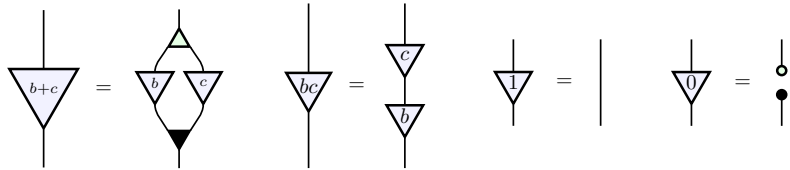
\begin{center}
    \scalebox{0.80}{
   \begin{tikzpicture}[-, thick, node distance=0.85cm]
   \node (bctop) {};
   \node [multiply] (bc) [below of=bctop, shift={(0,-0.59)}] {\(\scriptstyle{b+c}\)};
   \node (bcbottom) [below of=bc, shift={(0,-0.59)}] {};

   \draw (bctop) -- (bc) -- (bcbottom);

   \node (eq) [right of=bc, shift={(0.15,0)}] {\(=\)};

   \node [multiply] (b) [right of=eq, shift={(0,0.1)}] {\(\scriptstyle{b}\)};
   \node [delta] (dupe) [above right of=b, shift={(-0.2,0.1)}] {};
   \node (top) [above of=dupe, shift={(0,-0.2)}] {};
   \node [multiply] (c) [below right of=dupe, shift={(-0.2,-0.1)}] {\(\scriptstyle{c}\)};
   \node [plus] (adder) [below right of=b, shift={(-0.2,-0.3)}] {};
   \node (out) [below of=adder, shift={(0,0.2)}] {};

   \draw
   (dupe.left out) .. controls +(240:0.15) and +(90:0.15) .. (b.90)
   (dupe.right out) .. controls +(300:0.15) and +(90:0.15) .. (c.90)
   (top) -- (dupe.io)
   (adder.io) -- (out)
   (adder.left in) .. controls +(120:0.15) and +(270:0.15) .. (b.io)
   (adder.right in) .. controls +(60:0.15) and +(270:0.15) .. (c.io);
   \end{tikzpicture}
        \hspace{0.8cm}
   \begin{tikzpicture}[-, thick]
   \node (top) {};
   \node [multiply] (c) [below of=top] {\(c\)};
   \node [multiply] (b) [below of=c] {\(b\)};
   \node (bottom) [below of=b] {};

   \draw (top) -- (c) -- (b) -- (bottom);

   \node (eq) [left of=b, shift={(0.2,0.5)}] {\(=\)};

   \node (bctop) [left of=top, shift={(-0.6,0)}] {};
   \node [multiply] (bc) [left of=eq, shift={(0.2,0)}] {\(bc\)};
   \node (bcbottom) [left of=bottom, shift={(-0.6,0)}] {};

   \draw (bctop) -- (bc) -- (bcbottom);
   \end{tikzpicture}
        \hspace{0.8cm}
\raisebox{2em}{
   \begin{tikzpicture}[-, thick, node distance=0.85cm]
   \node (top) {};
   \node [multiply] (one) [below of=top] {1};
   \node (bottom) [below of=one] {};

   \draw (top) -- (one) -- (bottom);

   \node (eq) [right of=one] {\(=\)};
   \node (topid) [right of=top, shift={(0.6,0)}] {};
   \node (botid) [right of=bottom, shift={(0.6,0)}] {};

   \draw (topid) -- (botid);
   \end{tikzpicture}
}
        \hspace{0.6cm}
\raisebox{2em}{
   \begin{tikzpicture}[-, thick, node distance=0.85cm]
   \node [multiply] (prod) {\(0\)};
   \node (in0) [above of=prod] {};
   \node (out0) [below of=prod] {};
   \node (eq) [right of=prod] {\(=\)};
   \node [bang] (del) [right of=eq, shift={(-0.2,0.2)}] {};
   \node [zero] (ins) [right of=eq, shift={(-0.2,-0.2)}] {};
   \node (in1) [above of=del, shift={(0,-0.2)}] {};
   \node (out1) [below of=ins, shift={(0,0.2)}] {};

   \draw (in0) -- (prod) -- (out0);
   \draw (in1) -- (del);
   \draw (ins) -- (out1);
   \end{tikzpicture}
    }
}
\end{center}
Multiplicative inverses cannot be so expressed, so our signal-flow diagrams so far do not know that
\(k\) is a field.  Additive inverses also cannot be expressed in this way.  And indeed, a version of Theorem~\ref{presvk} holds whenever \(k\) is a commutative rig: that is, a commutative `ring without negatives', such as \(\N\).   See Section \ref{conclusions} for details.

While Theorem~\ref{presvk} is a step towards understanding the category-theoretic underpinnings of
control theory, it does not treat signal-flow diagrams that include `feedback'.  Feedback is one of
the most fundamental concepts in control theory because a control system without feedback may be
highly sensitive to disturbances or unmodeled behavior.  Feedback allows these uncontrolled
behaviors to be mollified.  As a string diagram, a basic feedback system might look schematically
like this:
  
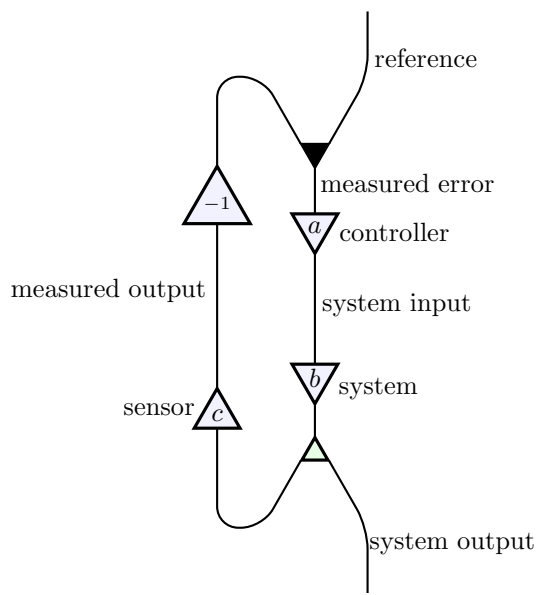
\begin{center}
   \begin{tikzpicture}[thick]
   \node (in) {};
   \node [coordinate] (inplus) [below of=in, label={[shift={(2.2em,0)}]reference}] {};
   \node [plus] (plus) [below of=inplus, shift={(-0.7,0)}]{};
   \node [multiply] (controller) [below of=plus, label={[shift={(3.0em,-0.5)}]controller},
   label={[shift={(3.5em,0.15)}]measured error}, label={[shift={(3.1em,-1.5)}]system input}] {\(a\)};
   \node [multiply] (system) [below of=controller, shift={(0,-1)},
   label={[shift={(2.4em,-0.6)}]system}] {\(b\)};
   \node [delta] (split) [below of=system, label={[shift={(5.2em,-1.7)}]system output}] {};
   \node [coordinate] (outsplit) [below of=split, shift={(0.7,0)}] {};
   \node (out) [below of=outsplit] {};
   \node [coordinate] (rcup) [below of=split, shift={(-0.7,0)}] {};
   \node [coordinate] (lcup) [left of=rcup, shift={(0.4,0)}] {};
   \node [upmultiply] (sensor) [above of=lcup, shift={(0,0.5)}, label={[shift={(-2.2em,-0.5)}]sensor},
   label={[shift={(-4.1em,1)}]measured output}] {\(c\)};
   \node [upmultiply] (minus) [above of=sensor, shift={(0,1.8)}] {\(\scriptstyle{-1}\)};
   \node [coordinate] (rcap) [above of=plus, shift={(-0.7,0)}] {};
   \node [coordinate] (lcap) [left of=rcap, shift={(0.4,0)}] {};

   \draw[rounded corners=8pt] (in) -- (inplus) -- (plus.right in) (plus) --
   (controller) -- (system) -- (split) (split.right out) --
   (outsplit) -- (out) (split.left out) -- (rcup) -- (lcup) --
   (sensor) -- (minus) -- (lcap) -- (rcap) -- (plus.left in);
   \end{tikzpicture}
  \end{center}
The user inputs a `reference' signal, which is fed into a controller, whose output is fed into a
system, or `plant', which in turn produces its own output.  But then the system's output is
duplicated, and one copy is fed into a sensor, whose output is added (or if we prefer, subtracted)
from the reference signal.

In string diagrams---unlike in the usual thinking on control theory---it is essential to be 
able to read any diagram from top to bottom as a composite of tensor products of generating 
morphisms.  Thus, to incorporate the idea of feedback, we need two more generating morphisms.  
These are the `cup':
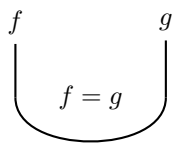
\begin{center}
   \begin{tikzpicture}[thick]
   \node (0) [label={[shift={(0,-1.6)}]}] {\(f=g\)};
   \node[coordinate] (3) [left of=0] {};
   \node[coordinate] (4) [right of=0] {};
   \node (1) [above of=3] {\(f\)};
   \node (2) [above of=4] {\(g\)};

   \path
   (1) edge (3)
   (2) edge (4)
   (3) edge [-, bend right=90] (4);
   \end{tikzpicture}
\end{center}
and `cap':
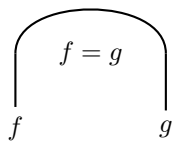
\begin{center}
   \begin{tikzpicture}[thick]
   \node (0) {\(f=g\)};
   \node[coordinate] (3) [left of=0] {};
   \node[coordinate] (4) [right of=0] {};
   \node (1) [below of=3] {\(f\)};
   \node (2) [below of=4] {\(g\)};

   \path
   (3) edge (1)
   (4) edge (2)
   (3) edge [bend left=90] (4);
   \end{tikzpicture}
\end{center}
These are not maps: they are relations.  The cup imposes the relation that its two inputs be equal,
while the cap does the same for its two outputs.  This is a way of describing how a signal flows
around a bend in a wire.

To make this precise, we use a category called \(\Relk\).  An object of this category is a
finite-dimensional vector space over \(k\), while a morphism from \(U\) to \(V\), denoted \(L \maps
U \asrelto V\), is a \Define{linear relation}, meaning a linear subspace
\[         L \subseteq U \oplus V  .\]
In particular, when \(k = \R(s)\), a linear relation \(L \maps k^m \to k^n\) is just an arbitrary 
system of constant-coefficient linear ordinary differential equations relating \(m\) input 
variables and \(n\) output variables.  

Since the direct sum \(U \oplus V\) is also the cartesian product of \(U\) and \(V\), a linear 
relation is indeed a relation in the usual sense, but with the property that if \(u \in U\) is 
related to \(v \in V\) and \(u' \in U\) is related to \(v' \in V\) then \(cu + c'u'\) is related to
\(cv + c'v'\) whenever \(c,c' \in k\).  We compose linear relations \(L \maps U \asrelto V\) and 
\(L' \maps V \asrelto W\) as follows:
\[         L'L = \{(u,w) \colon \; \exists\; v \in V \;\; (u,v) \in L \textrm{ and } 
(v,w) \in L'\} .\]
Any linear map \(f \maps U \to V\) gives a linear relation \(F \maps U \asrelto V\), namely the
graph of that map:
\[                  F = \{ (u,f(u)) : u \in U \}. \]
Composing linear maps thus becomes a special case of composing linear relations, so \(\Vectk\)
becomes a subcategory of \(\Relk\).  Furthermore, we can make \(\Relk\) into a monoidal category
using direct sums, and it becomes symmetric monoidal using the braiding already present in
\(\Vectk\).

In these terms, the \Define{cup} is the linear relation
\[                 \cup \maps k^2 \asrelto \{0\}   \]
given by
\[            \cup \; = \; \{ (x,x,0) : x \in k   \} \; \subseteq \; k^2 \oplus \{0\},   \]
while the \Define{cap} is the linear relation 
\[                 \cap \maps \{0\} \asrelto k^2  \]
given by
\[            \cap \; = \; \{ (0,x,x) : x \in k   \} \; \subseteq \; \{0\} \oplus k^2  .\]
These obey the \Define{zigzag relations}:
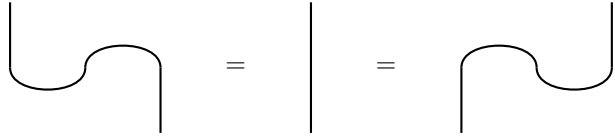
\begin{center}
   \begin{tikzpicture}[-, thick, node distance=1cm]
   \node (zigtop) {};
   \node [coordinate] (zigincup) [below of=zigtop] {};
   \node [coordinate] (zigcupcap) [right of=zigincup] {};
   \node [coordinate] (zigoutcap) [right of=zigcupcap] {};
   \node (zigbot) [below of=zigoutcap] {};
   \node (equal) [right of=zigoutcap] {\(=\)};
   \node (mid) [right of=equal] {};
   \node (vtop) [above of=mid] {};
   \node (vbot) [below of=mid] {};
   \node (equals) [right of=mid] {\(=\)};
   \node [coordinate] (zagoutcap) [right of=equals] {};
   \node (zagbot) [below of=zagoutcap] {};
   \node [coordinate] (zagcupcap) [right of=zagoutcap] {};
   \node [coordinate] (zagincup) [right of=zagcupcap] {};
   \node (zagtop) [above of=zagincup] {};
   \path
   (zigincup) edge (zigtop) edge [bend right=90] (zigcupcap)
   (zigoutcap) edge (zigbot) edge [bend right=90] (zigcupcap)
   (vtop) edge (vbot)
   (zagincup) edge (zagtop) edge [bend left=90] (zagcupcap)
   (zagoutcap) edge (zagbot) edge [bend left=90] (zagcupcap);
   \end{tikzpicture}
    \end{center}
Thus, they make \(\Relk\) into a compact closed category where \(k\), and thus every object, is its
own dual.  

Besides feedback, one of the things that make the cap and cup useful is that they allow any 
morphism \(L \maps U \asrelto V \) to be `plugged in backwards' and thus `turned around'.  For 
instance, turning around integration:
  
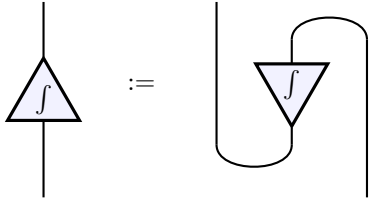
\begin{center}
   \begin{tikzpicture}[thick]
   \node [integral] (dot) {\(\int\)};
   \node [coordinate] (cupout) [below of=dot, shift={(0,0.2)}] {};
   \node [coordinate] (cupin) [left of=cupout] {};
   \node [coordinate] (capin) [above of=dot, shift={(0,-0.4)}] {};
   \node [coordinate] (in) [left of=capin, shift={(0,0.5)}] {};
   \node [coordinate] (capout) [right of=capin] {};
   \node [coordinate] (out) [right of=cupout, shift={(0,-0.7)}] {};

   \draw (capin) -- (dot) -- (cupout);
   \path
   (in) edge (cupin)
   (capout) edge (out)
   (cupin) edge [bend right=90] (cupout)
   (capin) edge [bend left=90] (capout);

   \node (eq) [left of=dot, shift={(-1,0)}] {\(:=\)};
   \node [upmultiply] (diff) [left of=eq, shift={(-0.3,-0.2)}] {\(\int\)};
   \node [coordinate] (diffin) [above of=diff, shift={(0,0.3)}] {};
   \node [coordinate] (diffout) [below of=diff, shift={(0,-0.3)}] {};

   \draw (diffin) -- (diff) -- (diffout);
   \end{tikzpicture}
  \end{center}
we obtain differentiation.  In general, using caps and cups we can turn around any linear relation
\(L \maps U \asrelto V\) and obtain a linear relation \(L^\dagger \maps V \asrelto U\), called the
\Define{adjoint} of \(L\), which turns out to given by
\[            L^\dagger = \{(v,u) : (u,v) \in L \}  .\]
For example, if \(c \in k\) is nonzero, the adjoint of scalar multiplication by \(c\) is
multiplication by \(c^{-1}\):
 
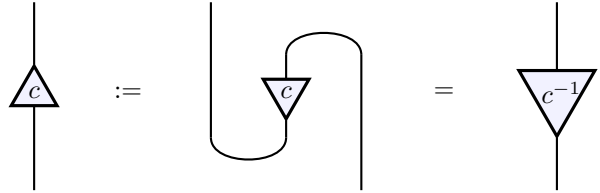
\begin{center}
   \begin{tikzpicture}[thick]
   \node [multiply] (c) {\(c\)};
   \node [coordinate] (cupout) [below of=c, shift={(0,0.4)}] {};
   \node [coordinate] (cupin) [left of=cupout] {};
   \node [coordinate] (capin) [above of=c, shift={(0,-0.5)}] {};
   \node [coordinate] (in) [left of=capin, shift={(0,0.7)}] {};
   \node [coordinate] (capout) [right of=capin] {};
   \node [coordinate] (out) [right of=cupout, shift={(0,-0.7)}] {};

   \draw (capin) -- (c) -- (cupout);
   \path
   (in) edge (cupin)
   (capout) edge (out)
   (cupin) edge [bend right=90] (cupout)
   (capin) edge [bend left=90] (capout);

   \node (eq) [right of=c, shift={(1.1,0)}] {\(=\)};

   \node [multiply] (mult) [right of=eq, shift={(0.5,0)}] {\(c^{-1}\!\!\)};
   \node [coordinate] (min) [above of=mult, shift={(0,0.2)}] {};
   \node [coordinate] (mout) [below of=mult, shift={(0,-0.3)}] {};

   \draw (min) -- (mult) -- (mout);

   \node (colon) [left of=c, shift={(-1.1,0)}] {\(:=\)};

   \node [upmultiply] (adj) [left of=colon, shift={(-0.25,0)}] {\(c\)};
   \node [coordinate] (adin) [above of=adj, shift={(0,0.2)}] {};
   \node [coordinate] (adout) [below of=adj, shift={(0,-0.3)}] {};

   \draw (adin) -- (adj) -- (adout);
   \end{tikzpicture}
  \end{center}
Thus, caps and cups allow us to express multiplicative inverses in terms of signal-flow diagrams!
One might think that a problem arises when when \(c = 0\), but no: the adjoint of scalar
multiplication by \(0\) is
\[          \{(0,x) : x \in k \} \subseteq k \oplus k .\]

In Lemma~\ref{gensrk} we show that \(\Relk\) is generated, as a symmetric monoidal category, by
these morphisms:
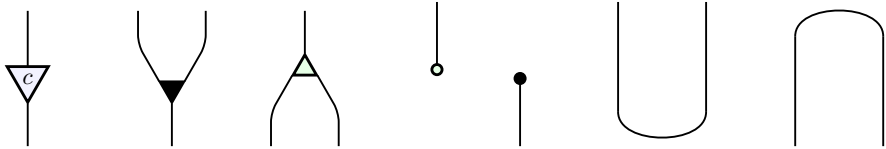
\begin{center}
\scalebox{0.9}{
 \begin{tikzpicture}[thick]
   \node [coordinate] (in) at (0,2) {};
   \node [multiply] (mult) at (0,1) {\(c\)};
   \node [coordinate] (out) at (0,0) {};
   \draw (in) -- (mult) -- (out);
   \end{tikzpicture}
\hspace{3 em}
\begin{tikzpicture}[thick]
   \node [plus] (adder) at (0,0.85) {};
   \node [coordinate] (f) at (-0.5,1.5) {}; 
   \node [coordinate] (g) at (0.5,1.5) {};
   \node [coordinate] (out) at (0,0) {};
   \node [coordinate] (pref) at (-0.5,2) {};
   \node [coordinate] (preg) at (0.5,2) {};

   \draw [rounded corners] (pref) -- (f) -- (adder.left in);
   \draw [rounded corners] (preg) -- (g) -- (adder.right in);
   \draw (adder) -- (out);
   \end{tikzpicture} 
\hspace{2em}
  \begin{tikzpicture}[thick]
   \node[delta] (dupe) at (0,1.15) {};
   \node[coordinate] (o1) at (-0.5,0.5) {};
   \node[coordinate] (o2) at (0.5,0.5) {};
   \node[coordinate] (in) at (0,2) {};
   \node [coordinate] (posto1) at (-0.5,0) {};
   \node [coordinate] (posto2) at (0.5,0) {};

   \draw[rounded corners] (posto1) -- (o1) -- (dupe.left out);
   \draw[rounded corners] (posto2) -- (o2) -- (dupe.right out);
   \draw (in) -- (dupe);
\end{tikzpicture}
\hspace{3em}
\begin{tikzpicture}[thick]
   \node[coordinate] (in) at (0,2) {};
   \node [bang] (mult) at (0,1) {};
   \node [hole] (heightHolder) at (0,0) {};
   \draw (in) -- (mult);
   \end{tikzpicture}
\hspace{2em}
  \begin{tikzpicture}[thick]
   \node[hole] (heightHolder) at (0,2) {};
   \node [coordinate] (out) at (0,0) {};
   \node [zero] (del) at (0,1) {};
   \draw (del) -- (out);
\end{tikzpicture}
\hspace{3em}
 \begin{tikzpicture}[thick]
   \node [coordinate] (3) at (0,0.375) {};
   \node [coordinate] (4) at (1.3,0.375) {};
   \node [coordinate] (1) at (0,2) {};
   \node [coordinate] (2) at (1.3,2) {};
   \path
   (1) edge (3)
   (2) edge (4)
   (3) edge [-, bend right=90] (4);
   \end{tikzpicture}
        \hspace{3em}
   \begin{tikzpicture}[thick]
   \node [coordinate] (3) at (0,1.625) {};
   \node [coordinate] (4) at (1.3,1.625) {};
   \node [coordinate] (1) at (0,0) {};
   \node [coordinate] (2) at (1.3,0) {};
   \path
   (3) edge (1)
   (4) edge (2)
   (3) edge [bend left=90] (4);
   \end{tikzpicture}
}
\end{center}
where \(c \in k\) is arbitrary.  

In Theorem~\ref{presrk} we find a complete set of relations obeyed by these generating morphisms,
thus giving a presentation of \(\Relk\) as a symmetric monoidal category.  To describe these
relations, it is useful to work with adjoints of the generating morphisms.  We have already seen
that the adjoint of scalar multiplication by \(c\) is scalar multiplication by \(c^{-1}\), except
when \(c = 0\).  Taking adjoints of the other four generating morphisms of \(\Vectk\), we obtain
four important but perhaps unfamiliar linear relations.  We draw these as `turned around' versions
of the original generating morphisms:

\begin{itemize}
\item \Define{Coaddition} is a linear relation from \(k\) to \(k^2\) that holds when the two 
outputs sum to the input:
\[           +^\dagger \maps k \asrelto k^2 \]
\[           +^\dagger = \{(x,y,z)  : \; x = y + z  \} \subseteq k \oplus k^2 \]
  
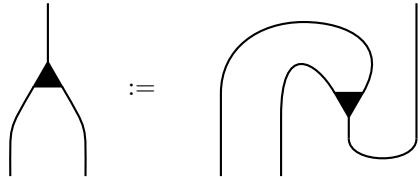
\begin{center}
   \begin{tikzpicture}[thick]
   \node [plus]       (adder)     at (0,0)       {};
   \node [coordinate] (sum)       at (0,-0.5)    {};
   \node [coordinate] (sumup)     at (0.9,-0.5)  {};
   \node [coordinate] (input)     at (0.9,1.3)   {};
   \node [coordinate] (highcap)   at (-0.6,1.05) {};
   \node [coordinate] (outerloop) at (-1.7,0.1)  {};
   \node [coordinate] (outerout)  at (-1.7,-1)   {};
   \node [coordinate] (innerloop) at (-0.9,-0.2) {};
   \node [coordinate] (innerout)  at (-0.9,-1)   {};

   \draw (innerloop) -- (innerout);
   \path (sum) edge [bend right=90] (sumup);
   \draw (adder.io) -- (sum)
   (sumup) -- (input)
   (outerloop) -- (outerout)
   (adder.left in) .. controls +(120:0.5) and +(0,1) .. (innerloop)
   (adder.right in) .. controls +(60:0.75) and +(0.6,0) .. (highcap)
   (highcap) .. controls +(-0.6,0) and +(0,0.6) .. (outerloop);

   \node              (eq)       at (-2.75,0.15) {\(:=\)};
   \node [coplus]     (coadder)  at (-4,0.3)     {};
   \node [coordinate] (topco)    at (-4,1.3)     {};
   \node [coordinate] (leftout)  at (-4.5,-1)    {};
   \node [coordinate] (rightout) at (-3.5,-1)    {};

   \draw
   (coadder.left out) .. controls +(240:0.7) and +(90:0.7) .. (leftout)
   (coadder.right out) .. controls +(300:0.7) and +(90:0.7) .. (rightout)
   (coadder.io) -- (topco);
   \end{tikzpicture}
  \end{center}
\item \Define{Cozero} is a linear relation from \(k\) to \(\{0\}\) that holds 
when the input is zero:
\[           0^\dagger \maps k \asrelto \{0\}   \]
\[           0^\dagger = \{ (0,0)\} \subseteq k \oplus \{0\}   \]
\begin{center}
\begin{tikzpicture}[thick]
   \node [zero] (Ze) at (0,0)    {};
   \node        (eq) at (-1,0)    {\(:=\)};
   \node [zero] (Ro) at (-1.8,-0.5) {};

   \draw[rounded corners=7pt]
   (Ze) -- (0,-0.5) -- (0.5,-0.5) -- (0.5,0.5);
   \draw (Ro) -- (-1.8,0.5);
\end{tikzpicture}
\end{center}
\item \Define{Coduplication} is a linear relation from \(k^2\) to \(k\) that holds when the two
inputs both equal the output:
\[           \Delta^\dagger \maps k^2 \asrelto k \]
\[           \Delta^\dagger = \{(x,y,z)  : \; x = y = z \} \subseteq k^2 \oplus k \]
  \begin{center}
   \begin{tikzpicture}[thick]
   \node [delta]      (copier)    at (0,0)        {};
   \node [coordinate] (original)  at (0,0.5)      {};
   \node [coordinate] (origdown)  at (0.9,0.5)    {};
   \node [coordinate] (output)    at (0.9,-1.3)   {};
   \node [coordinate] (lowcup)    at (-0.6,-1.05) {};
   \node [coordinate] (outerloop) at (-1.7,-0.1)  {};
   \node [coordinate] (outerin)   at (-1.7,1)     {};
   \node [coordinate] (innerloop) at (-0.9,0.2)   {};
   \node [coordinate] (innerout)  at (-0.9,1)     {};

   \draw (innerloop) -- (innerout);
   \path (original) edge [bend left=90] (origdown);
   \draw (copier.io) -- (original)
   (origdown) -- (output)
   (outerloop) -- (outerin)
   (copier.left out) .. controls +(240:0.5) and +(0,-1) .. (innerloop)
   (copier.right out) .. controls +(300:0.75) and +(0.6,0) .. (lowcup)
   (lowcup) .. controls +(-0.6,0) and +(0,-0.6) .. (outerloop);

   \node              (eq)       at (-2.75,-0.15) {\(:=\)};
   \node [codelta]    (pier)     at (-4,-0.3)     {};
   \node [coordinate] (bottomco) at (-4,-1.3)     {};
   \node [coordinate] (leftin)  at (-4.5,1)       {};
   \node [coordinate] (rightin) at (-3.5,1)       {};

   \draw
   (pier.left in) .. controls +(120:0.7) and +(270:0.7) .. (leftin)
   (pier.right in) .. controls +(60:0.7) and +(270:0.7) .. (rightin)
   (pier.io) -- (bottomco);
   \end{tikzpicture}
  \end{center}
\item \Define{Codeletion} is a linear relation from \(\{0\}\) to \(k\) that holds always:
\[          !^\dagger \maps \{0\} \asrelto k \]
\[          !^\dagger = \{(0,x) \} \subseteq \{0\} \oplus k \]
\begin{center}
\begin{tikzpicture}[thick]
   \node [bang] (Ba) at (0,0)    {};
   \node        (eq) at (-1,0)    {\(:=\)};
   \node [bang] (ng) at (-1.8,0.5) {};

   \draw[rounded corners=7pt]
   (Ba) -- (0,0.5) -- (0.5,0.5) -- (0.5,-0.5);
   \draw (ng) -- (-1.8,-0.5);
\end{tikzpicture}
\end{center}
\end{itemize}
Since \(+^\dagger,0^\dagger,\Delta^\dagger\) and \(!^\dagger\) automatically obey turned-around 
versions of the relations obeyed  by \(+,0,\Delta\) and \(!\), we see that \(k\) acquires a 
\emph{second} bicommutative bimonoid structure when considered as an object in \(\Relk\).  

Moreover, the four dark operations make \(k\) into a \Define{Frobenius monoid}.  This means that
\((k,+,0)\) is a monoid, \((k,+^\dagger, 0^\dagger)\) is a comonoid, and the \Define{Frobenius 
relation} holds:
\begin{center}
 \scalebox{1}{
   \begin{tikzpicture}[thick]
   \node [plus] (sum1) at (0.5,-0.216) {};
   \node [coplus] (cosum1) at (1,0.216) {};
   \node [coordinate] (sum1corner) at (0,0.434) {};
   \node [coordinate] (cosum1corner) at (1.5,-0.434) {};
   \node [coordinate] (sum1out) at (0.5,-0.975) {};
   \node [coordinate] (cosum1in) at (1,0.975) {};
   \node [coordinate] (1cornerin) at (0,0.975) {};
   \node [coordinate] (1cornerout) at (1.5,-0.975) {};

   \draw[rounded corners] (1cornerin) -- (sum1corner) -- (sum1.left in)
   (1cornerout) -- (cosum1corner) -- (cosum1.right out);
   \draw (sum1.right in) -- (cosum1.left out)
   (sum1.io) -- (sum1out)
   (cosum1.io) -- (cosum1in);

   \node (eq1) at (2,0) {\(=\)};
   \node [plus] (sum2) at (3,0.325) {};
   \node [coplus] (cosum2) at (3,-0.325) {};
   \node [coordinate] (sum2inleft) at (2.5,0.975) {};
   \node [coordinate] (sum2inright) at (3.5,0.975) {};
   \node [coordinate] (cosum2outleft) at (2.5,-0.975) {};
   \node [coordinate] (cosum2outright) at (3.5,-0.975) {};

   \draw (sum2inleft) .. controls +(270:0.3) and +(120:0.15) .. (sum2.left in)
   (sum2inright) .. controls +(270:0.3) and +(60:0.15) .. (sum2.right in)
   (cosum2outleft) .. controls +(90:0.3) and +(240:0.15) .. (cosum2.left out)
   (cosum2outright) .. controls +(90:0.3) and +(300:0.15) .. (cosum2.right out)
   (sum2.io) -- (cosum2.io);

   \node (eq2) at (4,0) {\(=\)};
   \node [plus] (sum3) at (5.5,-0.216) {};
   \node [coplus] (cosum3) at (5,0.216) {};
   \node [coordinate] (sum3corner) at (6,0.434) {};
   \node [coordinate] (cosum3corner) at (4.5,-0.434) {};
   \node [coordinate] (sum3out) at (5.5,-0.975) {};
   \node [coordinate] (cosum3in) at (5,0.975) {};
   \node [coordinate] (3cornerin) at (6,0.975) {};
   \node [coordinate] (3cornerout) at (4.5,-0.975) {};

   \draw[rounded corners] (3cornerin) -- (sum3corner) -- (sum3.right in)
   (3cornerout) -- (cosum3corner) -- (cosum3.left out);
   \draw (sum3.left in) -- (cosum3.right out)
   (sum3.io) -- (sum3out)
   (cosum3.io) -- (cosum3in);
   \end{tikzpicture}

}
\end{center}
All three expressions in this equation are linear relations saying that the sum of the two inputs
equal the sum of the two outputs.  

The operation sending each linear relation to its adjoint extends to a contravariant functor 
\[ \dagger \maps \Relk\ \to \Relk ,\]
which obeys a list of properties that are summarized by saying that \(\Relk\) is a
`\(\dagger\)-compact' category \cite{AC,Selinger}.  Because two of the operations in the Frobenius
monoid \((k, +,0,+^\dagger,0^\dagger)\) are adjoints of the other two, it is a
\Define{\(\dagger\)-Frobenius monoid}.  This Frobenius monoid is also \Define{special}, meaning that
comultiplication (in this case \(+^\dagger\)) followed by multiplication (in this case \(+\)) equals
the identity:
\begin{center}
   \begin{tikzpicture}[thick]
   \node [plus] (sum) at (0.4,-0.5) {};
   \node [coplus] (cosum) at (0.4,0.5) {};
   \node [coordinate] (in) at (0.4,1) {};
   \node [coordinate] (out) at (0.4,-1) {};
   \node (eq) at (1.3,0) {\(=\)};
   \node [coordinate] (top) at (2,1) {};
   \node [coordinate] (bottom) at (2,-1) {};

   \path (sum.left in) edge[bend left=30] (cosum.left out)
   (sum.right in) edge[bend right=30] (cosum.right out);
   \draw (top) -- (bottom)
   (sum.io) -- (out)
   (cosum.io) -- (in);
   \end{tikzpicture}
  \end{center}
This Frobenius monoid is also commutative---and cocommutative, but for Frobenius monoids this
follows from commutativity.

Starting around 2008, commutative special \(\dagger\)-Frobenius monoids have become important in the
categorical foundations of quantum theory, where they can be understood as `classical structures'
for quantum systems \cite{CPV,Vicary}.  The category \(\Fin\Hilb\) of finite-dimensional Hilbert
spaces and linear maps is a \(\dagger\)-compact category, where any linear map \(f \maps H \to K\)
has an adjoint \(f^\dagger \maps K \to H\) given by
\[         \langle f^\dagger \phi, \psi \rangle = \langle \phi, f \psi \rangle \]
for all \(\psi \in H, \phi \in K \).  A commutative special \(\dagger\)-Frobenius monoid in
\(\Fin\Hilb\) is then the same as a Hilbert space with a chosen orthonormal basis.  The reason is
that given an orthonormal basis \( \psi_i \) for a finite-dimensional Hilbert space \(H\), we can
make \(H\) into a  commutative special \(\dagger\)-Frobenius monoid with multiplication \(m \maps H
\otimes H \to H\) given by
\[     m (\psi_i \otimes \psi_j ) = \left\{ \begin{array}{cl}  \psi_i & i = j \\
                                                                                0 & i \ne j  
\end{array}\right.  \]
and unit \(i \maps \C \to H\) given by
\[   i(1) = \sum_i \psi_i . \]
The comultiplication \(m^\dagger\) duplicates basis states:
\[        m^\dagger(\psi_i) = \psi_i \otimes \psi_i  . \]
Conversely, any commutative special \(\dagger\)-Frobenius monoid in \(\Fin\Hilb\) arises this way.  

Considerably earlier, around 1995, commutative Frobenius monoids were recognized as important in
topological quantum field theory.  The reason, ultimately, is that the free symmetric monoidal
category on a commutative Frobenius monoid is \(2\Cob\), the category with 2-dimensional oriented
cobordisms as morphisms: see Kock's textbook \cite{Kock} and the many references therein.  But the
free symmetric monoidal category on a commutative \emph{special} Frobenius monoid was worked out 
even earlier \cite{CW,Kock2,RSW}: it is the category with finite sets as objects, where a morphism 
\(f \maps X \to Y\) is an isomorphism class of cospans
\[        X \longrightarrow S \longleftarrow Y  .\]
This category can be made into a \(\dagger\)-compact category in an obvious way, and then the
1-element set becomes a commutative special \(\dagger\)-Frobenius monoid.  

For all these reasons, it is interesting to find a commutative special \(\dagger\)-Frobenius monoid
lurking at the heart of control theory!  However, the Frobenius monoid here has yet another
property, which is more unusual.  Namely, the unit \(0 \maps \{0\} \asrelto k\) followed by the
counit \(0^\dagger \maps k \asrelto \{0\} \) is the identity:
\begin{center}
\scalebox{1}{
\begin{tikzpicture}[-, thick, node distance=0.7cm]
   \node [zero] (Bins) {};
   \node [zero] (Tins) [above of=Bins] {};
   \path
   (Tins) edge (Bins);
   \end{tikzpicture}
\quad
\raisebox{1em}{=}
}
\end{center}
We call a special Frobenius monoid that also obeys this extra law \Define{extra-special}.  One can
check that the free symmetric monoidal category on a commutative extra-special Frobenius monoid is
the category with finite sets as objects, where a morphism \(f \maps X \to Y\) is an equivalence
relation on the disjoint union \(X \sqcup Y\), and we compose \(f \maps X \to Y\) and \(g \maps Y
\to Z\) by letting \(f\) and \(g\) generate an equivalence relation on \(X \sqcup Y \sqcup Z\) and
then restricting this to \(X \sqcup Z\).

As if this were not enough, the light operations share many properties with the dark ones.  In
particular, these operations make \(k\) into a commutative extra-special \(\dagger\)-Frobenius
monoid in a second way.  In summary:
\begin{itemize}
\item \((k, +, 0, \Delta, !)\) is a bicommutative bimonoid;
\item \((k, \Delta^\dagger, !^\dagger, +^\dagger, 0^\dagger)\) is a bicommutative bimonoid;
\item \((k, +, 0, +^\dagger, 0^\dagger)\) is a commutative extra-special 
\(\dagger\)-Frobenius monoid;
\item \((k, \Delta^\dagger, !^\dagger, \Delta, !)\) is a commutative extra-special 
\(\dagger\)-Frobenius monoid.
\end{itemize}

It should be no surprise that with all these structures built in, signal-flow diagrams are a
powerful method of designing processes.  However, it is surprising that most of these structures 
are present in a seemingly very different context: the so-called `ZX calculus', a diagrammatic 
formalism for working with complementary observables in quantum theory \cite{CD}.  This arises 
naturally when one has an \(n\)-dimensional Hilbert space \(H\) with two orthonormal bases 
\(\psi_i, \phi_i \) that are `mutually unbiased', meaning that
\[           |\langle \psi_i, \phi_j\rangle|^2 = \displaystyle{\frac{1}{n}}  \]
for all \(1 \le i, j \le n\).  Each orthonormal basis makes \(H\) into commutative special
\(\dagger\)-Frobenius monoid in \(\Fin\Hilb\).  Moreover, the multiplication and unit of either one
of these Frobenius monoids fits together with the comultiplication and counit of the other to form 
a bicommutative bimonoid.  So, we have all the structure present in the list above---except
that these Frobenius monoids are only extra-special if \(H\) is 1-dimensional.  

The field \(k\) is also a 1-dimensional vector space, but this is a red herring: in \(\Relk\)
\emph{every} finite-dimensional vector space naturally acquires all four structures listed above,
since addition, zero, duplication and deletion are well-defined and obey all the relations we have
discussed.  We focus on \(k\) in this paper simply because it generates all the objects \(\Relk\)
via direct sum.

Finally, in \(\Relk\) the cap and cup are related to the light and dark operations as follows:
\begin{center}
\scalebox{1}{
   \begin{tikzpicture}[thick]
   \node (eq) at (0.2,-0.1) {\(=\)};
   \node [coordinate] (lcap) at (-1.5,0.5) {};
   \node [coordinate] (rcap) at (-0.5,0.5) {};
   \node [coordinate] (lcapbot) at (-1.5,-1) {};
   \node [coordinate] (rcapbot) at (-0.5,-1) {};
   \node [delta] (dub) at (1.25,0) {};
   \node [bang] (boom) at (1.25,0.65) {};
   \node [coordinate] (Leftout) at (0.75,-1) {};
   \node [coordinate] (Rightout) at (1.75,-1) {};

   \draw (dub.left out) .. controls +(240:0.5) and +(90:0.5) .. (Leftout)
      (dub.right out) .. controls +(300:0.5) and +(90:0.5) .. (Rightout);
   \draw (boom) -- (dub) (lcapbot) -- (lcap) (rcap) -- (rcapbot);
   \path (lcap) edge[bend left=90] (rcap);
   \end{tikzpicture}
\qquad \qquad
   \begin{tikzpicture}[thick]
   \node [multiply] (neg) at (0,0.1) {\(\scriptstyle{-1}\)};
   \node [coordinate] (cupInLeft) at (0,1) {};
   \node [coordinate] (Lcup) at (0,-0.5) {};
   \node [coordinate] (Rcup) at (1,-0.5) {};
   \node [coordinate] (cupInRight) at (1,1) {};
   \node (eq) at (1.7,0.1) {\(=\)};
   \node [coordinate] (SumLeftIn) at (2.25,1) {};
   \node [coordinate] (SumRightIn) at (3.25,1) {};
   \node [plus] (Sum) at (2.75,0) {};
   \node [zero] (coZero) at (2.75,-0.65) {};

   \draw (SumRightIn) .. controls +(270:0.5) and +(60:0.5) .. (Sum.right in)
      (SumLeftIn) .. controls +(270:0.5) and +(120:0.5) .. (Sum.left in);
   \draw (cupInLeft) -- (neg) -- (Lcup)
      (Rcup) -- (cupInRight)
      (Sum) -- (coZero);
   \path (Lcup) edge[bend right=90] (Rcup);
   \end{tikzpicture}
}
\end{center}
Note the curious factor of \(-1\) in the second equation, which breaks some of the symmetry we have
seen so far.  This equation says that two elements \(x, y \in k\) sum to zero if and only if \(-x =
y\).  Using the zigzag relations, the two equations above give
\begin{center}
 \begin{tikzpicture}[thick]
   \node (eq) {\(=\)};
   \node[delta] (Lup) at (-1,0.216) {};
   \node[plus] (Ldn) at (-1.5,-0.216) {};
   \node[coordinate] (Lupo) at (-0.5,-0.434) {};
   \node[coordinate] (Ldni) at (-2,0.434) {};
   \node[bang] (Lupi) at (-1,0.866) {};
   \node[zero] (Ldno) at (-1.5,-0.866) {};
   \node (Lo) at (-0.5,-1.082) {};
   \node (Li) at (-2,1.082) {};
   \node [multiply] (neg1) [right of=eq] {\(\scriptstyle{-1}\)};
   \node[coordinate] (inR) [above of=neg1] {};
   \node[coordinate] (outR) [below of=neg1] {};

   \draw[rounded corners] (Li) -- (Ldni) -- (Ldn.left in) (Lo) -- (Lupo) -- (Lup.right out);
   \draw (Ldn) -- (Ldno) (Lup) -- (Lupi) (Ldn.right in) -- (Lup.left out);
   \path (neg1) edge (inR) edge (outR);
   \end{tikzpicture}
\end{center}
We thus see that in \(\Relk\), both additive and multiplicative inverses can be expressed in terms
of the generating morphisms used in signal-flow diagrams.

Theorem~\ref{presrk} gives a presentation of \(\Relk\) based on the ideas just discussed.  Briefly,
it says that \(\Relk\) is equivalent to the symmetric monoidal category generated by an object \(k\)
and these morphisms:
\begin{enumerate}
\item addition \(+\maps k^2 \asrelto k\)
\item zero \(0 \maps \{0\} \asrelto k \)
\item duplication \(\Delta\maps k\asrelto k^2 \)
\item deletion \(! \maps k \asrelto 0\)
\item scalar multiplication \(c\maps k\asrelto k\) for any \(c\in k\)
\item cup \(\cup \maps k^2 \asrelto \{0\} \)
\item cap \(\cap \maps \{0\} \asrelto k^2 \)
\end{enumerate}
obeying these relations:
\begin{enumerate}
\item \((k, +, 0, \Delta, !)\) is a bicommutative bimonoid;
\item \(\cap\) and \(\cup\) obey the zigzag equations;
\item \((k, +, 0, +^\dagger, 0^\dagger)\) is a commutative extra-special 
\(\dagger\)-Frobenius monoid;
\item \((k, \Delta^\dagger, !^\dagger, \Delta, !)\) is a commutative extra-special 
\(\dagger\)-Frobenius monoid;
\item the field operations of \(k\) can be recovered from the generating morphisms; 
\item the generating morphisms (1)-(4) commute with scalar multiplication.
\end{enumerate}
Note that item (2) makes \(\Relk\) into a \(\dagger\)-compact category, allowing us to mention the
adjoints of generating morphisms in the subsequent relations.  Item (5) means that \(+,\cdot, 0,
1\) and also additive and multiplicative inverses in the field \(k\) can be expressed in terms of 
signal-flow diagrams in the manner we have explained.

\section{A presentation of \(\Vectk\)}
\label{finvect}

Our goal in this section is to find a presentation for the symmetric monoidal category \(\Vectk\).
To simplify some technicalities, we shall use Mac Lane's coherence theorem \cite{MacLane} to choose
a symmetric monoidal equivalence \(F \maps \Vectk' \to \Vectk\) where \(\Vectk'\) is strict.  This
allows us to avoid mentioning associators and unitors, since in \(\Vectk'\) these are identity
morphisms.  In what follows, we call \(\Vectk'\) simply \(\Vectk\), and call objects and morphisms
in \(\Vectk\) by the names of their images under \(F\).  Colloquially speaking, we `work in a strict
version' of \(\Vectk\), and do not bother to indicate that this is a different (though equivalent)
symmetric monoidal category.

We say a strict symmetric monoidal category \(C\) is \Define{generated} by a set \(O\) of objects
and a set \(M\) of morphisms going between tensor products of objects in \(O\) if the smallest
subcategory \(C_0\) of \(C\) containing:
\begin{itemize}
\item the objects in \(O\),
\item the morphisms in \(M\),
\item the tensor products of any objects or morphisms in \(C_0\)
\item the braiding for any pair of objects in \(C_0\)
\end{itemize}
has the property that the inclusion \(i \maps C_0 \to C\) is an equivalence of categories.  It
follows that \(i\) extends to an equivalence of symmetric monoidal categories.  In this situation we
call the elements of \(O\) \Define{generating objects} for \(C\), and call the elements of \(M\)
\Define{generating morphisms}.  

\begin{lemma} \label{gensvk} 
For any field \(k\), the object \(k\) together with the morphisms:
\begin{enumerate}
\item scalar multiplication \(c \maps k \to k\) for any \(c \in k\)
\item addition \(+ \maps k \oplus k \to k\)
\item zero \(0 \maps \{0\} \to k\)
\item duplication \(\Delta \maps k \to k \oplus k\)
\item deletion \(! \maps k \to \{0\}\)
\end{enumerate}
generate \(\Vectk\), the category of finite-dimensional vector spaces over \(k\) and linear maps, as
a symmetric monoidal category.
\end{lemma}

\begin{proof}
It suffices to show that \(k\) together with the morphisms in (1)--(5) generate the full subcategory
of \(\Vectk\) containing only the iterated direct sums \(k^n = k \oplus \cdots \oplus k\), since
this is equivalent to \(\Vectk\).  

A linear map in \(\Vectk\), \(T \maps k^m \to k^n\) can be expressed as \(n\) \(k\)-linear
combinations of \(m\) elements of \(k\).  That is, \(T(k_1, \ldots, k_m) = (\sum_j{a_{1j} k_j},
\ldots, \sum_j{a_{nj} k_j})\), \(a_{ij} \in k\).  Any \(k\)-linear combination of \(r\) elements can
be constructed with only addition, multiplication, and zero, with zero only necessary when
providing the unique \(k\)-linear combination for \(r=0\).  When \(r=1\), \(a_1 (k_1)\) is an
arbitrary \(k\)-linear combination.  For \(r>1\), \(+ (S_{r-1}, a_r (k_r))\) yields an arbitrary
\(k\)-linear combination on \(r\) elements, where \(S_{r-1}\) is an arbitrary \(k\)-linear
combination of \(r-1\) elements.  The inclusion of duplication allows process of forming
\(k\)-linear combinations to be repeated an arbitrary (finite) positive number of times, and
deletion allows the process to be repeated zero times.  When \(n\) \(k\)-linear combinations are
needed, each input may be duplicated \(n-1\) times.  Because \(\Vectk\) is being generated as a
symmetric monoidal category, the \(mn\) outputs can then be permuted into \(n\) collections of \(m\)
outputs: one output from each input for each collection.  Each collection can then form a
\(k\)-linear combination, as above.  The following diagrams illustrate the pieces that form this
inductive argument.
  
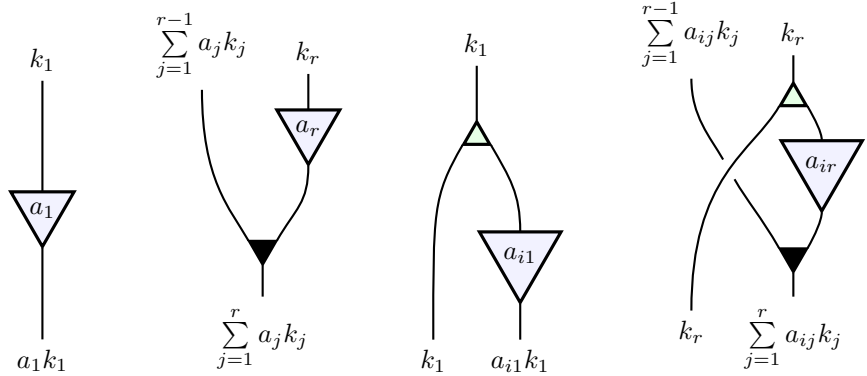
\begin{center}
   \begin{tikzpicture}[thick]
   \node (top) at (0,4) {\(k_1\)};
   \node [multiply] (times) at (0,2) {\(a_1\)};
   \node (bottom) at (0,0) {\(a_1 k_1\)};

   \draw (top) -- (times) -- (bottom);
   \end{tikzpicture}
       \hspace{0.7cm}
   \begin{tikzpicture}[thick]
   \node (sum) {\(\sum\limits_{j=1}^{r-1} a_j k_j\)};
   \node [coordinate] (UL) [below of=sum, shift={(0.1,-0.9)}] {};
   \node [plus] (adder) [below right of=UL, shift={(0,-0.1)}] {};
   \node [coordinate] (UR) [above right of=adder] {};
   \node [multiply] (mult) [above of=UR, shift={(-0.1,-0.1)}] {\(a_r\)};
   \node (next) [above of=mult] {\(k_r\)};
   \node (combo) [below of=adder, shift={(0,-0.2)}] {\(\sum\limits_{j=1}^{r} a_j k_j\)};

   \draw (next) -- (mult) (mult.io) .. controls +(270:0.5) and +(60:0.5) .. (adder.right in);
   \draw (sum.270) .. controls +(270:1.2) and +(120:0.5) .. (adder.left in);
   \draw (adder) -- (combo);
   \end{tikzpicture}
        \hspace{0.7cm}
   \begin{tikzpicture}[thick, node distance=1.1cm]
   \node (top) {\(k_1\)};
   \node [delta] (dupe) [below of=top, shift={(0,-0.125)}] {};
   \node [coordinate] (L) [below left of=dupe] {};
   \node [coordinate] (R) [below right of=dupe] {};
   \node [multiply] (mult) [below of=R, shift={(-0.2,0.3)}] {\(a_{i1}\)};
   \node (prod) [below of=mult, shift={(0,-0.3)}] {\(a_{i1} k_1\)};
   \node (id) [below of=L, shift={(0.2,-1.1)}] {\(k_1\)};

   \draw (dupe.right out) .. controls +(300:0.5) and +(90:0.5) .. (mult.90) (mult.io) -- (prod);
   \draw (dupe.left out) .. controls +(240:0.8) and +(90:2) .. (id);
   \draw (top) -- (dupe);
   \end{tikzpicture}
        \hspace{0.7cm}
   \begin{tikzpicture}[thick, node distance=1.1cm]
   \node (core) {\(\sum\limits_{j=1}^{r-1} a_{ij} k_j\)};
   \node [coordinate] (subcore) [below of=core, shift={(-0.2,0)}] {};
   \node (cross) [below right of=subcore, shift={(-0.08,0.08)}] {};
   \node [coordinate] (out) [below left of=cross, shift={(0.08,-0.08)}] {};
   \node (id) [below of=out, shift={(0.2,-0.2)}] {\(k_r\)};
   \node [plus] (adder) [below right of=cross, shift={(0.08,-0.38)}] {};
   \node (sum) [below of=adder] {\(\sum\limits_{j=1}^{r} a_{ij} k_j\)};
   \node [multiply] (mult) [above right of=adder, shift={(-0.4,0.45)}] {\(a_{ir}\)};
   \node [delta] (dupe) [above left of=mult, shift={(0.4,0.1)}] {};
   \node (in) [above of=dupe, shift={(0,-0.25)}] {\(k_r\)};

   \draw (core.270) .. controls +(270:0.5) and +(120:0.5) .. (cross)
   (cross) .. controls +(300:0.5) and +(120:0.5) .. (adder.left in)
   (dupe.left out) .. controls +(240:0.7) and +(90:1.8) .. (id.90);
   \draw (dupe.right out) .. controls +(300:0.2) and +(90:0.1) .. (mult.90)
   (mult.io) .. controls +(270:0.1) and +(60:0.2) .. (adder.right in)
   (adder) -- (sum) (in) -- (dupe);
   \end{tikzpicture}
  \end{center}
Since multiplication provides the map \(k_1 \mapsto a_1 k_1\), as in the far left diagram, the
middle-left diagram can be used inductively to form a \(k\)-linear combination of any number of
inputs.  In particular, we have any linear map \(S_r \maps k^m \to k\) given by \((k_1, \ldots, k_m)
\mapsto (\sum_j a_{rj} k_j)\).  Using duplication as in the middle-right diagram, one can produce
the map \(k_1 \mapsto (k_1, a_{i1} k_1)\), to which the right diagram can be inductively applied.
Thus we can build any linear map, \(T_j \in \Vectk\), \(T_j \maps k^m \to k^{m+1}\) given by \((k_1,
\ldots, k_m) \mapsto (k_1, \ldots, k_m, \sum_j a_{ij} k_j)\).  If we represent the identity map on
\(k^r\) as \(1^r\), the \(r\)-fold tensor product of the identity map on \(k\), any linear map \(T
\maps k^m \to k^n\) can be given by \((k_1, \ldots, k_m) \mapsto (\sum_j a_{1j} k_j, \ldots, \sum_j
a_{nj} k_j)\), which can be expressed as \(T = (S_1 \oplus 1^{n-1}) (T_2 \oplus 1^{n-2}) \cdots
(T_{n-1} \oplus 1^1) T_n\).  The above works as long as the vector spaces are not \(0\)-dimensional.
\(f \maps k^m \to \{0\}\) can be written as an \(m\)-fold tensor product of deletion, \(!^m\), and
\(f \maps \{0\} \to k^n\) can be written as an \(n\)-fold tensor product of zero, \(0^n\).  \(f
\maps \{0\} \to \{0\}\) is the empty morphism, which has an empty diagram for its string diagram.
\end{proof}

It is easy to see that the morphisms given in Lemma~\ref{gensvk} obey the following 18 relations:

\vskip 1em \noindent
\textbf{(1)--(3)}  Addition and zero make \(k\) into a commutative monoid:

  \begin{center}
    \scalebox{0.80}
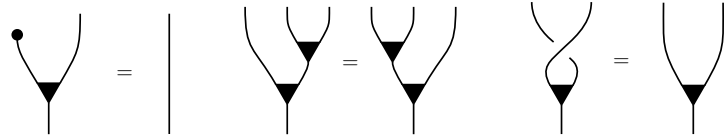
{
   \begin{tikzpicture}[-, thick, node distance=0.74cm]
   \node [plus] (summer) {};
   \node [coordinate] (sum) [below of=summer] {};
   \node [coordinate] (Lsum) [above left of=summer] {};
   \node [zero] (insert) [above of=Lsum, shift={(0,-0.35)}] {};
   \node [coordinate] (Rsum) [above right of=summer] {};
   \node [coordinate] (sumin) [above of=Rsum] {};
   \node (equal) [right of=Rsum, shift={(0,-0.26)}] {\(=\)};
   \node [coordinate] (in) [right of=equal, shift={(0,1)}] {};
   \node [coordinate] (out) [right of=equal, shift={(0,-1)}] {};

   \draw (insert) .. controls +(270:0.3) and +(120:0.3) .. (summer.left in)
         (summer.right in) .. controls +(60:0.6) and +(270:0.6) .. (sumin)
         (summer) -- (sum)    (in) -- (out);
   \end{tikzpicture}
        \hspace{1.0cm}
   \begin{tikzpicture}[-, thick, node distance=0.7cm]
   \node [plus] (uradder) {};
   \node [plus] (adder) [below of=uradder, shift={(-0.35,0)}] {};
   \node [coordinate] (urm) [above of=uradder, shift={(-0.35,0)}] {};
   \node [coordinate] (urr) [above of=uradder, shift={(0.35,0)}] {};
   \node [coordinate] (left) [left of=urm] {};

   \draw (adder.right in) .. controls +(60:0.2) and +(270:0.1) .. (uradder.io)
         (uradder.right in) .. controls +(60:0.35) and +(270:0.3) .. (urr)
         (uradder.left in) .. controls +(120:0.35) and +(270:0.3) .. (urm)
         (adder.left in) .. controls +(120:0.75) and +(270:0.75) .. (left)
         (adder.io) -- +(270:0.5);

   \node (eq) [right of=uradder, shift={(0,-0.25)}] {\(=\)};

   \node [plus] (ulsummer) [right of=eq, shift={(0,0.25)}] {};
   \node [plus] (summer) [below of=ulsummer, shift={(0.35,0)}] {};
   \node [coordinate] (ulm) [above of=ulsummer, shift={(0.35,0)}] {};
   \node [coordinate] (ull) [above of=ulsummer, shift={(-0.35,0)}] {};
   \node [coordinate] (right) [right of=ulm] {};

   \draw (summer.left in) .. controls +(120:0.2) and +(270:0.1) .. (ulsummer.io)
         (ulsummer.left in) .. controls +(120:0.35) and +(270:0.3) .. (ull)
         (ulsummer.right in) .. controls +(60:0.35) and +(270:0.3) .. (ulm)
         (summer.right in) .. controls +(60:0.75) and +(270:0.75) .. (right)
         (summer.io) -- +(270:0.5);
   \end{tikzpicture}
        \hspace{1.0cm}
   \begin{tikzpicture}[-, thick, node distance=0.7cm]
   \node [plus] (twadder) {};
   \node [coordinate] (twout) [below of=twadder] {};
   \node [coordinate] (twR) [above right of=twadder, shift={(-0.2,0)}] {};
   \node (cross) [above of=twadder] {};
   \node [coordinate] (twRIn) [above left of=cross, shift={(0,0.3)}] {};
   \node [coordinate] (twLIn) [above right of=cross, shift={(0,0.3)}] {};

   \draw (twadder.right in) .. controls +(60:0.35) and +(-45:0.25) .. (cross)
                            .. controls +(135:0.2) and +(270:0.4) .. (twRIn);
   \draw (twadder.left in) .. controls +(120:0.35) and +(-135:0.25) .. (cross.center)
                           .. controls +(45:0.2) and +(270:0.4) .. (twLIn);
   \draw (twout) -- (twadder);

   \node (eq) [right of=twR] {\(=\)};

   \node [coordinate] (L) [right of=eq] {};
   \node [plus] (adder) [below right of=L] {};
   \node [coordinate] (out) [below of=adder] {};
   \node [coordinate] (R) [above right of=adder] {};
   \node (cross) [above left of=R] {};
   \node [coordinate] (LIn) [above left of=cross] {};
   \node [coordinate] (RIn) [above right of=cross] {};

   \draw (adder.left in) .. controls +(120:0.7) and +(270:0.7) .. (LIn)
         (adder.right in) .. controls +(60:0.7) and +(270:0.7) .. (RIn)
         (out) -- (adder);
   \end{tikzpicture}
    }
\end{center}

\vskip 1em \noindent
\textbf{(4)--(6)}  Duplication and deletion make \(k\) into a cocommutative comonoid:

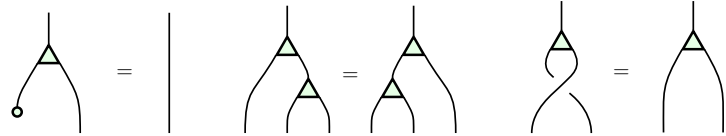
\begin{center}
    \scalebox{0.80}{
   \begin{tikzpicture}[-, thick, node distance=0.74cm]
   \node [delta] (dupe) {};
   \node [coordinate] (top) [above of=dupe] {};
   \node [coordinate] (Ldub) [below left of=dupe] {};
   \node [bang] (delete) [below of=Ldub, shift={(0,0.35)}] {};
   \node [coordinate] (Rdub) [below right of=dupe] {};
   \node [coordinate] (dubout) [below of=Rdub] {};
   \node (equal) [right of=Rdub, shift={(0,0.26)}] {\(=\)};
   \node [coordinate] (in) [right of=equal, shift={(0,1)}] {};
   \node [coordinate] (out) [right of=equal, shift={(0,-1)}] {};

   \draw (delete) .. controls +(90:0.3) and +(240:0.3) .. (dupe.left out)
         (dupe.right out) .. controls +(300:0.6) and +(90:0.6) .. (dubout)
         (dupe) -- (top)    (in) -- (out);
   \end{tikzpicture}
       \hspace{1.0cm}
   \begin{tikzpicture}[-, thick, node distance=0.7cm]
   \node [delta] (lrduper) {};
   \node [delta] (duper) [above of=lrduper, shift={(-0.35,0)}] {};
   \node [coordinate](lrm) [below of=lrduper, shift={(-0.35,0)}] {};
   \node [coordinate](lrr) [below of=lrduper, shift={(0.35,0)}] {};
   \node [coordinate](left) [left of=lrm] {};

   \draw (duper.right out) .. controls +(300:0.2) and +(90:0.1) .. (lrduper.io)
         (lrduper.right out) .. controls +(300:0.35) and +(90:0.3) .. (lrr)
         (lrduper.left out) .. controls +(240:0.35) and +(90:0.3) .. (lrm)
         (duper.left out) .. controls +(240:0.75) and +(90:0.75) .. (left)
         (duper.io) -- +(90:0.5);

   \node (eq) [right of=lrduper, shift={(0,0.25)}] {\(=\)};

   \node [delta] (lldubber) [right of=eq, shift={(0,-0.25)}] {};
   \node [delta] (dubber) [above of=lldubber, shift={(0.35,0)}] {};
   \node [coordinate] (llm) [below of=lldubber, shift={(0.35,0)}] {};
   \node [coordinate] (lll) [below of=lldubber, shift={(-0.35,0)}] {};
   \node [coordinate] (right) [right of=llm] {};

   \draw (dubber.left out) .. controls +(240:0.2) and +(90:0.1) .. (lldubber.io)
         (lldubber.left out) .. controls +(240:0.35) and +(90:0.3) .. (lll)
         (lldubber.right out) .. controls +(300:0.35) and +(90:0.3) .. (llm)
         (dubber.right out) .. controls +(300:0.75) and +(90:0.75) .. (right)
         (dubber.io) -- +(90:0.5);
   \end{tikzpicture}
       \hspace{1.0cm}
   \begin{tikzpicture}[-, thick, node distance=0.7cm]
   \node [coordinate] (twtop) {};
   \node [delta] (twdupe) [below of=twtop] {};
   \node [coordinate] (twR) [below right of=twdupe, shift={(-0.2,0)}] {};
   \node (cross) [below of=twdupe] {};
   \node [coordinate] (twROut) [below left of=cross, shift={(0,-0.3)}] {};
   \node [coordinate] (twLOut) [below right of=cross, shift={(0,-0.3)}] {};

   \draw (twdupe.left out) .. controls +(240:0.35) and +(135:0.25) .. (cross)
                           .. controls +(-45:0.2) and +(90:0.4) .. (twLOut)
         (twdupe.right out) .. controls +(300:0.35) and +(45:0.25) .. (cross.center)
                            .. controls +(-135:0.2) and +(90:0.4) .. (twROut)
         (twtop) -- (twdupe);

   \node (eq) [right of=twR] {\(=\)};

   \node [coordinate] (L) [right of=eq] {};
   \node [delta] (dupe) [above right of=L] {};
   \node [coordinate] (top) [above of=dupe] {};
   \node [coordinate] (R) [below right of=dupe] {};
   \node (uncross) [below left of=R] {};
   \node [coordinate] (LOut) [below left of=uncross] {};
   \node [coordinate] (ROut) [below right of=uncross] {};

   \draw (dupe.left out) .. controls +(240:0.7) and +(90:0.7) .. (LOut)
         (dupe.right out) .. controls +(300:0.7) and +(90:0.7) .. (ROut)
         (top) -- (dupe);
   \end{tikzpicture}
    }
\end{center}

\vskip 1em \noindent
\textbf{(7)--(10)}  The monoid and comonoid structures on \(k\) fit together to form a bimonoid:

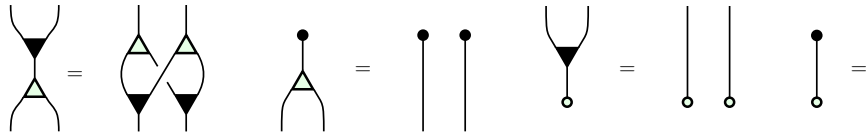
\begin{center}
    \scalebox{0.80}{
   \begin{tikzpicture}[thick]
   \node [plus] (adder) {};
   \node [coordinate] (f) [above of=adder, shift={(-0.4,-0.325)}] {};
   \node [coordinate] (g) [above of=adder, shift={(0.4,-0.325)}] {};
   \node [delta] (dupe) [below of=adder, shift={(0,0.25)}] {};
   \node [coordinate] (outL) [below of=dupe, shift={(-0.4,0.325)}] {};
   \node [coordinate] (outR) [below of=dupe, shift={(0.4,0.325)}] {};

   \draw (adder.io) -- (dupe.io)
         (f) .. controls +(270:0.4) and +(120:0.25) .. (adder.left in)
         (adder.right in) .. controls +(60:0.25) and +(270:0.4) .. (g)
         (dupe.left out) .. controls +(240:0.25) and +(90:0.4) .. (outL)
         (dupe.right out) .. controls +(300:0.25) and +(90:0.4) .. (outR);
   \end{tikzpicture}
      \raisebox{2.49em}{=}
      \hspace{1em}
   \begin{tikzpicture}[-, thick, node distance=0.7cm]
   \node [plus] (addL) {};
   \node (cross) [above right of=addL, shift={(-0.1,-0.0435)}] {};
   \node [plus] (addR) [below right of=cross, shift={(-0.1,0.0435)}] {};
   \node [delta] (dupeL) [above left of=cross, shift={(0.1,-0.0435)}] {};
   \node [delta] (dupeR) [above right of=cross, shift={(-0.1,-0.0435)}] {};
   \node [coordinate] (f) [above of=dupeL] {};
   \node [coordinate] (g) [above of=dupeR] {};
   \node [coordinate] (sum1) [below of=addL, shift={(0,0.2)}] {};
   \node [coordinate] (sum2) [below of=addR, shift={(0,0.2)}] {};

   \path
   (addL) edge (sum1) (addL.right in) edge (dupeR.left out) (addL.left in) edge [bend left=30] (dupeL.left out)
   (addR) edge (sum2) (addR.left in) edge (cross) (addR.right in) edge [bend right=30] (dupeR.right out)
   (dupeL) edge (f)
   (dupeL.right out) edge (cross)
   (dupeR) edge (g);
   \end{tikzpicture}
        \hspace{1.0cm}
   \begin{tikzpicture}[thick]
   \node [zero] (z) at (0,1) {};
   \node [delta] (dub) at (0,0.2) {};
   \node [coordinate] (oL) at (-0.35,-0.6) {};
   \node [coordinate] (oR) at (0.35,-0.6) {};

   \node (eq) at (1,0.42) {=};

   \node [zero] (zleft) at (2,1) {};
   \node [zero] (zright) at (2.7,1) {};
   \node [coordinate] (Lo) at (2,-0.6) {};
   \node [coordinate] (Ro) at (2.7,-0.6) {};

   \draw (z) -- (dub)
         (dub.left out) .. controls +(240:0.3) and +(90:0.5) .. (oL)
         (dub.right out) .. controls +(300:0.3) and +(90:0.5) .. (oR)
         (zleft) -- (Lo)
         (zright) -- (Ro);
   \end{tikzpicture}
        \hspace{1.0cm}
   \begin{tikzpicture}[thick]
   \node [bang] (b) at (0,-1) {};
   \node [plus] (sum) at (0,-0.2) {};
   \node [coordinate] (oL) at (-0.35,0.6) {};
   \node [coordinate] (oR) at (0.35,0.6) {};
   \node (spacemaker) at (0,-1.38) {};

   \node (eq) at (1,-0.47) {=};

   \node [bang] (bleft) at (2,-1) {};
   \node [bang] (bright) at (2.7,-1) {};
   \node [coordinate] (Lo) at (2,0.6) {};
   \node [coordinate] (Ro) at (2.7,0.6) {};

   \draw (b) -- (sum)
         (sum.left in) .. controls +(120:0.3) and +(270:0.5) .. (oL)
         (sum.right in) .. controls +(60:0.3) and +(270:0.5) .. (oR)
         (bleft) -- (Lo)
         (bright) -- (Ro);
   \end{tikzpicture}
       \hspace{1.0cm}
   \begin{tikzpicture}[thick]
   \node [zero] (z) at (0,0.11) {};
   \node [bang] (b) at (0,-1) {};
   \node (spacemaker) at (0,-1.38) {};

   \node (eq) at (0.7,-0.47) {=};

   \draw (z) -- (b);
   \end{tikzpicture}
    }
\end{center}

\vskip 1em \noindent
\textbf{(11)--(14)}  The rig structure of \(k\) can be recovered from the generating morphisms:

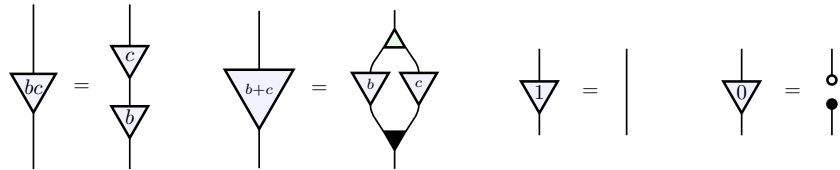
\begin{center}
    \scalebox{0.80}{
   \begin{tikzpicture}[-, thick]
   \node (top) {};
   \node [multiply] (c) [below of=top] {\(c\)};
   \node [multiply] (b) [below of=c] {\(b\)};
   \node (bottom) [below of=b] {};

   \draw (top) -- (c) -- (b) -- (bottom);

   \node (eq) [left of=b, shift={(0.2,0.5)}] {\(=\)};

   \node (bctop) [left of=top, shift={(-0.6,0)}] {};
   \node [multiply] (bc) [left of=eq, shift={(0.2,0)}] {\(bc\)};
   \node (bcbottom) [left of=bottom, shift={(-0.6,0)}] {};

   \draw (bctop) -- (bc) -- (bcbottom);
   \end{tikzpicture}
        \hspace{1.0cm}
   \begin{tikzpicture}[-, thick, node distance=0.85cm]
   \node (bctop) {};
   \node [multiply] (bc) [below of=bctop, shift={(0,-0.59)}] {\(\scriptstyle{b+c}\)};
   \node (bcbottom) [below of=bc, shift={(0,-0.59)}] {};

   \draw (bctop) -- (bc) -- (bcbottom);

   \node (eq) [right of=bc, shift={(0.15,0)}] {\(=\)};

   \node [multiply] (b) [right of=eq, shift={(0,0.1)}] {\(\scriptstyle{b}\)};
   \node [delta] (dupe) [above right of=b, shift={(-0.2,0.1)}] {};
   \node (top) [above of=dupe, shift={(0,-0.2)}] {};
   \node [multiply] (c) [below right of=dupe, shift={(-0.2,-0.1)}] {\(\scriptstyle{c}\)};
   \node [plus] (adder) [below right of=b, shift={(-0.2,-0.3)}] {};
   \node (out) [below of=adder, shift={(0,0.2)}] {};

   \draw
   (dupe.left out) .. controls +(240:0.15) and +(90:0.15) .. (b.90)
   (dupe.right out) .. controls +(300:0.15) and +(90:0.15) .. (c.90)
   (top) -- (dupe.io)
   (adder.io) -- (out)
   (adder.left in) .. controls +(120:0.15) and +(270:0.15) .. (b.io)
   (adder.right in) .. controls +(60:0.15) and +(270:0.15) .. (c.io);
   \end{tikzpicture}
        \hspace{1.0cm}
\raisebox{1.6em}{
   \begin{tikzpicture}[-, thick, node distance=0.85cm]
   \node (top) {};
   \node [multiply] (one) [below of=top] {1};
   \node (bottom) [below of=one] {};

   \draw (top) -- (one) -- (bottom);

   \node (eq) [right of=one] {\(=\)};
   \node (topid) [right of=top, shift={(0.6,0)}] {};
   \node (botid) [right of=bottom, shift={(0.6,0)}] {};

   \draw (topid) -- (botid);
   \end{tikzpicture}
}
        \hspace{1.0cm}
\raisebox{1.6em}{
   \begin{tikzpicture}[-, thick, node distance=0.85cm]
   \node [multiply] (prod) {\(0\)};
   \node (in0) [above of=prod] {};
   \node (out0) [below of=prod] {};
   \node (eq) [right of=prod] {\(=\)};
   \node [bang] (del) [right of=eq, shift={(-0.2,0.2)}] {};
   \node [zero] (ins) [right of=eq, shift={(-0.2,-0.2)}] {};
   \node (in1) [above of=del, shift={(0,-0.2)}] {};
   \node (out1) [below of=ins, shift={(0,0.2)}] {};

   \draw (in0) -- (prod) -- (out0);
   \draw (in1) -- (del);
   \draw (ins) -- (out1);
   \end{tikzpicture}
}
    }
\end{center}

\vskip 1em \noindent
\textbf{(15)--(16)}  Scalar multiplication commutes with addition and zero:

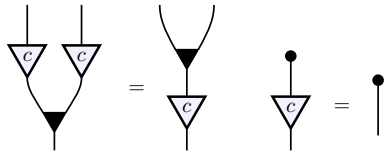
\begin{center}
    \scalebox{0.80}{
   \begin{tikzpicture}[-, thick, node distance=0.85cm]
   \node [plus] (adder) {};
   \node (out) [below of=adder, shift={(0,0.2)}] {};
   \node [multiply] (L) [above left of=adder, shift={(0.15,0.45)}] {\(c\)};
   \node [multiply] (R) [above right of=adder, shift={(-0.15,0.45)}] {\(c\)};
   \node [coordinate] (RIn) [above of=R] {};
   \node [coordinate] (LIn) [above of=L] {};

   \draw (adder.right in) .. controls +(60:0.2) and +(270:0.2) .. (R.io) (R) -- (RIn);
   \draw (adder.left in) .. controls +(120:0.2) and +(270:0.2) .. (L.io) (L) -- (LIn);
   \draw (out) -- (adder);
   \end{tikzpicture}
        \hspace{0.1cm}
   \begin{tikzpicture}[node distance=1.15cm]
   \node (eq){\(=\)};
   \node [below of=eq] {};
   \end{tikzpicture}
   \begin{tikzpicture}[-, thick, node distance=0.85cm]
   \node (out) {};
   \node [multiply] (c) [above of=out] {\(c\)};
   \node [plus] (adder) [above of=c] {};
   \node [coordinate] (L) [above left of=adder, shift={(0.15,0.25)}] {};
   \node [coordinate] (R) [above right of=adder, shift={(-0.15,0.25)}] {};

   \draw (R) .. controls +(270:0.2) and +(60:0.4) .. (adder.right in) (adder) -- (c) -- (out);
   \draw (L) .. controls +(270:0.2) and +(120:0.4) .. (adder.left in);
   \end{tikzpicture}
        \hspace{0.7cm}
   \begin{tikzpicture}[-, thick, node distance=0.85cm]
   \node [multiply] (prod) {\(c\)};
   \node (out0) [below of=prod] {};
   \node [zero] (ins0) [above of=prod] {};
   \node (eq) [right of=prod] {\(=\)};
   \node (out1) [below right of=eq] {};
   \node [zero] (ins1) [above of=out1, shift={(0,0.2)}] {};

   \draw (out0) -- (prod) -- (ins0);
   \draw (out1) -- (ins1);
   \end{tikzpicture}
       \hspace{0.7cm}
    }
\end{center}

\vskip 1em \noindent
\textbf{(17)--(18)}  Scalar multiplication commutes with duplication and deletion:

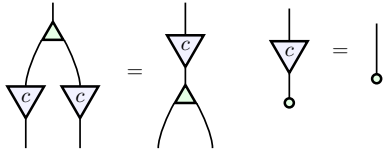
\begin{center}
    \scalebox{0.80}{
   \begin{tikzpicture}[-, thick, node distance=0.85cm]
   \node (top) {};
   \node [delta] (dupe) [below of=top, shift={(0,0.2)}] {};
   \node [multiply] (L) [below left of=dupe, shift={(0.15,-0.45)}] {\(c\)};
   \node [multiply] (R) [below right of=dupe, shift={(-0.15,-0.45)}] {\(c\)};
   \node [coordinate] (ROut) [below of=R] {};
   \node [coordinate] (LOut) [below of=L] {};

   \draw (dupe.left out) .. controls +(240:0.2) and +(90:0.2) .. (L.90) (L) -- (LOut);
   \draw (dupe.right out) .. controls +(300:0.2) and +(90:0.2) .. (R.90) (R) -- (ROut);
   \draw (top) -- (dupe);
   \end{tikzpicture}
        \hspace{0.1cm}
   \begin{tikzpicture}[node distance=1.15cm]
   \node (eq){\(=\)};
   \node [below of=eq] {};
   \end{tikzpicture}
   \begin{tikzpicture}[-, thick, node distance=0.85cm]
   \node (top) {};
   \node [multiply] (c) [below of=top] {\(c\)};
   \node [delta] (dupe) [below of=c] {};
   \node [coordinate] (L) [below left of=dupe, shift={(0.15,-0.25)}] {};
   \node [coordinate] (R) [below right of=dupe, shift={(-0.15,-0.25)}] {};

   \draw (dupe.left out) .. controls +(240:0.2) and +(90:0.2) .. (L);
   \draw (dupe.right out) .. controls +(300:0.2) and +(90:0.2) .. (R);
   \draw (top) -- (c) -- (dupe);
   \end{tikzpicture}
       \hspace{0.7cm}
   \begin{tikzpicture}[-, thick, node distance=0.85cm]
   \node [multiply] (prod) {\(c\)};
   \node (in0) [above of=prod] {};
   \node [bang] (del0) [below of=prod] {};
   \node (eq) [right of=prod] {\(=\)};
   \node (in1) [above right of=eq] {};
   \node [bang] (del1) [below of=in1, shift={(0,-0.2)}] {};
   \node [below of=del1, shift={(0,-0.2)}] {};

   \draw (in0) -- (prod) -- (del0);
   \draw (in1) -- (del1);
   \end{tikzpicture}
    }
  \end{center}

In fact, these relations are enough.  That is, together with the generating objects and morphisms,
they give a `presentation' of \(\Vectk\) as a symmetric monoidal category.  However, we need to 
make this concept precise.

Suppose \(C\) is generated by a set \(O\) of objects and a set \(M\) of morphisms going between
tensor products of objects in \(O\).  Define a \Define{formal morphism} to be a formal expression
built from symbols for morphisms in \(M\) via composition, identity morphisms, tensor product, the
unit object and the braiding.  Any formal morphism \(f\) can be \Define{evaluated} to obtain a
morphism \(\ev(f)\) in \(C\), which actually lies in \(C_0\).

Define a \Define{relation} to be a pair \(f,g\) of formal morphisms.  We say the relation
\Define{holds} in \(C\) if \(\ev(f) = \ev(g) \).  Suppose \(R\) is a set of relations that hold in
\(C\).  We say \((O,M,R)\) is a \Define{presentation} of \(C\) if given any two formal morphisms
\(j,k\) that evaluate to the same morphism, then we can go from \(j\) to \(k\) via a finite 
sequence of moves of these kinds:
\begin{enumerate}
\item replacing an instance of a generating morphism \(f\) in a formal morphism by the generating
morphism \(g\), where \((f,g) \in R\),
\item applying an equational law in the definition of strict symmetric monoidal category to a 
formal morphism.
\end{enumerate}
In intuitive terms, this means that there are enough relations to prove all the equations 
that hold in \(C\)---or more precisely, in the equivalent category \(C_0\).

\begin{thm} \label{presvk} The symmetric monoidal category \(\Vectk\) is presented by the object 
\(k\), the morphisms given in Lemma~\ref{gensvk}, and relations {\bf (1)--(18)} as listed above.
\end{thm}

\begin{proof}
To prove this, we show that these relations suffice to rewrite any formal morphism into a
standard form, with all formal morphisms that evaluate to the same morphism \(T \maps k^m \to k^n\)
in \(\Vectk\) having the same standard form.  To deal with moves of type (2), we draw formal 
morphisms as string diagrams built from generating morphisms and the braiding.  Two formal 
morphisms that differ only by equational laws in the definition of strict symmetric monoidal 
category will have topologically equivalent string diagrams.  It suffices, then, to show that any 
string diagram built from generating morphisms and the braiding can be put into a standard form 
using topological equivalences and relations \textbf{(1)--(18)}.

A qualitative description of this standard form will be helpful for understanding how an arbitrary
string diagram can be rewritten in this form.  By way of example, consider the linear
transformation \(T \maps \R^3 \to \R^2\) given by 
\[ (x_1, x_2, x_3) \mapsto (y_1, y_2) = (3x_1 + 7x_2 +
2x_3, 9x_1 + x_2).\]
Its standard form looks like this:
  \begin{center}
    \scalebox{0.80}
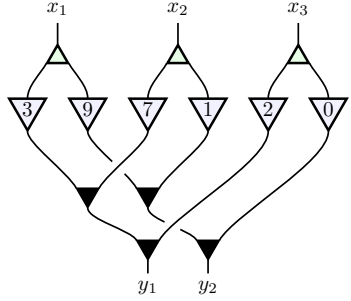
{
   \begin{tikzpicture}[thick]
   \node (i1) at (-2.5,3.4) {\(x_1\)};
   \node (i2) at (-0.5,3.4) {\(x_2\)};
   \node (i3) at (1.5,3.4) {\(x_3\)};
   \node[multiply] (a11) at (-3,1.732) {\(3\)};
   \node[multiply] (a21) at (-2,1.732) {\(9\)};
   \node[multiply] (a12) at (-1,1.732) {\(7\)};
   \node[multiply] (a22) at (0,1.732) {\(1\)};
   \node[multiply] (a13) at (1,1.732) {\(2\)};
   \node[multiply] (a23) at (2,1.732) {\(0\)};
   \node[delta] (D1) at (-2.5,2.6) {};
   \node[delta] (D2) at (-0.5,2.6) {};
   \node[delta] (D3) at (1.5,2.6) {};
   \node[plus] (S1T) at (-2,0.31) {};
   \node[plus] (S1B) at (-1,-0.556) {};
   \node[plus] (S2T) at (-1,0.31) {};
   \node[plus] (S2B) at (0,-0.556) {};
   \node (o1) at (-1,-1.24) {\(y_1\)};
   \node (o2) at (0,-1.24) {\(y_2\)};

   \draw
   (a11.io) .. controls +(270:0.2) and +(120:0.32) .. (S1T.left in)
   (a21.io) .. controls +(270:0.2) and +(120:0.32) .. (S2T.left in)
   (a22.io) .. controls +(270:0.2) and +(60:0.32) .. (S2T.right in)
   (S1T.io) .. controls +(270:0.2) and +(120:0.2) .. (S1B.left in)
   (i1) -- (D1) (i2) -- (D2) (i3) -- (D3)
   (S1B) -- (o1) (S2B) -- (o2)
   (S2T.io) .. controls +(270:0.2) and +(120:0.2) .. (S2B.left in)
   (a23.io) .. controls +(270:0.5) and +(60:0.5) .. (S2B.right in);
   \node [hole] (cross2) at (-0.59,-0.2) {};
   \node [hole] (cross1) at (-1.5,0.75) {};
   \draw
   (a12.io) .. controls +(270:0.2) and +(60:0.32) .. (S1T.right in)
   (a13.io) .. controls +(270:0.5) and +(60:0.5) .. (S1B.right in)
   (D1.left out) .. controls +(240:0.2) and +(90:0.2) .. (a11.90)
   (D1.right out) .. controls +(300:0.2) and +(90:0.2) .. (a21.90)
   (D2.left out) .. controls +(240:0.2) and +(90:0.2) .. (a12.90)
   (D2.right out) .. controls +(300:0.2) and +(90:0.2) .. (a22.90)
   (D3.left out) .. controls +(240:0.2) and +(90:0.2) .. (a13.90)
   (D3.right out) .. controls +(300:0.2) and +(90:0.2) .. (a23.90);
   \end{tikzpicture}
    }
\end{center}
This is a string diagram picture of the following equation:
   \[ T x =
   \left( \begin{array}{ccc}
   3 & 7 & 2 \\
   9 & 1 & 0 \end{array} \right)
   \left( \begin{array}{c} x_1 \\ x_2 \\ x_3 \end{array} \right) =
   \left( \begin{array}{c} y_1 \\ y_2 \end{array} \right)
  \]

In general, given a \(k\)-linear transformation \(T \maps k^m \to k^n\), we can describe it using an
\(n \times m\) matrix with entries in \(k\).  The case where \(m\) and/or \(n\) is zero gives a
matrix with no entries, so their standard form will be treated separately.  For positive values of
\(m\) and \(n\), the standard form has three distinct layers.  The top layer consists of \(m\)
clusters of \(n-1\) instances of \(\Delta\).  The middle layer is \(mn\) multiplications.  The \(n\)
outputs of the \(j\)th cluster connect to the inputs of the multiplications \(\{a_{1j}, \dotsc,
a_{nj}\}\), where \(a_{ij}\) is the \(ij\) entry of \(A\), the matrix for \(T\).  The bottom layer
consists of \(n\) clusters of \(m-1\) instances of \(+\).  There will generally be braiding in this
layer as well, but since the category is being generated as symmetric monoidal, the locations of the
braidings doesn't matter so long as the topology of the string diagram is preserved.  The topology
of the sum layer is that the \(i\)th sum cluster gets its \(m\) inputs from the outputs of the
multiplications \(\{a_{i1}, \dotsc, a_{im}\}\).  The arrangement of the instances of \(\Delta\) and
\(+\) within their respective clusters does not matter, due to the associativity of \(+\) via
relation \textbf{(2)} and coassociativity of \(\Delta\) via relation \textbf{(5)}.  For the sake of
making the standard form explicit with respect to these relations, we may assume the right output of
a \(\Delta\) is always connected to a multiplication input, and the right input of a \(+\) is always
connected to a multiplication output.  This gives a prescription for drawing the standard form of a
string diagram with a corresponding matrix \(A\).

The standard form for \(T \maps k^0 \to k^n\) is \(n\) zeros (\(0 \oplus \dotsb \oplus 0\)), and the
standard form for \(T \maps k^m \to k^0\) is \(m\) deletions (\(! \oplus \dotsb \oplus\, !\)).

Each of the generating morphisms can easily be put into standard form: the string diagrams for zero,
deletion, and multiplication are already in standard form.  The string diagram for duplication
(resp. addition) can be put into standard form by attaching a multiplication by \(1\), relation
\textbf{(13)}, to each of the outputs (resp. inputs).
  \begin{center}
    \scalebox{0.80}
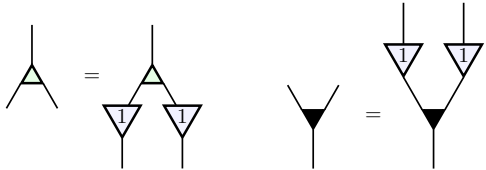
{
   \begin{tikzpicture}[thick]
   \node[delta] (dupe) {};
   \node[multiply] (o1) at (-0.5,-0.65) {\(1\)};
   \node[multiply] (o2) at (0.5,-0.65) {\(1\)};
   \node (in) [above of=dupe] {};
   \node (posto1) [below of=o1] {};
   \node (posto2) [below of=o2] {};

   \draw[rounded corners] (posto1) -- (o1) -- (dupe.left out);
   \draw[rounded corners] (posto2) -- (o2) -- (dupe.right out);
   \draw (in) -- (dupe);

   \node (equals) at (-1,0) {\(=\)};
   \node[delta] (dub) at (-2,0) {};
   \node (dubin) [above of=dub] {};
   \node (L) at (-2.5,-0.65) {};
   \node (R) at (-1.5,-0.65) {};

   \draw (L) -- (dub.left out) (R) -- (dub.right out) (dubin) -- (dub);
   \end{tikzpicture}
\hspace{1cm}
   \begin{tikzpicture}[thick]
   \node[plus] (summer) {};
   \node[multiply] (o1) at (-0.5,1.02) {\(1\)};
   \node[multiply] (o2) at (0.5,1.02) {\(1\)};
   \node (out) [below of=summer] {};
   \node (posto1) [above of=o1] {};
   \node (posto2) [above of=o2] {};

   \draw (posto1) -- (o1) -- (o1.io) -- (summer.left in);
   \draw (posto2) -- (o2) -- (o2.io) -- (summer.right in);
   \draw (out) -- (summer);

   \node (equals) at (-1,0) {\(=\)};
   \node[plus] (adder) at (-2,0) {};
   \node (addout) [below of=adder] {};
   \node (L) at (-2.5,0.65) {};
   \node (R) at (-1.5,0.65) {};

   \draw (L) -- (adder.left in) (R) -- (adder.right in) (addout) -- (adder);
   \end{tikzpicture}
    }
  \end{center}
The braiding morphism is just as basic to our argument as the generating morphisms, so we will need
to write the string diagram for \(B\) in standard form as well.  The matrix corresponding to
braiding is
\[\left( \begin{array}{cc}
   0 & 1 \\
   1 & 0 \end{array} \right),\]
so its standard form is as follows:
  \begin{center}
    \scalebox{0.80}
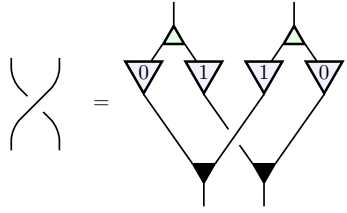
{
   \begin{tikzpicture}[thick]
   \node (UpUpLeft) at (-3.7,2.2) {};
   \node [coordinate] (UpLeft) at (-3.7,1.7) {};
   \node (mid) at (-3.3,1.3) {};
   \node [coordinate] (DownRight) at (-2.9,0.9) {};
   \node (DownDownRight) at (-2.9,0.4) {};
   \node [coordinate] (UpRight) at (-2.9,1.7) {};
   \node (UpUpRight) at (-2.9,2.2) {};
   \node [coordinate] (DownLeft) at (-3.7,0.9) {};
   \node (DownDownLeft) at (-3.7,0.4) {};

   \draw [rounded corners=2mm] (UpUpLeft) -- (UpLeft) -- (mid) --
   (DownRight) -- (DownDownRight) (UpUpRight) -- (UpRight) -- (DownLeft) -- (DownDownLeft);

   \node (eq) at (-2.2,1.3) {\(=\)};
   \node[coordinate] (i1) at (-1,3) {};
   \node[coordinate] (i2) at (1,3) {};
   \node[multiply] (a11) at (-1.5,1.832) {\(0\)};
   \node[multiply] (a21) at (-0.5,1.832) {\(1\)};
   \node[multiply] (a12) at (0.5,1.832) {\(1\)};
   \node[multiply] (a22) at (1.5,1.832) {\(0\)};
   \node[delta] (D1) at (-1,2.4) {};
   \node[delta] (D2) at (1,2.4) {};
   \node (cross) at (0,0.75) {};
   \node[plus] (S1) at (-0.5,0.2) {};
   \node[plus] (S2) at (0.5,0.2) {};
   \node[coordinate] (o1) at (-0.5,-0.4) {};
   \node[coordinate] (o2) at (0.5,-0.4) {};

   \draw
   (i1) -- (D1) (i2) -- (D2)
   (S1.io) -- (o1) (S2.io) -- (o2)
   (a11.io) -- (S1.left in) (a22.io) -- (S2.right in)
   (a21.io) -- (cross) -- (S2.left in) (a12.io) -- (S1.right in)
   (D1.left out) -- (a11) (D1.right out) -- (a21)
   (D2.left out) -- (a12) (D2.right out) -- (a22);
   \end{tikzpicture}
    }
\end{center}

For \(n > 1\), any morphism built from \(n\) copies of the \Define{basic morphisms}---that is,
generating morphisms and the braiding---can be built up from a morphism built from \(n-1\) copies by
composing or tensoring with one more basic morphism.  Thus, to prove that any string diagram built
from basic morphisms can be put into its standard form, we can proceed by induction on the number of
basic morphisms.

Furthermore, because strings can be extended using the identity morphism, relation \textbf{(13)} can
be used to show tensoring with any generating morphism is equivalent to tensoring with \(1\),
followed by a composition:  \(\Delta = \Delta \of 1\), \(+ = 1 \of +\), \(c = 1 \of c\), \(! =\, !
\of 1\), \(0 = 1 \of 0\).  In the case of braiding, the step of tensoring with \(1\) is repeated
once before making the composition:  \(B = (1 \oplus 1) \of B\).
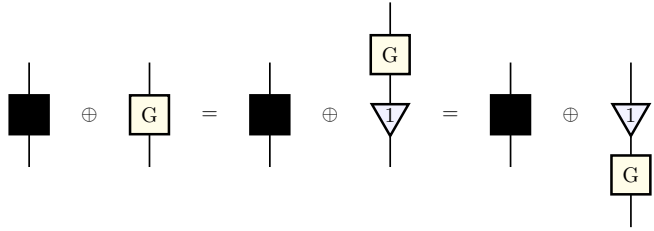
  \begin{center}
    \scalebox{0.80}{
   \begin{tikzpicture}[thick]
   \node[blackbox] (bb1) {};
   \node (tensor1) [right of=bb1] {\(\oplus\)};
   \node[sqnode] (gen1) [right of=tensor1] {G};
   \node (eq1) [right of=gen1] {\(=\)};
   \node[blackbox] (bb2) [right of=eq1] {};
   \node (tensor2) [right of=bb2] {\(\oplus\)};
   \node[sqnode] (gen2) [right of=tensor2, shift={(0,1)}] {G};
   \node (eq2) [right of=gen2, shift={(0,-1)}] {\(=\)};
   \node[blackbox] (bb3) [right of=eq2] {};
   \node (tensor3) [right of=bb3] {\(\oplus\)};
   \node[sqnode] (gen3) [right of=tensor3, shift={(0,-1)}] {G};

   \node[multiply] (times2) [below of=gen2] {\(1\)};
   \node[multiply] (times3) [above of=gen3] {\(1\)};

   \node (bb1in) [above of=bb1] {};
   \node (bb1out) [below of=bb1] {};
   \node (bb2in) [above of=bb2] {};
   \node (bb2out) [below of=bb2] {};
   \node (bb3in) [above of=bb3] {};
   \node (bb3out) [below of=bb3] {};
   \node (gen1in) [above of=gen1] {};
   \node (gen1out) [below of=gen1] {};
   \node (gen2in) [above of=gen2] {};
   \node (gen2out) [below of=times2] {};
   \node (gen3in) [above of=times3] {};
   \node (gen3out) [below of=gen3] {};

   \draw (bb1in) -- (bb1) -- (bb1out) (gen1in) -- (gen1) -- (gen1out)
   (bb2in) -- (bb2) -- (bb2out) (gen2in) -- (gen2) -- (times2) -- (gen2out)
   (bb3in) -- (bb3) -- (bb3out) (gen3in) -- (times3) -- (gen3) -- (gen3out);
   \end{tikzpicture}
    }
  \end{center}
Thus there are 11 cases to consider for this induction:  \(\oplus 1\), \(+ \of\), \(\of \Delta\),
\(\Delta \of\), \(\of +\), \(\of c\), \(c \of\), \(\of 0\), \(! \of\), \(B \of\), \(\of B\).
Without loss of generality, the string diagram \(S\) to which a generating morphism is added will be
assumed to be in standard form already.  Labels \(ij\) on diagrams illustrating these cases
correspond to strings incident to the multiplications \(a_{ij}\).
     \begin{itemize}[leftmargin=1em]
\item \Define{\(\oplus 1\)}\\
When tensoring morphisms together, the matrix corresponding to \(C \oplus D\) is the block diagonal
matrix
\[\left( \begin{array}{cc}
   C & 0 \\
   0 & D \end{array} \right),\]
where, by abuse of notation, the block \(C\) is the matrix corresponding to morphism \(C\), and
respectively \(D\) with \(D\).  Thus, when tensoring \(S\) by \(1\), we write the matrix for \(S\)
with one extra row and one extra column.  Each of these new entries will be \(0\) with the exception
of a \(1\) at the bottom of the extra column.  The string diagram corresponding to the new matrix
can be drawn in standard form as prescribed above.  Using relations \textbf{(14)}, \textbf{(4)}, and
\textbf{(1)}, the standard form reduces to \(S \oplus 1\).  The process is reversable (\(\ev(f) =
\ev(g)\) implies \(\ev(g) = \ev(f)\)), so if the string diagram \(S\) can be drawn in standard form,
the string diagram \(S \oplus 1\) can be drawn in standard form, too.  The diagrams below show the
relevant strings before they are reduced.
  \begin{center}
    \scalebox{0.80}
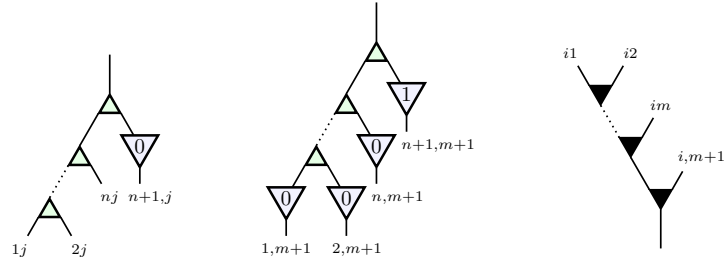
{
   \begin{tikzpicture}[thick]
   \node[delta] (D1) {};
   \node[delta] (D2) at (-0.5,-0.866) {};
   \node[delta] (D3) at (-1,-1.732) {};
   \node[multiply] (zero) at (0.5,-.65) {\(0\)};
   \node (i1) [above of=D1] {};
   \node (o1) at (0.5,-1.516) {\(\quad\scriptstyle{n+1,j}\)};
   \node (o2) at (0,-1.516) {\(\scriptstyle{nj}\)};
   \node (o3) at (-0.5,-2.382) {\(\scriptstyle{2j}\)};
   \node (o4) at (-1.5,-2.382) {\(\scriptstyle{1j}\)};

   \draw (i1) -- (D1) (D1.left out) -- (D2.io) (D1.right out) -- (zero)
   -- (o1) (D2.right out) -- (o2) (D3.right out) -- (o3) (D3.left out) -- (o4);
   \draw[dotted] (D2.left out) -- (D3.io);
   \end{tikzpicture}
       \hspace{1cm}
   \begin{tikzpicture}[thick]
   \node[delta] (D1) {};
   \node[delta] (D2) at (-0.5,-0.866) {};
   \node[delta] (D3) at (-1,-1.732) {};
   \node[multiply] (one) at (0.5,-.65) {\(1\)};
   \node (o1) at (0.5,-1.516) {\(\quad\qquad\scriptstyle{n+1,m+1}\)};
   \node (i1) [above of=D1] {};
   \node[multiply] (z2) at (0,-1.516) {\(0\)};
   \node (o2) at (0,-2.382) {\(\qquad\scriptstyle{n,m+1}\)};
   \node[multiply] (z3) at (-0.5,-2.382) {\(0\)};
   \node (o3) at (-0.5,-3.248) {\(\quad\scriptstyle{2,m+1}\)};
   \node[multiply] (z4) at (-1.5,-2.382) {\(0\)};
   \node (o4) at (-1.5,-3.248) {\(\scriptstyle{1,m+1}\)};

   \draw (i1) -- (D1) (D1.left out) -- (D2.io) (D1.right out) -- (one)
   -- (o1) (D2.right out) -- (z2) -- (o2) (D3.right out) -- (z3) --
   (o3) (D3.left out) -- (z4) -- (o4);
   \draw[dotted] (D2.left out) -- (D3.io);
   \end{tikzpicture}
       \hspace{1cm}
   \begin{tikzpicture}[thick]
   \node[plus] (P1) {};
   \node[plus] (P2) at (-0.5,0.866) {};
   \node[plus] (P3) at (-1,1.732) {};
   \node (zero) at (0.5,.65) {\(\quad\scriptstyle{i,m+1}\)};
   \node (o1) [below of=P1] {};
   \node (i2) at (0,1.516) {\(\scriptstyle{im}\)};
   \node (i3) at (-0.5,2.382) {\(\scriptstyle{i2}\)};
   \node (i4) at (-1.5,2.382) {\(\scriptstyle{i1}\)};

   \draw (o1) -- (P1) (P1.left in) -- (P2.io) (P1.right in) -- (zero)
   (P2.right in) -- (i2) (P3.right in) -- (i3) (P3.left in) -- (i4);
   \draw[dotted] (P2.left in) -- (P3.io);
   \end{tikzpicture}
    }
  \end{center}
Note that for \(i = n+1\) the multiplications \(a_{i2}, \dotsc, a_{im}\) going to the sum cluster
will be multiplication by zero, and \(a_{i,m+1} = 1\).  Otherwise \(a_{i,m+1} = 0\), and the rest
depend on the matrix corresponding to \(S\).  When \(S = (! \oplus \dotsb \oplus !)\), the matrix
corresponding to \(S \oplus 1\) has a single row, \((0 \cdots 0 \; 1)\), and the standard form
generated is just the middle diagram above.  When the same simplifications are applied, no sum
cluster exists to eliminate the zeros, so the standard form still simplifies to \(S \oplus 1\).
Dually, when \(S = (0 \oplus \dotsb \oplus 0)\), the matrix representation of \(S \oplus 1\) is a
column matrix.  No duplication cluster exists in the standard form for this matrix, so the same
simplifications again reduce to \(S \oplus 1\).
\item \Define{\(+ \of\)}\\
If we compose the string diagram for addition with \(S\), first consider only the affected clusters
of additions: two clusters are combined into a larger cluster.  Without loss of generality we can
assume these are the first two clusters, or formally, \((+ \oplus 1^{n-2})(S)\).  We can rearrange
the sums using the associative law, relation \textbf{(2)}, and permute the inputs of this
large cluster using the commutative law, relation \textbf{(3)}.  After several iterations of these
two relations, the desired result is obtained:
  \begin{center}
    \scalebox{0.80}
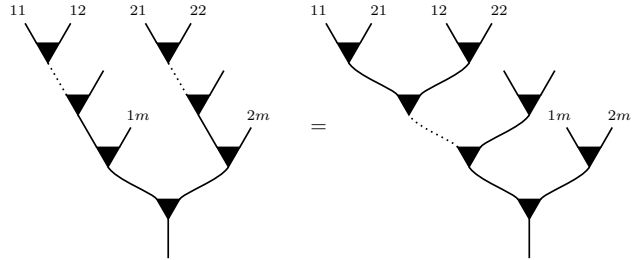
{
   \begin{tikzpicture}[thick]
   \node[plus] (base)                {};
   \node (out) [below of=base]       {};
   \node (Laim)     at (-0.25,0.217) {};
   \node (Lbase)    at (-0.87,0.433) {};
   \node[plus] (L1) at (-1.0,0.866)  {};
   \node[plus] (L2) at (-1.5,1.732)  {};
   \node[plus] (L3) at (-2.0,2.598)  {};
   \node (Raim)     at (0.25,0.217)  {};
   \node (Rbase)    at (0.87,0.433)  {};
   \node[plus] (R1) at (1.0,0.866)   {};
   \node[plus] (R2) at (0.5,1.732)   {};
   \node[plus] (R3) at (0.0,2.598)   {};
   \node (iL1)      at (-0.5,1.516)  {\(\scriptstyle{1m}\)};
   \node (iL2)      at (-1.0,2.382)  {};
   \node (iL3)      at (-1.5,3.248)  {\(\scriptstyle{12}\)};
   \node (iL4)      at (-2.5,3.248)  {\(\scriptstyle{11}\)};
   \node (iR1)      at (1.5,1.516)   {\(\scriptstyle{2m}\)};
   \node (iR2)      at (1.0,2.382)   {};
   \node (iR3)      at (0.5,3.248)   {\(\scriptstyle{22}\)};
   \node (iR4)      at (-0.5,3.248)  {\(\scriptstyle{21}\)};

   \draw (base.left in) .. controls (Laim) and (Lbase) .. (L1.io);
   \draw (base.right in) .. controls (Raim) and (Rbase) .. (R1.io);
   \draw (R1.left in) -- (R2.io) (L1.left in) -- (L2.io) (out) -- (base)
   (L1.right in) -- (iL1) (L2.right in) -- (iL2) (L3.right in) -- (iL3) (L3.left in) -- (iL4)
   (R1.right in) -- (iR1) (R2.right in) -- (iR2) (R3.right in) -- (iR3) (R3.left in) -- (iR4);
   \draw[dotted] (R2.left in) -- (R3.io) (L2.left in) -- (L3.io);

   \node (eq) at (2.5,1.3) {\(=\)};

   \node[plus] (base2) at (6,0)       {};
   \node (basement) [below of=base2]  {};
   \node[plus] (R)     at (7,0.866)   {};
   \node[plus] (L)     at (5,0.866)   {};
   \node[plus] (LR)    at (6,1.732)   {};
   \node[plus] (LL)    at (4,1.732)   {};
   \node[plus] (LLL)   at (3,2.598)   {};
   \node[plus] (LLR)   at (5,2.598)   {};
   \node (iRl)         at (6.5,1.516) {\(\scriptstyle{1m}\)};
   \node (iRr)         at (7.5,1.516) {\(\scriptstyle{2m}\)};
   \node (iLRl)        at (5.5,2.382) {};
   \node (iLRr)        at (6.5,2.382) {};
   \node (iLLRl)       at (4.5,3.248) {\(\scriptstyle{12}\)};
   \node (iLLRr)       at (5.5,3.248) {\(\scriptstyle{22}\)};
   \node (iLLLl)       at (2.5,3.248) {\(\scriptstyle{11}\)};
   \node (iLLLr)       at (3.5,3.248) {\(\scriptstyle{21}\)};

   \draw (base2.left in) .. controls (5.75,0.217) and (5.13,0.433) .. (L.io)
   (base2.right in) .. controls (6.25,0.217) and (6.87,0.433) .. (R.io)
   (base2) -- (basement) (R.right in) -- (iRr) (R.left in) -- (iRl)
   (LR.right in) -- (iLRr) (LR.left in) -- (iLRl)
   (LLR.right in) -- (iLLRr) (LLR.left in) -- (iLLRl)
   (LLL.right in) -- (iLLLr) (LLL.left in) -- (iLLLl)
   (L.right in) .. controls (5.25,1.083) and (5.87,1.299) .. (LR.io)
   (LL.left in) .. controls (3.75,1.949) and (3.13,2.165) .. (LLL.io)
   (LL.right in) .. controls (4.25,1.949) and (4.87,2.165) .. (LLR.io);
   \draw[dotted] (L.left in) .. controls (4.75,1.083) and (4.13,1.299) .. (LL.io);
   \end{tikzpicture}
    }
  \end{center}
Now the right side of relation \textbf{(12)} appears in the diagram \(m\) times with \(a_{1j}\) and
\(a_{2j}\) in place of \(b\) and \(c\).  Relation \textbf{(12)} can therefore be used to simplify
to the multiplications \(a_{1j}+a_{2j}\).
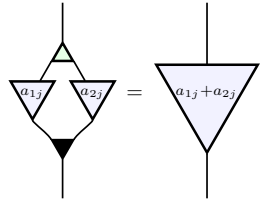
  \begin{center}
    \scalebox{0.80}{
   \begin{tikzpicture}[thick]
   \node[delta] (dub) {};
   \node (top) [above of=dub] {};
   \node[multiply] (a1j) at (-0.5,-0.65) {\(\!\scriptstyle{a_{1j}}\!\)};
   \node[multiply] (a2j) at (0.5,-0.65) {\(\!\scriptstyle{a_{2j}}\!\)};
   \node[plus] (sum) at (0,-1.516) {};
   \node (bottom) [below of=sum] {};

   \draw (top) -- (dub) (dub.left out) -- (a1j)
   (dub.right out) -- (a2j) (sum) -- (bottom)
   (a1j.io) .. controls (-0.47,-1.15) and (-0.25,-1.299) .. (sum.left in)
   (a2j.io) .. controls (0.47,-1.15) and (0.25,-1.299) .. (sum.right in);

   \node (eq) at (1.2,-0.65) {\(=\)};
   \node[multiply] (added) at (2.4,-0.65) {\(\!\!\!{}^{a_{1j}+a_{2j}}\!\!\!\)};
   \node (in) at (2.4,1) {};
   \node (out) at (2.4,-2.516) {};

   \draw (in) -- (added) -- (out);
   \end{tikzpicture}
    }
  \end{center}
The simplification removes one instance of \(\Delta\) from each of the \(m\) clusters of \(\Delta\)
and \(m\) instances of \(+\) from the large addition cluster.  There will remain \((m-1) + (m-1) +
(1) - (m) = m-1\) instances of \(+\), which is the correct number for the cluster.  I.e. the
composition has been reduced to standard form.
\\
The argument is vastly simpler if \(S = (0 \oplus \dotsb \oplus 0)\).  In that case relation
\textbf{(1)} deletes the addition and one of the \(0\) morphisms, and \(S\) is still in the same
form.
  \begin{center}
    \scalebox{0.80}
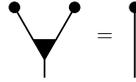
{
   \begin{tikzpicture}[thick]
   \node [plus] (sum) at (0,0) {};
   \node [zero] (zero1) at (-0.5,0.65) {};
   \node [zero] (zero2) at (0.5,0.65) {};
   \node (eq) at (1,0.15) {\(=\)};
   \node [zero] (zero3) at (1.5,0.65) {};
   \node (sumout) at (0,-0.65) {};
   \node (zeroout) at (1.5,-0.65) {};

   \draw (zero1) -- (sum.left in) (sum.right in) -- (zero2) (sum.io)
   -- (sumout) (zero3) -- (zeroout);
   \end{tikzpicture}
    }
  \end{center}
\item \Define{\(\of \Delta\)}\\
The argument for \(S \of (\Delta \oplus 1^{m-2})\) is dual to the above argument, using the light
relations \textbf{(4)}, \textbf{(5)} and \textbf{(6)} instead of the dark relations \textbf{(1)},
\textbf{(2)} and \textbf{(3)}.
\item \Define{\(\Delta \of\)}\\
For \((\Delta \oplus 1^{n-1}) \of S\), relation \textbf{(7)} can be used iteratively to ``float''
the \(\Delta\) layer above each of the two \(+\) clusters formed by the first iteration.
  \begin{center}
    \scalebox{0.80}
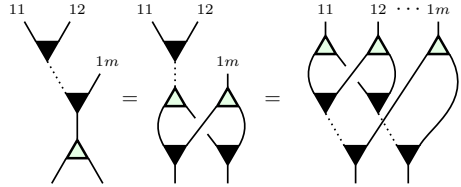
{
   \begin{tikzpicture}[thick]
   \node[plus] (base) {};
   \node[delta] (dub) at (0,-0.866) {};
   \node (oLa) at (-0.5,-1.516) {};
   \node (oRa) at (0.5,-1.516) {};
   \node (1ma) at (0.5,0.65) {\(\scriptstyle{1m}\)};
   \node[plus] (top) at (-0.5,0.866) {};
   \node (11a) at (-1,1.516) {\(\scriptstyle{11}\)};
   \node (12a) at (0,1.516) {\(\scriptstyle{12}\)};

   \draw[dotted] (base.left in) -- (top.io);
   \draw (oLa) -- (dub.left out) (oRa) -- (dub.right out) (dub) -- (base)
   (base.right in) -- (1ma) (top.right in) -- (12a) (top.left in) -- (11a);

   \node (eq1) at (0.875,0) {\(=\)};
   \node (oLb) at (1.625,-1.516) {};
   \node[plus] (baseLb) at (1.625,-0.866) {};
   \node[delta] (dubLb) at (1.625,0) {};
   \node[plus] (topb) at (1.625,0.866) {};
   \node (11b) at (1.125,1.516) {\(\scriptstyle{11}\)};
   \node (12b) at (2.125,1.516) {\(\scriptstyle{12}\)};
   \node (oRb) at (2.5,-1.516) {};
   \node[plus] (baseRb) at (2.5,-0.866) {};
   \node[delta] (dubRb) at (2.5,0) {};
   \node (1mb) at (2.5,0.65) {\(\scriptstyle{1m}\)};
   \node (crossb) at (2.0625,-0.433) {};

   \draw[dotted] (dubLb) -- (topb);
   \draw (11b) -- (topb.left in) (topb.right in) -- (12b) (dubRb) -- (1mb)
   (dubLb.right out) -- (crossb) -- (baseRb.left in) 
   (dubRb.left out) -- (baseLb.right in)
   (baseRb) -- (oRb) (baseLb) -- (oLb);
   \path (baseLb.left in) edge [bend left=30] (dubLb.left out)
   (baseRb.right in) edge [bend right=30] (dubRb.right out);

   \node (eq2) at (3.25,0) {\(=\)};
   \node (11c) at (4.125,1.516) {\(\scriptstyle{11}\)};
   \node[delta] (dubLc) at (4.125,0.866) {};
   \node[plus] (topLc) at (4.125,0) {};
   \node (ucross) at (4.5625,0.433) {};
   \node (12c) at (5,1.516) {\(\scriptstyle{12}\)};
   \node[delta] (dubRc) at (5,0.866) {};
   \node[plus] (topRc) at (5,0) {};
   \node[plus] (baseLc) at (4.625,-0.866) {};
   \node (oLc) at (4.625,-1.516) {};
   \node (1mc) at (6,1.516) {\(\scriptstyle{1m}\)};
   \node (dots) at (5.5,1.516) {\(\cdots\)};
   \node[delta] (dubmc) at (6,0.866) {};
   \node[plus] (baseRc) at (5.5,-0.866) {};
   \node (oRc) at (5.5,-1.516) {};

   \draw[dotted] (baseLc.left in) -- (topLc.io) (baseRc.left in) -- (topRc.io);
   \draw (oRc) -- (baseRc) (oLc) -- (baseLc) (dubmc) -- (1mc) (dubLc) -- (11c) (dubRc) -- (12c)
   (topLc.right in) -- (dubRc.left out) 
   (topRc.left in) -- (ucross) -- (dubLc.right out) 
   (baseLc.right in) -- (dubmc.left out)
   (baseRc.right in) .. controls (5.875,-0.433) and (6.625,0) .. (dubmc.right out);
   \path (topLc.left in) edge [bend left=30] (dubLc.left out)
   (topRc.right in) edge [bend right=30] (dubRc.right out);
   \end{tikzpicture}
    }
  \end{center}
Each of these instances of \(\Delta\) can pass through the multiplication layer to \(\Delta\)
clusters using relation \textbf{(17)}.
\\
As before, we consider the subcase \(S = (0 \oplus \dotsb \oplus 0)\) separately.  Relation
\textbf{(8)} removes the duplication and creates a new zero, so \(S\) remains in the same form.
\item \Define{\(\of +\)}\\
For \(S(+ \oplus 1^{m-1})\), the argument is dual to the previous one: relation \textbf{(7)} is used
to ``float'' the additions down, relation \textbf{(15)} sends the additions through the
multiplications, and relation \textbf{(9)} removes the addition and creates a new deletion in the
subcase \(S = (! \oplus \dotsb \oplus !)\).
\item \Define{\(\of c\)}\\
We can iterate relation \textbf{(17)} when a multiplication is composed on top, as in \(S(c \oplus
1^{m-1})\).
  \begin{center}
    \scalebox{0.80}
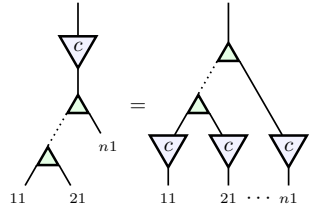
{

   \begin{tikzpicture}[thick]
   \node[multiply] (topc) {\(c\)};
   \node (top1) at (0,0.866) {};
   \node[delta] (dub1) [below of=topc] {};
   \node[delta] (dub2) at (-0.5,-1.866) {};
   \node (11a) at (-1,-2.516) {\(\scriptstyle{11}\)};
   \node (21a) at (0,-2.516) {\(\scriptstyle{21}\)};
   \node (n1a) at (0.5,-1.65) {\(\scriptstyle{n1}\)};

   \draw[dotted] (dub1.left out) -- (dub2.io);
   \draw (top1) -- (topc) -- (dub1) (dub1.right out) -- (n1a)
   (dub2.left out) -- (11a) (dub2.right out) -- (21a);

   \node (eq) at (1,-1) {\(=\)};
   \node (top2) at (2.5,0.866) {};
   \node[delta] (delt1) at (2.5,-0.134) {};
   \node[delta] (delt2) at (2,-1) {};
   \node[multiply] (c1) at (1.5,-1.65) {\(c\)};
   \node[multiply] (c2) at (2.5,-1.65) {\(c\)};
   \node[multiply] (c3) at (3.5,-1.65) {\(c\)};
   \node (11b) at (1.5,-2.516) {\(\scriptstyle{11}\)};
   \node (21b) at (2.5,-2.516) {\(\scriptstyle{21}\)};
   \node (dots) at (3,-2.516) {\(\cdots\)};
   \node (n1b) at (3.5,-2.516) {\(\scriptstyle{n1}\)};

   \draw[dotted] (delt1.left out) -- (delt2.io);
   \draw (delt2.left out) -- (c1) -- (11b) (delt2.right out) -- (c2) -- (21b)
   (delt1.right out) -- (c3) -- (n1b) (top2) -- (delt1);
   \end{tikzpicture}
    }
  \end{center}
The double multiplications in the multiplication layer reduce to a single multiplication via
relation \textbf{(11)}, \(c \of a_{ij} = ca_{ij}\), which leaves the diagram in standard form.  The
composition does nothing when \(S = (! \oplus \dotsb \oplus !)\), due to relation \textbf{(18)}. 
\item \Define{\(c \of\)}\\
A dual argument can be made for \((c \oplus 1^{n-1}) \of S\) using relations \textbf{(15)},
\textbf{(11)} and \textbf{(16)}.
\item \Define{\(\of 0\)}\\
For \(S(0 \oplus 1^{m-1})\), relations \textbf{(8)} and \textbf{(16)} eradicate the first \(\Delta\)
cluster and all the multiplications incident to it, leaving behind \(n\) zeros.  Relation
\textbf{(1)} erases each of these zeros along with one addition per addition cluster, leaving a
diagram that is in standard form.
  \begin{center}
    \scalebox{0.80}
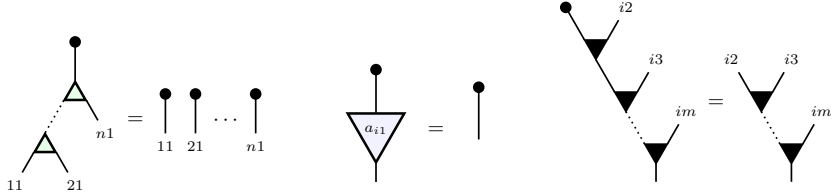
{
   \begin{tikzpicture}[thick]
   \node[zero] (top) {};
   \node[delta] (apex) at (0,-0.866) {};
   \node[delta] (dub) at (-0.5,-1.732) {};
   \node (11a) at (-1,-2.382) {\(\scriptstyle{11}\)};
   \node (21a) at (0,-2.382) {\(\scriptstyle{21}\)};
   \node (n1a) at (0.5,-1.516) {\(\scriptstyle{n1}\)};

   \draw[dotted] (apex.left out) -- (dub.io);
   \draw (top) -- (apex) (apex.right out) -- (n1a)
   (dub.left out) -- (11a) (dub.right out) -- (21a);

   \node (eq) at (1,-1.3) {\(=\)};
   \node[zero] (z1) at (1.5,-0.866) {};
   \node[zero] (z2) at (2,-0.866) {};
   \node[zero] (z3) at (3,-0.866) {};
   \node (dots) at (2.5,-1.3) {\(\cdots\)};
   \node (11b) at (1.5,-1.732) {\(\scriptstyle{11}\)};
   \node (21b) at (2,-1.732) {\(\scriptstyle{21}\)};
   \node (n1b) at (3,-1.732) {\(\scriptstyle{n1}\)};

   \draw (z1) -- (11b) (z2) -- (21b) (z3) -- (n1b);
   \end{tikzpicture}
       \hspace{1cm}
   \begin{tikzpicture}[thick]
   \node[zero] (zip) {};
   \node[multiply] (c) [below of=zip] {\(\scriptstyle{a_{i1}}\)};
   \node (out) [below of=c] {};
   \node (eq) [right of=c] {\(=\)};
   \node[zero] (nada) [above right of=eq] {};
   \node (bot) [below of=nada] {};

   \draw (zip) -- (c) -- (out) (nada) -- (bot);
   \end{tikzpicture}
       \hspace{1cm}
   \begin{tikzpicture}[thick]
   \node[zero] (zip) at (-0.5,0.65) {};
   \node[plus] (topa) {};
   \node[plus] (mida) at (0.5,-0.866) {};
   \node[plus] (bota) at (1,-1.732) {};
   \node (outa) at (1,-2.382) {};
   \node (i2a) at (0.5,0.65) {\(\scriptstyle{i2}\)};
   \node (i3a) at (1,-0.216) {\(\scriptstyle{i3}\)};
   \node (ima) at (1.5,-1.082) {\(\scriptstyle{im}\)};

   \draw[dotted] (mida.io) -- (bota.left in);
   \draw (bota) -- (outa) (bota.right in) -- (ima)
   (mida.right in) -- (i3a) (mida.left in) -- (topa.io)
   (topa.left in) -- (zip) (topa.right in) -- (i2a);

   \node (eq) at (2,-0.95) {\(=\)};
   \node[plus] (top) at (2.75,-0.866) {};
   \node[plus] (bot) at (3.25,-1.732) {};
   \node (out) at (3.25,-2.382) {};
   \node (i2) at (2.25,-0.216) {\(\scriptstyle{i2}\)};
   \node (i3) at (3.25,-0.216) {\(\scriptstyle{i3}\)};
   \node (im) at (3.75,-1.082) {\(\scriptstyle{im}\)};

   \draw[dotted] (top.io) -- (bot.left in);
   \draw (bot.right in) -- (im) (top.right in) -- (i3)
   (top.left in) -- (i2) (bot) -- (out);
   \end{tikzpicture}
    }
  \end{center}
When \(S = (! \oplus \dotsb \oplus !)\), the zero annihilates one of the deletions via relation
\textbf{(10)}.
\item \Define{\(! \of\)}\\
A dual argument erases the indicated output for the composition \((! \oplus 1^{n-1}) \of S\) using
relations \textbf{(9)}, \textbf{(18)}, and \textbf{(4)}.  Again, relation \textbf{(10)} annihilates
the deletion and one of the zeros if \(S = (0 \oplus \dotsb \oplus 0)\).
\item \Define{\(B \of\)}\\
Since this category of string diagrams is symmetric monoidal, an appended braiding will naturally
commute with the addition cluster morphisms.  The principle that only the topology matters means the
composition \((B \oplus 1^{n-2}) \of S\) is in standard form.  Braiding will similarly commute with
deletion morphisms.
  \begin{center}
    \scalebox{0.80}
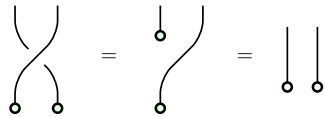
{
   \begin{tikzpicture}[thick,node distance=0.5cm]
   \node [coordinate] (fstart) {};
   \node [coordinate] (ftop) [below of=fstart] {};
   \node (center) [below right of=ftop] {};
   \node [coordinate] (fout) [below right of=center] {};
   \node [bang] (fend) [below of=fout] {};
   \node [coordinate] (gtop) [above right of=center] {};
   \node [coordinate] (gstart) [above of=gtop] {};
   \node [coordinate] (gout) [below left of=center] {};
   \node [bang] (gend) [below of=gout] {};

   \draw [rounded corners=0.25cm] (fstart) -- (ftop) -- (center) --
   (fout) -- (fend) (gstart) -- (gtop) -- (gout) -- (gend);

   \node (eq1) [below right of=gtop, shift={(0.5,0)}] {\(=\)};
   \node [bang] (ftop1) [above right of=eq1, shift={(0.5,0)}] {};
   \node [coordinate] (fstart1) [above of=ftop1] {};
   \node (center1) [below right of=ftop1] {};
   \node [coordinate] (gtop1) [above right of=center1] {};
   \node [coordinate] (gstart1) [above of=gtop1] {};
   \node [coordinate] (gout1) [below left of=center1] {};
   \node [bang] (gend1) [below of=gout1] {};

   \draw [rounded corners=0.25cm] (fstart1) -- (ftop1)
   (gstart1) -- (gtop1) -- (gout1) -- (gend1);

   \node (eq2) [below right of=gtop1, shift={(0.35,0)}] {\(=\)};
   \node [bang] (big) [right of=eq2, shift={(0.2,-0.5)}] {};
   \node [bang] (bang) [right of=big] {};

   \draw (big) -- +(0,1) (bang) -- +(0,1);
   \end{tikzpicture}
    }
  \end{center}
\item \Define{\(\of B\)}\\
Composing with \(B\) on the top, braiding commutes with duplication, multiplication and zero, so \(S
\of (B \oplus 1^{m-2})\) almost trivially comes into standard form.  
\end{itemize}
\end{proof}

An interesting exercise is to use these relations to derive a relation that expresses the braiding
in terms of other basic morphisms.  One example of such a relation appeared in
Section~\ref{sigflow}.  Here is another:
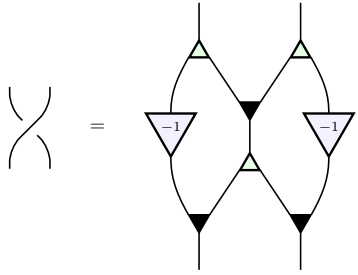
\begin{center}
 \scalebox{0.75}{
   \begin{tikzpicture}[-, thick]
   \node (Lin) {};
   \node [delta] (Ldub) [below of=Lin] {};
   \node [coordinate] (ULang) [below of=Ldub, shift={(-0.5,0)}] {};
   \node [multiply] (Lmin) [below of=ULang, shift={(0,0.65)}] {\(\scriptstyle{-1}\)};
   \node [coordinate] (BLang) [below of=Lmin, shift={(0,0.35)}] {};
   \node [plus] (Lsum) [below of=BLang, shift={(0.5,0)}] {};
   \node [plus] (Usum) [below of=Ldub, shift={(0.9,0)}] {};
   \node [delta] (Bdub) [below of=Usum] {};
   \node [delta] (Rdub) [above of=Usum, shift={(0.9,0)}] {};
   \node (Rin) [above of=Rdub] {};
   \node [coordinate] (URang) [below of=Rdub, shift={(0.5,0)}] {};
   \node [multiply] (Rmin) [below of=URang, shift={(0,0.65)}] {\(\scriptstyle{-1}\)};
   \node [coordinate] (BRang) [below of=Rmin, shift={(0,0.35)}] {};
   \node [plus] (Rsum) [below of=BRang, shift={(-0.5,0)}] {};
   \node (Lout) [below of=Lsum] {};
   \node (Rout) [below of=Rsum] {};

   \draw (Lin) -- (Ldub) (Ldub.right out) -- (Usum.left in) (Usum) --
   (Bdub) (Bdub.left out) -- (Lsum.right in) (Lsum) -- (Lout)
   (Rin) -- (Rdub) (Rdub.left out) -- (Usum.right in) (Bdub.right out)
   -- (Rsum.left in) (Rsum) -- (Rout);
   \draw (Ldub.left out) .. controls +(240:0.5) and +(90:0.3) .. (Lmin.90)
   (Lmin.io) .. controls +(270:0.3) and +(120:0.5) .. (Lsum.left in)
   (Rdub.right out) .. controls +(300:0.5) and +(90:0.3) .. (Rmin.90)
   (Rmin.270) .. controls +(270:0.3) and +(60:0.5) .. (Rsum.right in);

   \node (eq) [left of=Lmin, shift={(-0.3,0)}] {\(=\)};

   \node (center) [left of=eq, shift={(-0.2,0)}] {};
   \node [coordinate] (ftop) [above left of=center, shift={(0.35,-0.35)}] {};
   \node (fstart) [above of=ftop, shift={(0,-0.5)}] {};
   \node [coordinate] (fout) [below right of=center, shift={(-0.35,0.35)}] {};
   \node (fend) [below of=fout, shift={(0,0.5)}] {};
   \node [coordinate] (gtop) [above right of=center, shift={(-0.35,-0.35)}] {};
   \node (gstart) [above of=gtop, shift={(0,-0.5)}] {};
   \node [coordinate] (gout) [below left of=center, shift={(0.35,0.35)}] {};
   \node (gend) [below of=gout, shift={(0,0.5)}] {};

   \draw [rounded corners=7pt] (fstart) -- (ftop) -- (center) --
   (fout) -- (fend) (gstart) -- (gtop) -- (gout) -- (gend);
   \end{tikzpicture}
}
  \end{center}
With a few more relations, \(\Vectk\) can be presented as merely a monoidal category.  Lafont
\cite{Lafont} did this in the special case where \(k\) is the field with two elements.

\section{A presentation of \(\Relk\)}
\label{finrel}

Now we give a presentation for the symmetric monoidal category \(\Relk\).   As we did in the
previous section for \(\Vectk\), we work in a strict version of the symmetric monoidal category
\(\Relk\).

\begin{lemma} \label{gensrk} 
For any field \(k\), the object \(k\) together with the morphisms:
\begin{itemize}
\item addition \(+ \maps k \oplus k \asrelto k\)
\item zero \(0 \maps \{0\} \asrelto k\)
\item duplication \(\Delta \maps k \asrelto k \oplus k\)
\item deletion \(! \maps k \asrelto \{0\}\)
\item multiplication \(c \maps k \asrelto k\) for any \(c \in k\)
\item cup \(\cup \maps k \oplus k \asrelto \{0\} \)
\item cap \(\cap \maps \{0\} \asrelto k \oplus k \)
\end{itemize}
generate \(\Relk\), the category of finite-dimensional vector spaces over \(k\) and linear
relations, as a symmetric monoidal category.
\end{lemma}

\begin{proof}
A morphism of \(\Relk\), \(R\maps k^m \asrelto k^n\) is a subspace of \(k^m \oplus k^n \iso
k^{m+n}\).  It can be expressed as a system of \(k\)-linear equations in \(k^{m+n}\).
Lemma~\ref{gensvk} tells us any number of arbitrary \(k\)-linear combinations of the inputs may be
generated.  Any \(k\)-linear equation of those inputs can be formed by setting such a \(k\)-linear
combination equal to zero.  In particular, if caps are placed on each of the outputs to make them
inputs and all the \(k\)-linear combinations are set equal to zero, any \(k\)-linear system of
equations of the inputs and outputs can be formed.  Expressed in terms of string diagrams,
  
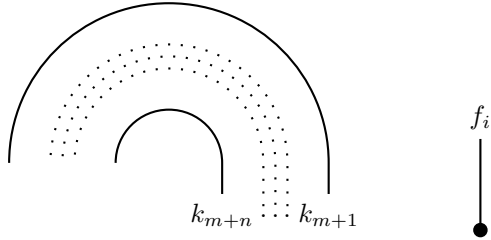
\begin{center}
   \begin{tikzpicture}[-, thick, node distance=0.708cm]
   \node (mn) {\(k_{m+n}\)};
   \node (dots) [right of=mn] {\(\dots\)};
   \node (m1) [right of=dots] {\(k_{m+1}\)};
   \node [coordinate] (inbend1) [above of=m1] {};
   \node [coordinate] (inbend2) [above of=dots] {};
   \node [coordinate] (inbend3) [above of=mn] {};
   \node (blank) [left of=inbend3] {};
   \node [coordinate] (midbend3) [above of=blank] {};
   \node [coordinate] (midbend2) [above of=midbend3] {};
   \node [coordinate] (midbend1) [above of=midbend2] {};
   \node [coordinate] (outbend3) [left of=blank] {};
   \node [coordinate] (outbend2) [left of=outbend3] {};
   \node [coordinate] (outbend1) [left of=outbend2] {};

   \node (dota) [right of=mn, shift={(0.16,0)}] {};
   \node [coordinate] (inbenda2) [above of=dota] {};
   \node [coordinate] (midbenda2) [above of=midbend3, shift={(0,0.16)}] {};
   \node [coordinate] (outbenda2) [left of=outbend3, shift={(-0.16,0)}] {};
   \node (dotd) [right of=mn, shift={(-0.16,0)}] {};
   \node [coordinate] (inbendd2) [above of=dotd] {};
   \node [coordinate] (midbendd2) [above of=midbend3, shift={(0,-0.16)}] {};
   \node [coordinate] (outbendd2) [left of=outbend3, shift={(0.16,0)}] {};

   \draw[loosely dotted,out=-45,in=-135,relative]
   (dota) -- (inbenda2) to (midbenda2) to (outbenda2)
   (dotd) -- (inbendd2) to (midbendd2) to (outbendd2)
   (dots) -- (inbend2) to (midbend2) to (outbend2);

   \draw[out=-45,in=-135,relative]
   (m1) -- (inbend1) to (midbend1) to (outbend1)
   (mn) -- (inbend3) to (midbend3) to (outbend3);
   \end{tikzpicture}
\hspace{1cm}
   \begin{tikzpicture}[-, thick, node distance=1cm]
   \node (f_i) at (0,1.5) {\(f_i\)};
   \node [zero] (zero) at (0,0) {};
   \draw (f_i) -- (zero);
   \end{tikzpicture}
  \end{center}
The left diagram turns the \(n\) outputs into inputs by placing caps on all of them.  The morphism
zero gives the \(k\)-linear combination zero, so an arbitrary \(k\)-linear combination in
\(k^{m+n}\) is set equal to zero (\(f_i=0\)) via the cozero morphism.  These elements can be
combined with Lemma~\ref{gensvk} to express any system of \(k\)-linear equations in \(k^{m+n}\).
\end{proof}
Putting these elements together, taking the \(\Vectk\) portion as a black box and drawing a single
string to denote zero or more copies of \(k\), the picture is fairly simple:
  
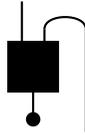
\begin{center}
   \begin{tikzpicture}[thick]
   \node [blackbox] (blackbox) {};
   \node [zero] (zilch) at (0,-0.7) {};
   \node (outs) at (0.7,-1) {};
   \node [coordinate] (capR) at (0.7,0.5) {};
   \node [coordinate] (capL) at (0.15,0.5) {};
   \node (ins) at (-0.15,1) {};

   \path (capL) edge[bend left=90] (capR);
   \draw (ins) -- (-0.15,0) (blackbox) -- (zilch) (capL) -- (0.15,0)
   (capR) -- (outs);
   \end{tikzpicture}
  \end{center}

To obtain a presentation of \(\Relk\) as a symmetric monoidal category, we need to find enough
relations obeyed by the generating morphisms listed in Lemma~\ref{gensrk}.  Relations \textbf{(1)--(18)} from
Theorem~\ref{presvk} still apply, but we need more.  

For convenience, in the list below we draw the adjoint of any generating morphism by rotating it by
\(180^\circ\).  It will follow from relations (19) and (20) that the cap is the adjoint of the cup,
so this convenient trick is consistent even in that case, where \emph{a priori} there might have
been an ambiguity.

\vskip 1em \noindent
\textbf{(19)--(20)} \(\cap\) and \(\cup\) obey the zigzag relations, and thus give a
\(\dagger\)-compact category:

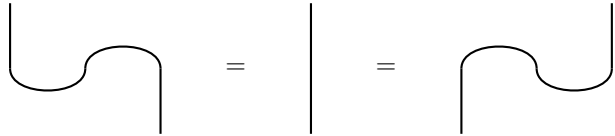
\begin{center}
   \begin{tikzpicture}[-, thick, node distance=1cm]
   \node (zigtop) {};
   \node [coordinate] (zigincup) [below of=zigtop] {};
   \node [coordinate] (zigcupcap) [right of=zigincup] {};
   \node [coordinate] (zigoutcap) [right of=zigcupcap] {};
   \node (zigbot) [below of=zigoutcap] {};
   \node (equal) [right of=zigoutcap] {\(=\)};
   \node (mid) [right of=equal] {};
   \node (vtop) [above of=mid] {};
   \node (vbot) [below of=mid] {};
   \node (equals) [right of=mid] {\(=\)};
   \node [coordinate] (zagoutcap) [right of=equals] {};
   \node (zagbot) [below of=zagoutcap] {};
   \node [coordinate] (zagcupcap) [right of=zagoutcap] {};
   \node [coordinate] (zagincup) [right of=zagcupcap] {};
   \node (zagtop) [above of=zagincup] {};
   \path
   (zigincup) edge (zigtop) edge [bend right=90] (zigcupcap)
   (zigoutcap) edge (zigbot) edge [bend right=90] (zigcupcap)
   (vtop) edge (vbot)
   (zagincup) edge (zagtop) edge [bend left=90] (zagcupcap)
   (zagoutcap) edge (zagbot) edge [bend left=90] (zagcupcap);
   \end{tikzpicture}
    \end{center}
   
\vskip 1em \noindent
\textbf{(21)--(22)} \( (k, +, 0, +^\dagger, 0^\dagger)\) is a Frobenius monoid:

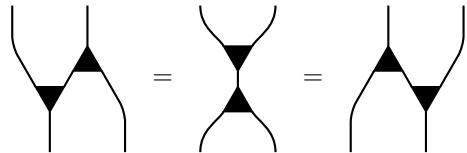
\begin{center}
 \scalebox{1}{
   \begin{tikzpicture}[thick]
   \node [plus] (sum1) at (0.5,-0.216) {};
   \node [coplus] (cosum1) at (1,0.216) {};
   \node [coordinate] (sum1corner) at (0,0.434) {};
   \node [coordinate] (cosum1corner) at (1.5,-0.434) {};
   \node [coordinate] (sum1out) at (0.5,-0.975) {};
   \node [coordinate] (cosum1in) at (1,0.975) {};
   \node [coordinate] (1cornerin) at (0,0.975) {};
   \node [coordinate] (1cornerout) at (1.5,-0.975) {};

   \draw[rounded corners] (1cornerin) -- (sum1corner) -- (sum1.left in)
   (1cornerout) -- (cosum1corner) -- (cosum1.right out);
   \draw (sum1.right in) -- (cosum1.left out)
   (sum1.io) -- (sum1out)
   (cosum1.io) -- (cosum1in);

   \node (eq1) at (2,0) {\(=\)};
   \node [plus] (sum2) at (3,0.325) {};
   \node [coplus] (cosum2) at (3,-0.325) {};
   \node [coordinate] (sum2inleft) at (2.5,0.975) {};
   \node [coordinate] (sum2inright) at (3.5,0.975) {};
   \node [coordinate] (cosum2outleft) at (2.5,-0.975) {};
   \node [coordinate] (cosum2outright) at (3.5,-0.975) {};

   \draw (sum2inleft) .. controls +(270:0.3) and +(120:0.15) .. (sum2.left in)
   (sum2inright) .. controls +(270:0.3) and +(60:0.15) .. (sum2.right in)
   (cosum2outleft) .. controls +(90:0.3) and +(240:0.15) .. (cosum2.left out)
   (cosum2outright) .. controls +(90:0.3) and +(300:0.15) .. (cosum2.right out)
   (sum2.io) -- (cosum2.io);

   \node (eq2) at (4,0) {\(=\)};
   \node [plus] (sum3) at (5.5,-0.216) {};
   \node [coplus] (cosum3) at (5,0.216) {};
   \node [coordinate] (sum3corner) at (6,0.434) {};
   \node [coordinate] (cosum3corner) at (4.5,-0.434) {};
   \node [coordinate] (sum3out) at (5.5,-0.975) {};
   \node [coordinate] (cosum3in) at (5,0.975) {};
   \node [coordinate] (3cornerin) at (6,0.975) {};
   \node [coordinate] (3cornerout) at (4.5,-0.975) {};

   \draw[rounded corners] (3cornerin) -- (sum3corner) -- (sum3.right in)
   (3cornerout) -- (cosum3corner) -- (cosum3.left out);
   \draw (sum3.left in) -- (cosum3.right out)
   (sum3.io) -- (sum3out)
   (cosum3.io) -- (cosum3in);
   \end{tikzpicture}
}
\end{center}

\vskip 1em \noindent
\textbf{(23)--(24)} \( (k, \Delta^\dagger, !^\dagger, \Delta, !) \) is a Frobenius monoid:

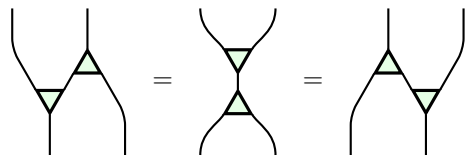
\begin{center}
 \scalebox{1}{
\begin{tikzpicture}[thick]
   \node [codelta] (sum1) at (0.5,-0.216) {};
   \node [delta] (cosum1) at (1,0.216) {};
   \node [coordinate] (sum1corner) at (0,0.434) {};
   \node [coordinate] (cosum1corner) at (1.5,-0.434) {};
   \node [coordinate] (sum1out) at (0.5,-0.975) {};
   \node [coordinate] (cosum1in) at (1,0.975) {};
   \node [coordinate] (1cornerin) at (0,0.975) {};
   \node [coordinate] (1cornerout) at (1.5,-0.975) {};

   \draw[rounded corners] (1cornerin) -- (sum1corner) -- (sum1.left in)
   (1cornerout) -- (cosum1corner) -- (cosum1.right out);
   \draw (sum1.right in) -- (cosum1.left out)
   (sum1.io) -- (sum1out)
   (cosum1.io) -- (cosum1in);

   \node (eq1) at (2,0) {\(=\)};
   \node [codelta] (sum2) at (3,0.325) {};
   \node [delta] (cosum2) at (3,-0.325) {};
   \node [coordinate] (sum2inleft) at (2.5,0.975) {};
   \node [coordinate] (sum2inright) at (3.5,0.975) {};
   \node [coordinate] (cosum2outleft) at (2.5,-0.975) {};
   \node [coordinate] (cosum2outright) at (3.5,-0.975) {};

   \draw (sum2inleft) .. controls +(270:0.3) and +(120:0.15) .. (sum2.left in)
   (sum2inright) .. controls +(270:0.3) and +(60:0.15) .. (sum2.right in)
   (cosum2outleft) .. controls +(90:0.3) and +(240:0.15) .. (cosum2.left out)
   (cosum2outright) .. controls +(90:0.3) and +(300:0.15) .. (cosum2.right out)
   (sum2.io) -- (cosum2.io);

   \node (eq2) at (4,0) {\(=\)};
   \node [codelta] (sum3) at (5.5,-0.216) {};
   \node [delta] (cosum3) at (5,0.216) {};
   \node [coordinate] (sum3corner) at (6,0.434) {};
   \node [coordinate] (cosum3corner) at (4.5,-0.434) {};
   \node [coordinate] (sum3out) at (5.5,-0.975) {};
   \node [coordinate] (cosum3in) at (5,0.975) {};
   \node [coordinate] (3cornerin) at (6,0.975) {};
   \node [coordinate] (3cornerout) at (4.5,-0.975) {};

   \draw[rounded corners] (3cornerin) -- (sum3corner) -- (sum3.right in)
   (3cornerout) -- (cosum3corner) -- (cosum3.left out);
   \draw (sum3.left in) -- (cosum3.right out)
   (sum3.io) -- (sum3out)
   (cosum3.io) -- (cosum3in);
\end{tikzpicture}
}
\end{center}

\vskip 1em \noindent
\textbf{(25)--(26)} The Frobenius monoid \( (k, +, 0, +^\dagger, 0^\dagger)\) is extra-special:

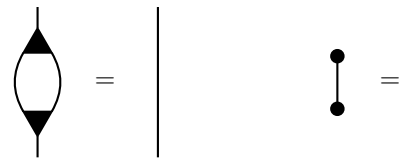
\begin{center}
\begin{tikzpicture}[thick]
   \node [plus] (sum) at (0.4,-0.5) {};
   \node [coplus] (cosum) at (0.4,0.5) {};
   \node [coordinate] (in) at (0.4,1) {};
   \node [coordinate] (out) at (0.4,-1) {};
   \node (eq) at (1.3,0) {\(=\)};
   \node [coordinate] (top) at (2,1) {};
   \node [coordinate] (bottom) at (2,-1) {};

   \path (sum.left in) edge[bend left=30] (cosum.left out)
   (sum.right in) edge[bend right=30] (cosum.right out);
   \draw (top) -- (bottom)
   (sum.io) -- (out)
   (cosum.io) -- (in);
\end{tikzpicture}
\qquad 
\qquad
\qquad
\begin{tikzpicture}[thick]
   \node [zero] (Bins) at (0,-0.35) {};
   \node [zero] (Tins) at (0,0.35) {};
   \node (eq) at (0.7,0) {\(=\)};
   \node [hole] at (0,-0.865) {};
   \draw (Tins) -- (Bins);
   \end{tikzpicture}
\end{center}

\vskip 1em \noindent
\textbf{(27)--(28)} The Frobenius monoid \( (k, \Delta^\dagger, !^\dagger, \Delta, !)\) is
extra-special:

\begin{center}
\begin{tikzpicture}[thick]
   \node [codelta] (sum) at (0.4,-0.5) {};
   \node [delta] (cosum) at (0.4,0.5) {};
   \node [coordinate] (in) at (0.4,1) {};
   \node [coordinate] (out) at (0.4,-1) {};
   \node (eq) at (1.3,0) {\(=\)};
   \node [coordinate] (top) at (2,1) {};
   \node [coordinate] (bottom) at (2,-1) {};

   \path (sum.left in) edge[bend left=30] (cosum.left out)
   (sum.right in) edge[bend right=30] (cosum.right out);
   \draw (top) -- (bottom)
   (sum.io) -- (out)
   (cosum.io) -- (in);
\end{tikzpicture}
\qquad 
\qquad
\qquad
\begin{tikzpicture}[thick]
   \node [bang] (Bins) at (0,-0.35) {};
   \node [bang] (Tins) at (0,0.35) {};
   \node (eq) at (0.7,0) {\(=\)};
   \node [hole] at (0,-0.865) {};
   \draw (Tins) -- (Bins);
   \end{tikzpicture}
\end{center}

\vskip 1em \noindent
\textbf{(29)} \(\cup\) with a factor of \(-1\) inserted can be expressed in terms of \(+\) and
\(0\):

\begin{center}
\scalebox{1}{
   \begin{tikzpicture}[thick]
   \node [multiply] (neg) at (0,0.1) {\(\scriptstyle{-1}\)};
   \node [coordinate] (cupInLeft) at (0,1) {};
   \node [coordinate] (Lcup) at (0,-0.5) {};
   \node [coordinate] (Rcup) at (1,-0.5) {};
   \node [coordinate] (cupInRight) at (1,1) {};
   \node (eq) at (1.7,0.1) {\(=\)};
   \node [coordinate] (SumLeftIn) at (2.25,1) {};
   \node [coordinate] (SumRightIn) at (3.25,1) {};
   \node [plus] (Sum) at (2.75,0) {};
   \node [zero] (coZero) at (2.75,-0.65) {};

   \draw (SumRightIn) .. controls +(270:0.5) and +(60:0.5) .. (Sum.right in)
      (SumLeftIn) .. controls +(270:0.5) and +(120:0.5) .. (Sum.left in);
   \draw (cupInLeft) -- (neg) -- (Lcup)
      (Rcup) -- (cupInRight)
      (Sum) -- (coZero);
   \path (Lcup) edge[bend right=90] (Rcup);
   \end{tikzpicture}
}
\end{center}

\vskip 1em \noindent
\textbf{(30)} \(\cap\) can be expressed in terms of \(\Delta\) and \(!\):

\begin{center}
\scalebox{1}{
   \begin{tikzpicture}[thick]
   \node (eq) at (0.2,-0.1) {\(=\)};
   \node [coordinate] (lcap) at (-1.5,0.5) {};
   \node [coordinate] (rcap) at (-0.5,0.5) {};
   \node [coordinate] (lcapbot) at (-1.5,-1) {};
   \node [coordinate] (rcapbot) at (-0.5,-1) {};
   \node [delta] (dub) at (1.25,0) {};
   \node [bang] (boom) at (1.25,0.65) {};
   \node [coordinate] (Leftout) at (0.75,-1) {};
   \node [coordinate] (Rightout) at (1.75,-1) {};

   \draw (dub.left out) .. controls +(240:0.5) and +(90:0.5) .. (Leftout)
      (dub.right out) .. controls +(300:0.5) and +(90:0.5) .. (Rightout);
   \draw (boom) -- (dub) (lcapbot) -- (lcap) (rcap) -- (rcapbot);
   \path (lcap) edge[bend left=90] (rcap);
   \end{tikzpicture}
}
\end{center}

\vskip 1em \noindent
\textbf{(31)} For any \(c \in k\) with \(c \ne 0\), scalar multiplication by \(c^{-1}\) is the
adjoint of scalar multiplication by \(c\):

\begin{center}
   \begin{tikzpicture}[thick]
   \node[upmultiply] (c) {\(c\)};
   \node[coordinate] (in1) [above of=c] {};
   \node[coordinate] (out1) [below of=c] {};

   \draw (in1) -- (c) -- (out1);

   \node (eq) [right of=c] {\(=\)};
   \node[multiply] (mult) [right of=eq, shift={(0.5,0)}] {\(c^{-1}\!\!\)};
   \node[coordinate] (in) [above of=mult] {};
   \node[coordinate] (out) [below of=mult] {};

   \draw (in) -- (mult) -- (out);
   \end{tikzpicture}
  \end{center}

\vskip 1em
Some curious identities can be derived from relations \textbf{(1)--(31)}, beyond those already
arising from \textbf{(1)--(18)}.  For example:

\vskip 1em \noindent
\textbf{(D1)--(D2)} Deletion and zero can be expressed in terms of other generating morphisms:
  \begin{center}
   \begin{tikzpicture}[thick]
   \node [bang] (bang) at (0.2,0) {};
   \node (eq1) at (1,1.15) {\(=\)};
   \node (rel1) at (1,0.8) {(27)};
   \node [bang] (buck) at (2,0) {};
   \node [codelta] (codub) at (2,0.65) {};
   \node [delta] (dub) at (2,1.65) {};
   \node (eq2) at (3,1.15) {\(=\)};
   \node (rel2) at (3,0.8) {(30)\({}^\dagger\)};
   \node [delta] (dupe) at (4,1) {};

   \path
   (codub.left in) edge[bend left=30] (dub.left out)
   (codub.right in) edge[bend right=30] (dub.right out);
   \draw
   (bang) -- +(up:2.3) (dub.io) -- (2,2.3) (dupe.io) -- (4,2.3)
   (codub.io) -- (buck)
   (dupe.left out) .. controls +(240:1.2) and +(300:1.2) .. (dupe.right out);
   \end{tikzpicture}
\hskip 2em
   \begin{tikzpicture}[thick]
   \node [zero] (Zero1) at (0,1) {};
   \node (eq1) at (0.8,-0.15) {\(=\)};
   \node (rel1) at (0.8,-0.5) {(28)};
   \node [bang] (cobang) at (1.6,1) {};
   \node [bang] (bang) at (1.6,0.2) {};
   \node [zero] (Zero2) at (1.6,-0.5) {};
   \node (eq2) at (2.4,-0.15) {\(=\)};
   \node (rel2) at (2.4,-0.5) {(14)};
   \node [multiply] (times) at (3.4,-0.15) {\(0\)};
   \node [bang] (cobuck) at (3.4,1) {};
   \node (eq3) at (4.4,-0.15) {\(=\)};
   \node (rel3) at (4.4,-0.5) {(D1)\({}^\dagger\)};
   \node [multiply] (oh) at (5.4,-0.5) {\(0\)};
   \node [codelta] (cod) at (5.4,0.2) {};

   \draw
   (Zero1) -- +(down:2.3) (cobang) -- (bang) (Zero2) -- +(down:0.8)
   (cobuck) -- (times) -- (3.4,-1.3) (cod.io) -- (oh) -- (5.4,-1.3)
   (cod.left in) .. controls +(120:1.2) and +(60:1.2) .. (cod.right in);
   \end{tikzpicture}
  \end{center}
This does not diminish the role of deletion and zero.  Indeed, regarding these generating morphisms
as superfluous buries some of the structure of \(\Relk\).

\vskip 1em \noindent
\textbf{(D3)} Addition can be expressed in terms of coaddition and scalar multiplication by \(-1\),
and the cup:
  \begin{center}
    \scalebox{1}
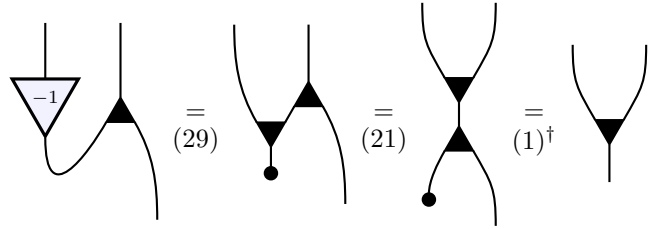
{
   \begin{tikzpicture}[thick]
   \node [multiply] (neg1) at (0,0) {\(\scriptstyle{-1}\)};
   \node [coplus] (cosum1) at (1,-0.21) {};

   \draw
   (neg1) -- (0,1) (cosum1.io) -- (1,1)
   (cosum1.right out) .. controls +(300:0.5) and +(90:0.5) .. +(0.3,-1.3)
   (neg1.io) .. controls +(270:1) and +(240:0.5) .. (cosum1.left out);
   \node (eq2) at (2,-0.216) {\(=\)};
   \node (rel2) at (2,-0.566) {(29)};
   \node [plus] (sum3) at (3,-0.432) {};
   \node [zero] (coz3) at (3,-1) {};
   \node [coplus] (cosum3) at (3.5,0) {};

   \draw
   (sum3.left in) .. controls +(120:0.5) and +(270:0.5) .. +(-0.3,1.3)
   (sum3) -- (coz3) (sum3.right in) -- (cosum3.left out)
   (cosum3.right out) .. controls +(300:0.5) and +(90:0.5) .. +(0.3,-1.3)
   (cosum3.io) -- +(0,0.76);
   \node (eq3) at (4.5,-0.216) {\(=\)};
   \node (rel3) at (4.5,-0.566) {(21)};
   \node [plus] (sum4) at (5.5,0.159) {};
   \node [coplus] (cosum4) at (5.5,-0.591) {};
   \node [zero] (coz4) at (5.1,-1.35) {};

   \draw
   (sum4.io) -- (cosum4.io)
   (cosum4.left out) .. controls +(240:0.2) and +(90:0.25) .. (coz4)
   (sum4.left in) .. controls +(120:0.5) and +(270:0.5) .. +(-0.3,1)
   (sum4.right in) .. controls +(60:0.5) and +(270:0.5) .. +(0.3,1)
   (cosum4.right out) .. controls +(300:0.5) and +(90:0.5) .. +(0.3,-1);
   \node (eq4) at (6.5,-0.216) {\(=\)};
   \node (rel4) at (6.5,-0.566) {(1)\({}^\dagger\)};
   \node [plus] (sum5) at (7.5,-0.4) {};

   \draw
   (sum5.io) -- +(0,-0.5)
   (sum5.left in) .. controls +(120:0.5) and +(270:0.5) .. +(-0.3,1)
   (sum5.right in) .. controls +(60:0.5) and +(270:0.5) .. +(0.3,1);
   \end{tikzpicture}
    }
  \end{center}

\vskip 1em \noindent
\textbf{(D4)} Duplication can be expressed in terms of coduplication and the cap:
  \begin{center}
    \scalebox{1}
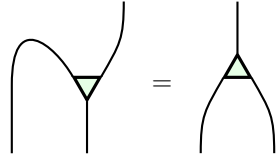
{
   \begin{tikzpicture}[thick]
   \node [coordinate] (neg1) at (0,0) {};
   \node [codelta] (codub) at (1,0) {};

   \draw
   (neg1) -- (0,-0.91) (codub.io) -- (1,-0.91)
   (codub.right in) .. controls +(60:0.5) and +(270:0.5) .. +(0.3,1)
   (neg1) .. controls +(90:1) and +(120:0.5) .. (codub.left in);

   \node (eq) at (2,0) {\(=\)};
   \node [delta] (dub) at (3,0.2) {};

   \draw
   (dub.io) -- +(0,0.69)
   (dub.left out) .. controls +(240:0.5) and +(90:0.5) .. +(-0.3,-1)
   (dub.right out) .. controls +(300:0.5) and +(90:0.5) .. +(0.3,-1);
   \end{tikzpicture}
    }
  \end{center}
where the proof is similar to that of \textbf{(D3)}.

\vskip 1em \noindent
\textbf{(D5)--(D7)} We can reformulate the bimonoid relations (7)--(9) using adjoints:
  \begin{center}
    \scalebox{1}
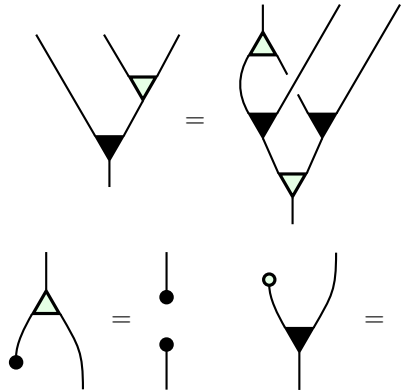
{
   \begin{tikzpicture}[thick, node distance=0.7cm]
   \node (in1z) {};
   \node (in2z) [right of=in1z, shift={(0.2,0)}] {};
   \node (in3z) [right of=in2z, shift={(0.45,0)}] {};
   \node [codelta] (nabzip) [below right of=in2z, shift={(0.1,-0.3)}] {};
   \node [plus] (add) [below left of=nabzip, shift={(0.05,-0.3)}] {};
   \node (outz) [below of=add] {};
   \node (equal) [below right of=nabzip, shift={(0.2,0)}] {\(=\)};
   \node [plus] (addl) [right of=equal, shift={(0.2,0)}] {};
   \node (cross) [above right of=addl, shift={(-0.1,0)}] {};
   \node [delta] (delta) [above left of=cross, shift={(0.1,0)}] {};
   \node (in1u) [above of=delta] {};
   \node [plus] (addr) [below right of=cross, shift={(-0.1,0)}] {};
   \node [codelta] (nablunzip) [below left of=addr, shift={(0.1,-0.3)}] {};
   \node (outu) [below of=nablunzip] {};
   \node (in2u) [right of=in1u, shift={(0.4,0)}] {};
   \node (in3u) [right of=in2u, shift={(0.1,0)}] {};

   \draw (in1z) -- (add.left in) (add) -- (outz) (in2z) --
   (nabzip.left in) (in3z) -- (nabzip.right in) (nabzip.io) -- (add.right in);
   \path
   (delta.left out) edge [bend right=30] (addl.left in);
   \draw (in1u) -- (delta) (delta.right out) -- (cross) -- (addr.left in);
   \draw (in2u) -- (addl.right in) (in3u) -- (addr.right in);
   \draw (addl.io) -- (nablunzip.left in) (addr.io) -- (nablunzip.right in)
   (nablunzip) -- (outu);
   \end{tikzpicture}
}
\end{center}
\begin{center}
\scalebox{1}{
   \begin{tikzpicture}[thick]
   \node [zero] (coz1) at (0,0) {};
   \node [delta] (dub) at (0.4,0.75) {};

   \draw
   (coz1) .. controls +(90:0.25) and +(240:0.2) .. (dub.left out)
   (dub.right out) .. controls +(300:0.5) and +(90:0.5) .. +(0.3,-1)
   (dub.io) -- +(0,0.5);

   \node (eq) at (1.4,0.552) {\(=\)};
   \node [zero] (coz2) at (2,0.867) {};
   \node [zero] (zero) at (2,0.237) {};

   \draw (coz2) -- +(0,0.6) (zero) -- +(0,-0.6);
   \end{tikzpicture}
\hskip 3em
   \begin{tikzpicture}[thick]
   \node [bang] (cobang1) at (0,0) {};
   \node [plus] (add) at (0.4,-0.75) {};

   \draw
   (cobang1) .. controls +(270:0.25) and +(120:0.2) .. (add.left in)
   (add.right in) .. controls +(60:0.5) and +(270:0.5) .. +(0.3,1)
   (add.io) -- +(0,-0.5);

   \node (eq) at (1.4,-0.552) {\(=\)};
   \node [bang] (cobang2) at (2,-0.867) {};
   \node [bang] (bang) at (2,-0.237) {};

   \draw (bang) -- +(0,0.6) (cobang2) -- +(0,-0.6);
   \end{tikzpicture}
    }
  \end{center}

\vskip 1em \noindent
\textbf{(D8)--(D9)} When \(c \ne 1\), we have:
  
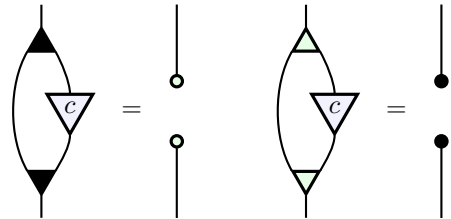
\begin{center}
   \begin{tikzpicture}[thick]
   \node [coplus] (cosum) at (0,1.8) {};
   \node [multiply] (times) at (0.38,0.95) {\(c\)};
   \node [plus] (sum) at (0,0) {};

   \draw
   (cosum.io) -- +(0,0.3) (sum.io) -- +(0,-0.3)
   (cosum.left out) .. controls +(240:0.5) and +(120:0.5) .. (sum.left in)
   (cosum.right out) .. controls +(300:0.15) and +(90:0.15) .. (times.90)
   (times.io) .. controls +(270:0.15) and +(60:0.15) .. (sum.right in);

   \node (eq) at (1.2,0.9) {\(=\)};
   \node [bang] (bang) at (1.8,1.3) {};
   \node [bang] (cobang) at (1.8,0.5) {};

   \draw (bang) -- +(0,1.02) (cobang) -- +(0,-1.02);
   \end{tikzpicture}
\hskip 3em
   \begin{tikzpicture}[thick]
   \node [delta] (dub) at (0,1.8) {};
   \node [multiply] (times) at (0.38,0.95) {\(c\)};
   \node [codelta] (codub) at (0,0) {};

   \draw
   (dub.io) -- +(0,0.3) (sum.io) -- +(0,-0.3)
   (dub.left out) .. controls +(240:0.5) and +(120:0.5) .. (codub.left in)
   (dub.right out) .. controls +(300:0.15) and +(90:0.15) .. (times.90)
   (times.io) .. controls +(270:0.15) and +(60:0.15) .. (codub.right in);

   \node (eq) at (1.2,0.9) {\(=\)};
   \node [zero] (coz) at (1.8,1.3) {};
   \node [zero] (zero) at (1.8,0.5) {};

   \draw (coz) -- +(0,1.02) (zero) -- +(0,-1.02);
   \end{tikzpicture}
  \end{center}
\noindent
We leave the derivation of \textbf{(D5)--(D9)} as exercises for the reader.

Next we show that relations \textbf{(1)--(31)} are enough to give a presentation of \(\Relk\) as a symmetric
monoidal category.  As before, we do this by giving a standard form that any morphism can be written
in and use induction to show that an arbitrary diagram can be rewritten in its standard form using
the given relations.

\begin{thm} \label{presrk} The symmetric monoidal category \(\Relk\) is presented by the object
\(k\), the morphisms given in Lemma~\ref{gensrk}, and relations {\bf (1)--(31)} as listed above.
\end{thm}

\begin{proof}
We prove this theorem by using the relations \textbf{(1)--(31)} to put any string diagram built from
the generating morphisms and braiding into a standard form, so that any two string diagrams
corresponding to the same morphism in \(\Relk\) have the same standard form.

As before, we induct on the number of \Define{basic morphisms} involved in a string diagram, where
the basic morphisms are the generating morphisms together with the braiding.  If we let \(R \maps
k^m \asrelto k^n\) be a morphism in \(\Relk\), we can build a string diagram \(S\) for \(R\) as in
Lemma~\ref{gensrk}.  Each output of \(S\) is capped, and, together with the inputs of \(S\), form
inputs for a \(\Vectk\) block, \(T\).  For some \(r \leq m+n\), there are \(r\) outputs of
\(T\)--linear combinations of the \(m+n\) inputs--each set equal to zero via \((0^\dagger)^r\).
When \(T\) is in standard form for \(\Vectk\), we say \(S\) is in \Define{prestandard form}, and can
be depicted as follows:
  \begin{center}
    \scalebox{1}
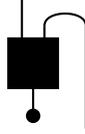
{
   \begin{tikzpicture}[thick]
   \node [blackbox] (blackbox) {};
   \node [zero] (zilch) at (0,-0.7) {};
   \node (outs) at (0.7,-1) {};
   \node [coordinate] (capR) at (0.7,0.5) {};
   \node [coordinate] (capL) at (0.15,0.5) {};
   \node (ins) at (-0.15,1) {};

   \path (capL) edge[bend left=90] (capR);
   \draw (ins) -- (-0.15,0) (blackbox) -- (zilch) (capL) -- (0.15,0)
   (capR) -- (outs);
   \end{tikzpicture}
    }
  \end{center}
While the linear subspace of \(k^{m+n}\) defined by \(R\) is determined by a system of \(r\) linear
equations, the converse is not true, meaning there may be multiple prestandard string diagrams for a
single morphism \(R\).  The second stage of this proof collapses all the prestandard forms into a
standard form using some basic linear algebra.  The standard form will correspond to when the matrix
representation of \(T\) is written in row-reduced echelon form.  For this stage it will suffice to
show all the elementary row operations correspond to relations that hold between diagrams.  By
Theorem \ref{presvk}, an arbitrary \(\Vectk\) block can be rewritten in its standard form, so the
\(\Vectk\) blocks here need not be demonstrated in their standard form.
\\
When there is one basic morphism, there are eight cases to consider, one per basic morphism.  In
each of these basic cases, the block of the diagram equivalent to a morphism in \(\Vectk\) is
denoted by a dashed rectangle.  We first consider \(\cup\).

\vskip 1em \noindent
\textbf{(D10)}
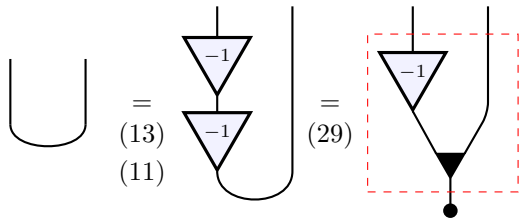
  \begin{center}
    \scalebox{1}{
   \begin{tikzpicture}[thick]
   \node (Lin0) at (0,1) {};
   \node [coordinate] (Lcup) at (0,0) {};
   \node [coordinate] (Rcup) at (1,0) {};
   \node (Rin0) at (1,1) {};
   \node (eq1) at (1.75,0.25) {\(=\)};
   \node (rel1) at (1.75,-0.15) {(13)};
   \node (rel2) at (1.75,-0.65) {(11)};
   \node (Lin1) at (2.75,1.7) {};
   \node [multiply] (neg1a) at (2.75,0.9) {\(\scriptstyle{-1}\)};
   \node [multiply] (neg1b) at (2.75,-0.1) {\(\scriptstyle{-1}\)};
   \node (Rin1) at (3.75,1.7) {};
   \node (eq2) at (4.25,0.25) {\(=\)};
   \node (rel3) at (4.25,-0.15) {(29)};
   \node (Lin4) at (5.35,1.7) {};
   \node [multiply] (neg) at (5.35,0.7) {\(\scriptstyle{-1}\)};
   \node (Rin4) at (6.35,1.7) {};
   \node [coordinate] (Rcor4) at (6.35,0.15) {};
   \node [plus] (S4) at (5.85,-0.5) {};
   \node [zero] (coZ4) at (5.85,-1.15) {};

   \draw[rounded corners]
   (Rin4) -- (Rcor4) -- (S4.right in);
   \draw (Lin0) -- (Lcup) (Rcup) -- (Rin0)
   (Lin1) -- (neg1a) -- (neg1b)
   (neg1b.io) .. controls +(270:0.5) and +(270:0.5) .. +(1,0) -- (Rin1)
   (Lin4) -- (neg) (neg.io) -- (S4.left in) (S4) -- (coZ4);
   \draw [color=red, dashed, thin] (4.75,1.2) rectangle (6.75,-0.9);
   \path (Lcup) edge[bend right=90] (Rcup);
   \end{tikzpicture}
    }
  \end{center}
Capping each of the inputs turns this into the standard form of \(\cap\).  Aside from deletion, the
remaining generating morphisms can be formed by introducing a zigzag at each output and rewriting
the resulting cups as above.  The standard forms for \(0\) and \(!\) have simpler expressions.
  \begin{center}
    \scalebox{0.80}
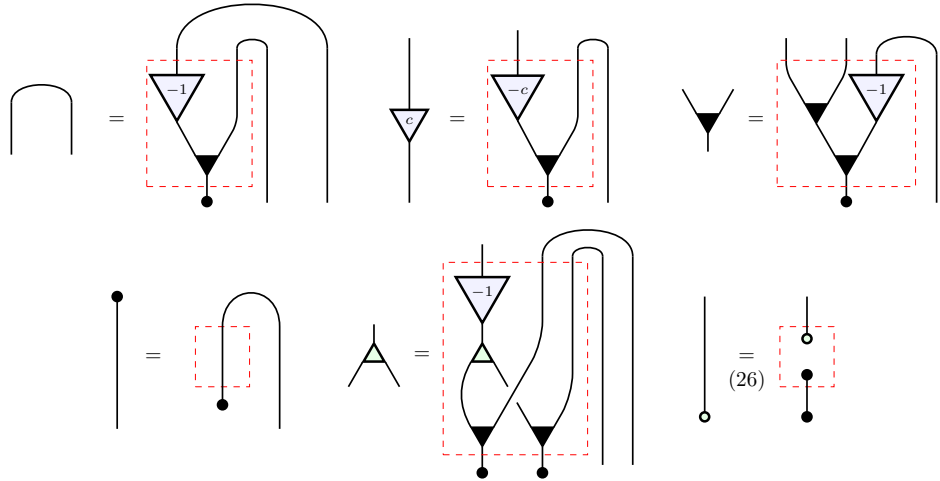
{
   \begin{tikzpicture}[thick]
   \node (Lout0) at (-2.75,-0.5) {};
   \node [coordinate] (Lcap0) at (-2.75,0.5) {};
   \node [coordinate] (Rcap0) at (-1.75,0.5) {};
   \node (Rout0) at (-1.75,-0.5) {};
   \node (eq) at (-1,0.2) {\(=\)};
   \node [coordinate] (Lin) at (0,1.4) {};
   \node [multiply] (neg) at (0,0.7) {\(\scriptstyle{-1}\)};
   \node [coordinate] (Rin) at (1,1.4) {};
   \node [coordinate] (Rcor) at (1,0.15) {};
   \node [plus] (S) at (0.5,-0.5) {};
   \node [zero] (coZ) at (0.5,-1.15) {};
   \node [coordinate] (Rcap) at (2.5,1.4) {};
   \node [coordinate] (Lcap) at (1.5,1.4) {};
   \node (Rout) at (2.5,-1.3) {};
   \node (Lout) at (1.5,-1.3) {};

   \draw (Lout0) -- (Lcap0) (Rcap0) -- (Rout0) (Lin) -- (neg) (S) -- (coZ)
   (neg.io) -- (S.left in) (Rcap) -- (Rout) (Lcap) -- (Lout);
   \path (Lcap0) edge[bend left=90] (Rcap0) (Lin) edge[bend left=90]
   (Rcap) (Rin) edge[bend left=90] (Lcap);
   \draw[rounded corners] (Rin) -- (Rcor) -- (S.right in);
   \draw[color=red, dashed, thin] (-0.5,1.2) rectangle (1.25,-0.9);
   \end{tikzpicture}
       \hspace{0.69cm}
   \begin{tikzpicture}[thick]
   \node (in) at (-1.8,1.7) {};
   \node [multiply] (times) at (-1.8,0.2) {\(\scriptstyle{c}\)};
   \node (out) at (-1.8,-1.3) {};
   \node (eq) at (-1,0.2) {\(=\)};
   \node [coordinate] (Lin) at (0,1.7) {};
   \node [multiply] (neg) at (0,0.7) {\(\scriptstyle{-c}\)};
   \node [coordinate] (Rin) at (1,1.4) {};
   \node [coordinate] (Rcor) at (1,0.15) {};
   \node [plus] (S) at (0.5,-0.5) {};
   \node [zero] (coZ) at (0.5,-1.15) {};
   \node [coordinate] (Lcap) at (1.5,1.4) {};
   \node (Lout) at (1.5,-1.3) {};

   \draw (in) -- (times) -- (out) (Lin) -- (neg) (S) -- (coZ)
   (neg.io) -- (S.left in) (Lcap) -- (Lout);
   \path (Rin) edge[bend left=90] (Lcap);
   \draw[rounded corners] (Rin) -- (Rcor) -- (S.right in);
   \draw[color=red, dashed, thin] (-0.5,1.2) rectangle (1.25,-0.9);
   \end{tikzpicture}
       \hspace{0.69cm}
   \begin{tikzpicture}[thick]
   \node (in1) at (-2.3,0.85) {};
   \node (in2) at (-1.3,0.85) {};
   \node [plus] (S) at (-1.8,0.2) {};
   \node (out) at (-1.8,-0.45) {};
   \node (eq) at (-1,0.2) {\(=\)};
   \node [coordinate] (Lin) at (-0.5,1.016) {};
   \node (Linup) at (-0.5,1.7) {};
   \node [coordinate] (Rin) at (0.5,1.016) {};
   \node (Rinup) at (0.5,1.7) {};
   \node [multiply] (neg) at (1,0.7) {\(\scriptstyle{-1}\)};
   \node [coordinate] (Lcap) at (1,1.3) {};
   \node [plus] (Stop) at (0,0.366) {};
   \node [plus] (Sbot) at (0.5,-0.5) {};
   \node [zero] (coZ) at (0.5,-1.15) {};
   \node [coordinate] (Rcap) at (2,1.3) {};
   \node (Lout) at (2,-1.3) {};

   \draw[rounded corners] (Linup) -- (Lin) -- (Stop.left in)
   (Rinup) -- (Rin) -- (Stop.right in);
   \draw (in1) -- (S.left in) (S.right in) -- (in2) (S) -- (out)
   (Stop.io) -- (Sbot.left in)
   (Sbot) -- (coZ) (neg.io) -- (Sbot.right in) (neg) -- (Lcap)
   (Rcap) -- (Lout);
   \path (Lcap) edge[bend left=90] (Rcap);
   \draw[color=red, dashed, thin] (-0.65,1.2) rectangle (1.65,-0.9);
   \end{tikzpicture}
    }

    \scalebox{0.80}{
   \begin{tikzpicture}[thick]
   \node [zero] (zero1) at (-0.2,1) {};
   \node (eq) at (0.4,0) {\(=\)};
   \node [zero] (coz) at (1.55,-0.8) {};
   \node [hole] (placeholder) at (0,-1.9) {};

   \draw[color=red, dashed, thin] (2,0.5) rectangle (1.1,-0.5);
   \draw (zero1) -- +(down:2.2) (coz) -- (1.55,0.5)
   .. controls +(90:0.75) and +(90:0.75) .. (2.5,0.5) -- (2.5,-1.2);
   \end{tikzpicture}
       \hspace{0.7cm}
   \begin{tikzpicture}[thick]
   \node (out1) at (-2.3,0) {};
   \node (out2) at (-1.3,0) {};
   \node [delta] (D) at (-1.8,0.65) {};
   \node (in) at (-1.8,1.3) {};
   \node (eq) at (-1,0.65) {\(=\)};
   \node [coordinate] (in1) at (0,2.5) {};
   \node [multiply] (neg) at (0,1.7) {\(\scriptstyle{-1}\)};
   \node [delta] (D1) at (0,0.65) {};
   \node [plus] (S1) at (0,-0.65) {};
   \node [plus] (S2) at (1,-0.65) {};
   \node [zero] (coZ1) at (0,-1.3) {};
   \node [zero] (coZ2) at (1,-1.3) {};

   \node (cross) at (0.5,0) {};
   \node [coordinate] (Loutup) at (1,0.866) {};
   \node [coordinate] (Routup) at (1.5,0) {};
   \node [coordinate] (oLcap) at (1,2.3) {};
   \node [coordinate] (iLcap) at (1.5,2.3) {};
   \node [coordinate] (iRcap) at (2,2.3) {};
   \node [coordinate] (oRcap) at (2.5,2.3) {};
   \node (Lout) at (2,-1.3) {};
   \node (Rout) at (2.5,-1.3) {};

   \draw[rounded corners=10pt] (S1.right in) -- (Loutup) -- (oLcap)
   (iLcap) -- (Routup) -- (S2.right in);
   \draw (out1) -- (D.left out) (D.right out) -- (out2) (D) -- (in)
   (in1) -- (neg) -- (D1) (D1.right out) -- (cross) -- (S2.left in)
   (S2) -- (coZ2) (S1) -- (coZ1) (iRcap) -- (Lout) (oRcap) -- (Rout);
   \path (oLcap) edge[bend left=90] (oRcap) (iLcap) edge[bend left=90]
   (iRcap) (D1.left out) edge[bend right=30] (S1.left in);
   \draw[color=red, dashed, thin] (-0.65,2.2) rectangle (1.75,-1);
   \end{tikzpicture}
       \hspace{0.7cm}
   \begin{tikzpicture}[thick]
   \node [bang] (del1) at (0,-1) {};
   \node (eq) at (0.7,0) {\(=\)};
   \node (rel1) at (0.7,-0.4) {(26)};
   \node [bang] (del2) at (1.7,0.3) {};
   \node [zero] (zero) at (1.7,-0.3) {};
   \node [zero] (coz) at (1.7,-1) {};
   \node [hole] (placeholder) at (0,-1.9) {};

   \draw[color=red, dashed, thin] (1.25,-0.5) rectangle (2.15,0.5);
   \draw (del1) -- +(up:2) (del2) -- +(up:0.7) (coz) -- (zero);
   \end{tikzpicture}
    }
  \end{center}
Braiding is two copies of multiplication by 1 that have been braided together.
  \begin{center}
    \scalebox{0.80}
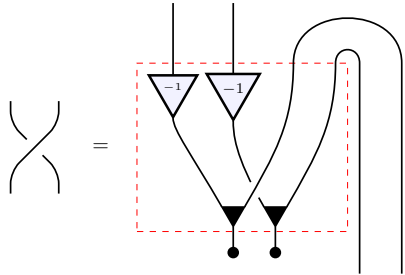
{
   \begin{tikzpicture}[thick]
   \node (UpUpLeft) at (-3.7,2) {};
   \node [coordinate] (UpLeft) at (-3.7,1.5) {};
   \node (mid) at (-3.3,1.1) {};
   \node [coordinate] (DownRight) at (-2.9,0.7) {};
   \node (DownDownRight) at (-2.9,0.2) {};
   \node [coordinate] (UpRight) at (-2.9,1.5) {};
   \node (UpUpRight) at (-2.9,2) {};
   \node [coordinate] (DownLeft) at (-3.7,0.7) {};
   \node (DownDownLeft) at (-3.7,0.2) {};

   \draw [rounded corners=2mm] (UpUpLeft) -- (UpLeft) -- (mid) --
   (DownRight) -- (DownDownRight) (UpUpRight) -- (UpRight) -- (DownLeft) -- (DownDownLeft);

   \node (eq) at (-2.2,1.1) {\(=\)};
   \node [plus] (sum1) at (0,0) {};
   \node [plus] (sum2) at (0.7,0) {};
   \node [zero] (coz1) at (0,-0.65) {};
   \node [zero] (coz2) at (0.7,-0.65) {};
   \node [multiply] (neg1) at (-1,2.066) {\(\scriptstyle^{-1}\)};
   \node [multiply] (neg2) at (0,2.066) {\(\scriptstyle{-1}\)};

   \draw[color=red, dashed, thin] (-1.6,-0.3) rectangle (1.9,2.5);
   \draw (sum2.left in) .. controls +(120:1) and +(270:0.2) .. (neg2.io);
   \node [hole] (hole) at (0.35,0.39) {};
   \draw (sum1.right in) .. controls +(60:1) and +(270:1) .. (1,2.5)
   .. controls +(90:1) and +(90:1) .. (2.8,2.5) -- (2.8,-1)
   (sum2.right in) .. controls +(60:1) and +(270:1) .. (1.7,2.5)
   .. controls +(90:0.3) and +(90:0.3) .. (2.1,2.5) -- (2.1,-1)
   (sum1.io) -- (coz1) (sum2.io) -- (coz2)
   (sum1.left in) .. controls +(120:1) and +(270:0.2) .. (neg1.io)
   (neg1) -- (-1,3.5) (neg2) -- (0,3.5);
   \end{tikzpicture}
    }
  \end{center}
Assuming any string diagram with \(j\) basic morphisms can be written in prestandard form, we show
an arbitrary diagram with \(j+1\) basic morphisms can be written in prestandard form as well.  Let
\(S\) be a string diagram on \(j\) basic morphisms, rewritten into prestandard form, with a
maximal \(\Vectk\) subdiagram \(T\).  Several cases are considered: those putting a basic morphism
above \(S\), beside \(S\), and below \(S\).
     \begin{itemize}[leftmargin=1em]
     \item \Define{\(S \of G\) for a basic morphism \(G \neq \cap\)}\\
If a diagram \(G\) is composed above \(S\), \(G\) can combine with \(T\) to make a larger \(\Vectk\)
subdiagram if \(G\) is \(c\), \(\Delta\), \(+\), \(B\), or \(0\), as these are morphisms in
\(\Vectk\).  The generating morphisms \(\cap\), \(\cup\) and \(!\) are not on this list, though a
composition with \(\cup\) (resp. \(!\)) would be equivalent to tensoring by \(\cup\) (resp. \(!\)).
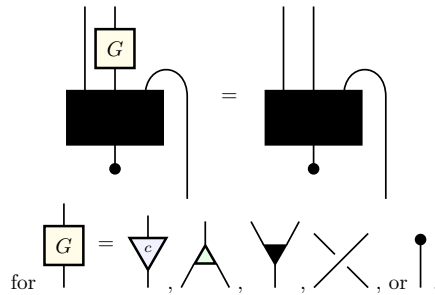
  \begin{center}
    \scalebox{0.80}{
   \begin{tikzpicture}[thick]
   \filldraw[fill=black,draw=black] (-0.1,0.1) rectangle (1.5,-0.8);
   \node (eq) at (2.6,0) {\(=\)};
   \node [sqnode] (G) at (0.7,0.8) {\(G\)};
   \filldraw[fill=black,draw=black] (3.2,0.1) rectangle (4.8,-0.8);
   \node [zero] (Z) at (0.7,-1.2) {};
   \node [zero] (ero) at (4,-1.2) {};
   \draw (0.2,0) -- (0.2,1.5) (0.7,1.5) -- (G) -- (0.7,0) (0.7,-0.7) -- (Z)
   (ero) -- (4,-0.7) (3.5,0) -- (3.5,1.5) (4,0) -- (4,1.5);
   \draw (4.5,0) .. controls +(90:0.6) and +(90:0.6) .. (5.2,0) -- (5.2,-1.7)
         (1.2,0) .. controls +(90:0.6) and +(90:0.6) .. (1.9,0) -- (1.9,-1.7);
   \end{tikzpicture}
    }\\
    \scalebox{0.80}{
for   \begin{tikzpicture}[thick]
   \node [sqnode] (G) at (0,0) {\(G\)};
   \draw (0,0.7) -- (G) -- (0,-0.7);
   \node (eq) at (0.7,0) {\(=\)};
   \end{tikzpicture}
   \begin{tikzpicture}[thick]
   \node [multiply] (c) at (0,0) {\(\scriptstyle{c}\)};
   \draw (0,0.55) -- (c) -- (0,-0.65);
   \end{tikzpicture}, 
   \begin{tikzpicture}[thick]
   \node [delta] (dub) at (0,0) {};
   \draw (0,0.6) -- (dub) (dub.left out) -- (-0.4,-0.5) (dub.right out) -- (0.4,-0.5);
   \end{tikzpicture} , 
   \begin{tikzpicture}[thick]
   \node [plus] (sum) at (0,0) {};
   \draw (0,-0.6) -- (sum) (sum.left in) -- (-0.4,0.5) (sum.right in) -- (0.4,0.5);
   \end{tikzpicture}, 
   \begin{tikzpicture}[thick]
   \draw (0.45,-0.45) -- (-0.45,0.45);
   \node [hole] (hole) at (0,0) {};
   \draw (0.45,0.45) -- (-0.45,-0.45);
   \end{tikzpicture} , or
   \begin{tikzpicture}[thick]
   \node [zero] (zero) at (0,0.4) {};
   \draw (0,-0.4) -- (zero);
   \end{tikzpicture} .
    }
  \end{center}
Putting these morphisms on top of \(S\) reduces to performing those compositions on \(T\).  The
maximal \(\Vectk\) subdiagram now includes \(T\) and \(G\), with \(S\) unchanged outside the
\(\Vectk\) block.
     \item \Define{\(B \of S\)}\\
\(B\) commutes with caps because the category is symmetric monoidal, so capping the braiding is
equivalent to putting the braiding on top of \(T\).  \(B\) is ``absorbed'' into \(T\), just as in
the \(S \of G\) case.
     \item \Define{\(S \oplus G\) for any basic morphism \(G\)}\\
If any two prestandard string diagrams \(S\) and \(S'\) are tensored together, the result combines
into one prestandard diagram.  This is evident because the category of string diagrams is symmetric
monoidal, and the \(\Vectk\) blocks can be placed next to each other as the tensor of two \(\Vectk\)
blocks.  These combine into a single \(\Vectk\) block, and absorbing all the braidings into this
block as above brings the diagram into prestandard form.  Since each basic morphism can be written
as a prestandard diagram, the tensor \(S \oplus G\) is a special case of this.
  \begin{center}
    \scalebox{0.80}
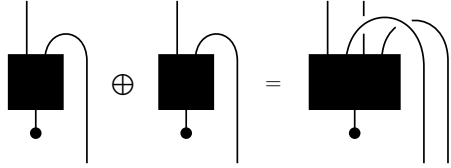
{
   \begin{tikzpicture}[thick]
   \filldraw[fill=black,draw=black] (-0.1,0.1) rectangle (0.8,-0.8);
   \node (oplus) at (1.8,-0.4) {\(\bigoplus\)};
   \filldraw[fill=black,draw=black] (2.4,0.1) rectangle (3.3,-0.8);
   \node (eq) at (4.3,-0.4) {\(=\)};
   \filldraw[fill=black,draw=black] (4.9,0.1) rectangle (6.4,-0.8);
   \node [zero] (Z) at (0.35,-1.2) {};
   \node [zero] (e) at (2.85,-1.2) {};
   \node [zero] (ro) at (5.65,-1.2) {};
   \draw (0.2,0) -- (0.2,1) (0.35,-0.7) -- (Z) (ro) -- (5.65,-0.7)
   (e) -- (2.85,-0.7) (2.7,0) -- (2.7,1) (5.2,0) -- (5.2,1)
   (5.8,0) -- (5.8,1);
   \draw (3,0) .. controls +(90:0.6) and +(90:0.6) .. (3.7,0) -- (3.7,-1.7)
         (0.5,0) .. controls +(90:0.6) and +(90:0.6) .. (1.2,0) -- (1.2,-1.7)
         (6.1,0) .. controls +(90:0.9) and +(90:0.9) .. (7.2,0) -- (7.2,-1.7);
   \node [hole] (ho) at (5.8,0.6) {};
   \node [hole] (le) at (6.46,0.63) {};
   \draw (5.5,-0.1) .. controls +(90:1.1) and +(90:1.1) .. (6.8,-0.1) -- (6.8,-1.7);
   \end{tikzpicture}
    }
  \end{center}
     \item \Define{\(c \of S\) for \(c \neq 0\)}\\
Because the outputs of \(S\) are capped, putting any morphism on the bottom of \(S\) is equivalent
(via relations \textbf{(19)} and \textbf{(20)}) to putting its adjoint on top of \(T\).  Putting \(c
\neq 0\) below \(S\) reduces to putting \(c^{-1}\) on top of \(T\) by relation \textbf{(31)}.  The
case of \(c = 0\) will be considered below.  The other cases of adjoints of generating morphisms
that need to be considered more carefully are the ones that put \(\Delta^\dagger\), \(+^\dagger\)
and \(\cap = \cup^\dagger\) on top of \(T\).
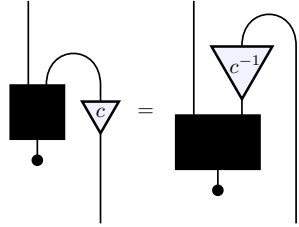
  \begin{center}
    \scalebox{0.80}{
   \begin{tikzpicture}[thick]
   \filldraw[fill=black,draw=black] (-0.45,0.45) rectangle (0.45,-0.45);
   \node[multiply] (c) at (1.05,0) {\(c\)};
   \node[zero] (coz1) at (0,-0.8) {};

   \draw[out=-90,in=-90,relative,looseness=2]
   (c.io) -- (1.05,-1.85) (c.90) -- (1.05,0.45) to (0.15,0.45)
   (coz1) -- +(0,0.5) (-0.15,0.35) -- +(0,1.5);

   \node (eq) at (1.8,0) {\(=\)};
   \filldraw[fill=black,draw=black] (2.3,-0.05) rectangle (3.7,-0.95);
   \node[zero] (coz2) at (3,-1.3) {};
   \node[multiply] (cinv) at (3.4,0.8) {\(c^{-1}\!\!\)};

   \draw[out=90,in=90,relative,looseness=2]
   (2.6,-0.15) -- +(0,2) (cinv.io) -- +(0,-0.5) (coz2) -- +(0,0.5)
   (cinv.90) to +(0.9,0) -- (4.3,-1.85);
   \end{tikzpicture}
    }
  \end{center}
     \item \Define{\(\Delta \of S\)}\\
When putting \(\Delta^\dagger\) on top of \(T\), the idea is to make it ``trickle down.''  If there
is a nonzero multiplication incident to the \(\Delta\) cluster, \(\Delta^\dagger\) can slide through
the \(\Delta\)s using relation \textbf{(23)} to the first nonzero multiplication, switching to
relation \textbf{(24)}.  When it encounters this \(c\), relation \textbf{(31)} turns \(c\) into
\((c^{-1})^\dagger\), relation \textbf{(17)}\({}^\dagger\) allows \(\Delta^\dagger\) to pass through
\((c^{-1})^\dagger\).  Both copies of \((c^{-1})^\dagger\) can return to being \(c\) by another
application of relation \textbf{(31)}, and the \(\Delta^\dagger\) moves on to the next layer.
  \begin{center}
    \scalebox{0.80}
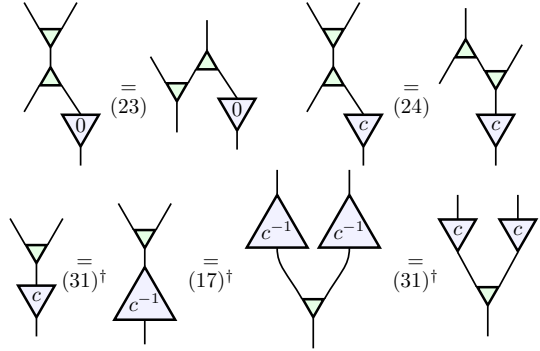
{
   \begin{tikzpicture}[thick]
   \node[codelta] (cod1) at (0,0) {};
   \node[delta] (dub1) at (0,-0.734) {};
   \node[multiply] (zip1) at (0.5,-1.5) {\(0\)};

   \draw (cod1.left in) -- +(120:0.5) (cod1.right in) -- + (60:0.5) (cod1.io) -- (dub1.io)
   (dub1.right out) -- (zip1.90) (zip1.io) -- +(0,-0.3) (dub1.left out) -- +(240:0.5);

   \node (eq) at (1.3,-0.9) {\(=\)};
   \node (rel) at (1.3,-1.2) {(23)};
   \node[delta] (dub2) at (2.6,-0.468) {};
   \node[codelta] (cod2) at (2.1,-0.9) {};
   \node[multiply] (zip2) at (3.1,-1.234) {\(0\)};

   \draw (dub2.io) -- +(90:0.5) (dub2.right out) -- (zip2.90) (zip2.io) -- +(0,-0.3)
   (dub2.left out) -- (cod2.right in) (cod2.left in) -- +(120:0.5) (cod2.io) -- +(0,-0.5);
   \end{tikzpicture}
\qquad
   \begin{tikzpicture}[thick]
   \node[codelta] (cod1) at (0,0) {};
   \node[delta] (dub1) at (0,-0.734) {};
   \node[multiply] (zip1) at (0.5,-1.5) {\(c\)};

   \draw (cod1.left in) -- +(120:0.5) (cod1.right in) -- + (60:0.5) (cod1.io) -- (dub1.io)
   (dub1.right out) -- (zip1.90) (zip1.io) -- +(0,-0.3) (dub1.left out) -- +(240:0.5);

   \node (eq) at (1.3,-0.9) {\(=\)};
   \node (rel) at (1.3,-1.2) {(24)};
   \node[delta] (dub2) at (2.2,-0.268) {};
   \node[codelta] (cod2) at (2.7,-0.7) {};
   \node[multiply] (zip2) at (2.7,-1.5) {\(c\)};

   \draw (dub2.io) -- +(90:0.5) (dub2.left out) -- +(240:0.5) (cod2.io) -- (zip2.90)
   (zip2.io) -- +(0,-0.3) (dub2.right out) -- (cod2.left in) (cod2.right in) -- +(60:0.5);
   \end{tikzpicture}
    }
    \scalebox{0.80}{
   \begin{tikzpicture}[thick]
   \node[multiply] (c) at (0,0) {\(c\)};
   \node[codelta] (cod1) at (0,0.75) {};
   \node (eq1) at (0.8,0.55) {\(=\)};
   \node (rel1) at (0.8,0.25) {(31)\({}^\dagger\)};

   \draw (c.90) -- (cod1.io) (cod1.left in) -- +(120:0.5) (cod1.right in) -- +(60:0.5)
   (c.io) -- +(0,-0.3);

   \node[upmultiply] (cinv2) at (1.8,-0.1) {\(\!c^{-1}\!\)};
   \node[codelta] (cod2) at (1.8,1) {};
   \node (eq2) at (2.9,0.55) {\(=\)};
   \node (rel2) at (2.9,0.25) {(17)\({}^\dagger\)};

   \draw (cinv2.io) -- (cod2.io) (cod2.left in) -- +(120:0.5) (cod2.right in) -- +(60:0.5)
   (cinv2.270) -- +(0,-0.4);

   \node[upmultiply] (cinv3a) at (4,1.1) {\(\!c^{-1}\!\)};
   \node[upmultiply] (cinv3b) at (5.2,1.1) {\(\!c^{-1}\!\)};
   \node[codelta] (cod3) at (4.6,-0.15) {};
   \node (eq3) at (6.3,0.55) {\(=\)};
   \node (rel3) at (6.3,0.25) {(31)\({}^\dagger\)};

   \draw (cinv3a.io) -- +(0,0.4) (cinv3b.io) -- +(0,0.4) (cod3.io) -- +(0,-0.5)
   (cod3.left in) .. controls +(120:0.5) and +(270:0.3) .. (cinv3a.270)
   (cod3.right in) .. controls +(60:0.5) and +(270:0.3) .. (cinv3b.270);

   \node[multiply] (c4a) at (7,1.1) {\(c\)};
   \node[multiply] (c4b) at (8,1.1) {\(c\)};
   \node[codelta] (cod4) at (7.5,0.05) {};

   \draw (c4a.90) -- +(0,0.4) (c4b.90) -- +(0,0.4) (cod4.io) -- +(0,-0.5)
   (cod4.left in) -- (c4a.270) (cod4.right in) -- (c4b.270);
   \end{tikzpicture}
    }
  \end{center}
When the codelta gets to a \(+\) cluster, derived relation \textbf{(D5)} has a net effect of
bringing it to the bottom of the subdiagram, as the other morphisms involved all belong to
\(\Vectk\).  This allows the process to be repeated on the next addition until \(\Delta^\dagger\)
reaches the bottom of the \(+\) cluster.  Once there, codelta interacts with the cozero layer below
\(T\); relation \textbf{(8)}\({}^\dagger\) reduces it to a pair of cozeros.
  \begin{center}
    \scalebox{0.80}
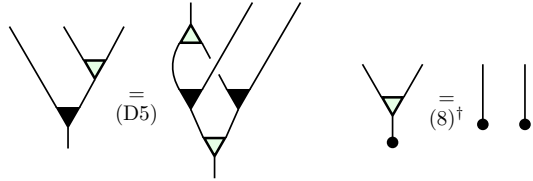
{
   \begin{tikzpicture}[-, thick, node distance=0.7cm]
   \node (in1z) {};
   \node (in2z) [right of=in1z, shift={(0.2,0)}] {};
   \node (in3z) [right of=in2z, shift={(0.45,0)}] {};
   \node [codelta] (nabzip) [below right of=in2z, shift={(0.1,-0.3)}] {};
   \node [plus] (add) [below left of=nabzip, shift={(0.05,-0.3)}] {};
   \node (outz) [below of=add] {};
   \node (equal) [below right of=nabzip, shift={(0.2,0)}] {\(=\)};
   \node (rel) [below of=equal, shift={(0,0.4)}] {(D5)};
   \node [plus] (addl) [right of=equal, shift={(0.2,0)}] {};
   \node (cross) [above right of=addl, shift={(-0.1,0)}] {};
   \node [delta] (delta) [above left of=cross, shift={(0.1,0)}] {};
   \node (in1u) [above of=delta] {};
   \node [plus] (addr) [below right of=cross, shift={(-0.1,0)}] {};
   \node [codelta] (nablunzip) [below left of=addr, shift={(0.1,-0.3)}] {};
   \node (outu) [below of=nablunzip] {};
   \node (in2u) [right of=in1u, shift={(0.4,0)}] {};
   \node (in3u) [right of=in2u, shift={(0.1,0)}] {};

   \draw (in1z) -- (add.left in) (add) -- (outz) (in2z) --
   (nabzip.left in) (in3z) -- (nabzip.right in) (nabzip.io) -- (add.right in);
   \path
   (delta.left out) edge [bend right=30] (addl.left in);
   \draw (in1u) -- (delta) (delta.right out) -- (cross) -- (addr.left in);
   \draw (in2u) -- (addl.right in) (in3u) -- (addr.right in);
   \draw (addl.io) -- (nablunzip.left in) (addr.io) -- (nablunzip.right in)
   (nablunzip) -- (outu);
   \end{tikzpicture}
   \qquad
   \begin{tikzpicture}[thick]
   \node [codelta] (nabla) at (0,-0.65) {};
   \node [zero] (Z1) at (0,-1.3) {};
   \node [coordinate] (il1) at (-0.5,0) {};
   \node [coordinate] (ir1) at (0.5,0) {};
   \node (eq) at (0.9,-0.6) {\(=\)};
   \node (rel) at (0.9,-0.9) {(8)\({}^\dagger\)};
   \node [coordinate] (il2) at (1.5,0) {};
   \node [coordinate] (ir2) at (2.2,0) {};
   \node [zero] (ZL) [below of=il2] {};
   \node [zero] (ZR) [below of=ir2] {};
   \node [hole] (space) at (0,-2) {};

   \draw (il1) -- (nabla.left in) (nabla) -- (Z1) (nabla.right in) -- (ir1)
   (il2) -- (ZL) (ir2) -- (ZR);
   \end{tikzpicture}
    }
  \end{center}
If all the multiplications incident to the \(\Delta\) cluster are by \(0\), rather than trickling
down, \(\Delta^\dagger\) composes with \(!\) (due to relation \textbf{(14)}), which gives \(\cup\)
by relation \textbf{(30)}\({}^\dagger\).  By the zigzag identities, this cup becomes a cap that is
tensored with a subdiagram of \(S\) that is in prestandard form.
  \begin{center}
    \scalebox{0.80}
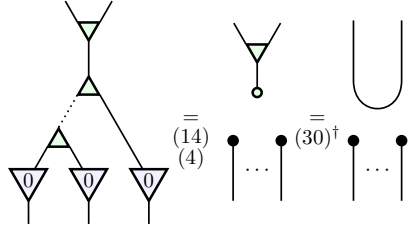
{
   \begin{tikzpicture}[thick]
   \node[codelta] (top2) at (1,0.866) {};
   \node[delta] (delt1) at (1,-0.134) {};
   \node[delta] (delt2) at (0.5,-1) {};
   \node[multiply] (c1) at (0,-1.65) {\(0\)};
   \node[multiply] (c2) at (1,-1.65) {\(0\)};
   \node[multiply] (c3) at (2,-1.65) {\(0\)};
   \node (11b) at (0,-2.516) {};
   \node (21b) at (1,-2.516) {};
   \node (n1b) at (2,-2.516) {};

   \draw[dotted] (delt1.left out) -- (delt2.io);
   \draw (delt2.left out) -- (c1) -- (11b) (delt2.right out) -- (c2) -- (21b)
   (delt1.right out) -- (c3) -- (n1b) (top2.io) -- (delt1.io)
   (top2.left in) -- +(120:0.4) (top2.right in) -- +(60:0.4);

   \node (eq) at (2.7,-0.65) {\(=\)};
   \node (rel1) at (2.7,-0.95) {(14)};
   \node (rel2) at (2.7,-1.3) {(4)};

   \node[codelta] (cod) at (3.8,0.5) {};
   \node[bang] (bang) at (3.8,-0.2) {};
   \node[zero] (ze) at (3.4,-1) {};
   \node[zero] (ro) at (4.2,-1) {};
   \node (dots) at (3.8,-1.5) {\(\cdots\)};

   \draw (cod.io) -- (bang) (ze) -- +(0,-1) (ro) -- +(0,-1)
   (cod.left in) -- +(120:0.4) (cod.right in) -- +(60:0.4);

   \node (eq2) at (4.8,-0.65) {\(=\)};
   \node (rel3) at (4.8,-0.95) {(30)\({}^\dagger\)};
   \node[zero] (ze1) at (5.4,-1) {};
   \node[zero] (ro1) at (6.2,-1) {};
   \node (dots1) at (5.8,-1.5) {\(\cdots\)};

   \draw[out=-90,in=-90,relative,looseness=2]
   (5.4,1) -- ++(0,-1) to +(0.8,0) -- +(0,1)
   (ze1) -- +(0,-1) (ro1) -- +(0,-1);
   \end{tikzpicture}
    }
  \end{center}
     \item \Define{\(+ \of S\)}\\
There is a similar trickle down argument for \(+^\dagger\).  First rewriting all multiplications by
zero via relation \textbf{(14)}, the two \(\Delta\) clusters incident to the coaddition can either
reduce to \(\Delta\) clusters that are incident only to nonzero multiplications or reduce to a
single deletion, as above, if none of the incident multiplications were nonzero.  There are three
cases of what can happen from here.
        \begin{itemize}[leftmargin=1em]
        \item \Define{Both \(\Delta\) clusters were incident to only zero-multiplications}\\
In the first case, as above, the \(\Delta\) clusters will reduce to \(!\) incident to the outputs of
\(+^\dagger\).  Relations \textbf{(D7)} and \textbf{(28)} delete the coaddition.
        \item \Define{One \(\Delta\) cluster was incident to only zero-multiplications}\\
Without loss of generality, the \(!\) incident to \(+^\dagger\) is on the left.  Relation
\textbf{(D7)} replaces \(!\) and \(+^\dagger\) with \(!^\dagger \of !\), and relation \textbf{(30)}
replaces \(\Delta\) and \(!^\dagger\) with a cap.  The \(\Delta\) was -- and the cap is -- incident
to some multiplication by \(c \ne 0\).  Without loss of generality, \(c\) is incident to the bottom
addition in the cluster.  Relation \textbf{(29)} replaces the addition and cozero with a cup and
multiplication by \(-1\), which combines with \(c\) by relation \textbf{(11)}.  The cup and cap turn
\(-c\) around to its adjoint, which is \(-c^{-1}\) by relation \textbf{(31)}.
  \begin{center}
    \scalebox{0.80}
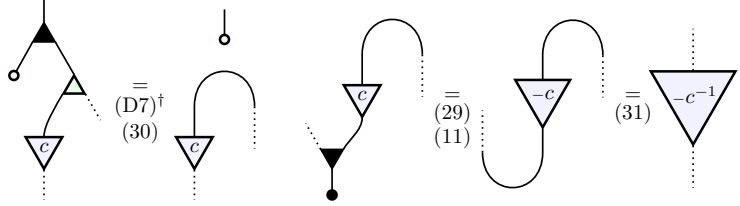
{
   \begin{tikzpicture}[thick]
   \node[coplus] (cop) at (0.5,0.866) {};
   \node[delta] (dub) at (1,0) {};
   \node[bang] (cobig) at (0,0.2) {};
   \node[multiply] (c1) at (0.5,-1) {\(c\)};
   \node (eq) at (2.1,0) {\(=\)};
   \node (rel1a) at (2.1,-0.3) {(D7)\({}^\dagger\)};
   \node (rel1b) at (2.1,-0.75) {(30)};

   \draw (cop.io) -- +(0,0.5) (cop.left out) -- (cobig) (cop.right out) -- (dub.io)
   (dub.left out) .. controls +(240:0.2) and +(90:0.2) .. (c1.90);
   \draw[dotted] (c1.io) -- +(0,-0.5) (dub.right out) -- +(300:0.5);

   \node[bang] (cobang) at (3.5,0.8) {};
   \node[multiply] (c2) at (3,-1) {\(c\)};

   \draw[out=90,in=90,relative,looseness=2] (cobang) -- +(0,0.5) (c2.90) -- ++(0,0.5) to +(1,0);
   \draw[dotted] (c2.io) -- +(0,-0.5) (c2.90) ++(1,0.5) -- +(0,-0.7);
   \end{tikzpicture}
\qquad
   \begin{tikzpicture}[thick]
   \node[multiply] (c) at (1,-1) {\(c\)};
   \node[plus] (sum) at (0.5,-2) {};
   \node[zero] (coz) at (0.5,-2.65) {};
   \node (eq1) at (2.5,-1) {\(=\)};
   \node (rel1) at (2.5,-1.3) {(29)};
   \node (rel1b) at (2.5,-1.7) {(11)};

   \draw[out=90,in=90,relative,looseness=2] (c.90) -- ++(0,0.5) to +(1,0) (sum.io) -- (coz)
   (c.io) .. controls +(270:0.2) and +(60:0.2) .. (sum.right in);
   \draw[dotted] (sum.left in) -- +(120:0.5) (c.90) ++(1,0.5) -- +(0,-0.7);

   \node[multiply] (negc) at (4,-1) {\(\scriptstyle{-}\)\(c\)};
   \node (eq2) at (5.5,-1) {\(=\)};
   \node (rel2) at (5.5,-1.3) {(31)};

   \draw[out=90,in=90,relative,looseness=2] (negc.90) -- ++(0,0.4) to +(1,0)
   (negc.io) to (negc.io) -- ++(0,-0.4) to +(-1,0);
   \draw[dotted] (negc.90) ++(1,0.4) -- +(0,-0.7) (negc.io) ++(-1,-0.4) -- +(0,0.7);

   \node[multiply] (inv) at (6.5,-1) {\(\!\scriptstyle{-}\)\(c^{-1}\!\!\)};
   \draw[dotted] (inv.90) -- +(0,0.7) (inv.io) -- +(0,-0.7);
   \end{tikzpicture}
    }
  \end{center}
An addition cluster is above \(-c^{-1}\) and a duplication cluster is below, but because those
clusters are not otherwise connected to each other, there is a vertical arrangement of the morphisms
in the \(\Vectk\) block of the string diagram such that no cups or caps are present.
        \item \Define{Both \(\Delta\) clusters are incident to at least one nonzero multiplication}\\
Using relation \textbf{(D5)}\({}^\dagger\), a \(+^\dagger\) will pass through one \(\Delta\) at a
time.  A new \(\Delta^\dagger\) is created each time, but this can trickle down as before.
  \begin{center}
    \scalebox{0.80}
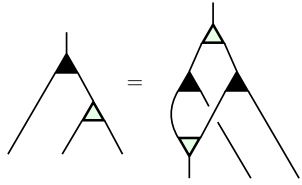
{
   \begin{tikzpicture}[-, thick, node distance=0.7cm]
   \node (out1z) {};
   \node (out2z) [right of=out1z, shift={(0.2,0)}] {};
   \node (out3z) [right of=out2z, shift={(0.45,0)}] {};
   \node [delta] (delzip) [above right of=out2z, shift={(0.1,0.3)}] {};
   \node [coplus] (coadd) [above left of=delzip, shift={(0.05,0.3)}] {};
   \node (inz) [above of=coadd] {};
   \node (equal) [above right of=delzip, shift={(0.2,0)}] {\(=\)};
   \node [coplus] (coaddl) [right of=equal, shift={(0.2,0)}] {};
   \node (cross) [below right of=coaddl, shift={(-0.1,0)}] {};
   \node [codelta] (nabla) [below left of=cross, shift={(0.1,0)}] {};
   \node (out1u) [below of=nabla] {};
   \node [coplus] (coaddr) [above right of=cross, shift={(-0.1,0)}] {};
   \node [delta] (deltunzip) [above left of=coaddr, shift={(0.1,0.3)}] {};
   \node (inu) [above of=deltunzip] {};
   \node (out2u) [right of=out1u, shift={(0.4,0)}] {};
   \node (out3u) [right of=out2u, shift={(0.1,0)}] {};

   \draw (out1z) -- (coadd.right in) (coadd) -- (inz) (out2z) --
   (delzip.left out) (out3z) -- (delzip.right out) (delzip.io) -- (coadd.left in);
   \path
   (nabla.left in) edge [bend left=30] (coaddl.right in);
   \draw (out1u) -- (nabla) (nabla.right in) -- (coaddr.right in);
   \draw (out2u) -- (cross)  -- (coaddl.left in) (out3u) -- (coaddr.left in);
   \draw (coaddl.io) -- (deltunzip.left out) (coaddr.io) -- (deltunzip.right out) (deltunzip) -- (inu);
   \end{tikzpicture}
    }
  \end{center}
Once the \(\Delta^\dagger\) trickles down, there are two possibilities for what is directly beneath
each \(+^\dagger\): either the same scenario will recur with a \(\Delta\) connected to one or both
outputs, which can only happen finitely many times, or two nonzero multiplications will be below the
\(+^\dagger\).  A multiplication by any unit in \(k\), \(c \ne 0\), can move through a coaddition by
inserting \(c c^{-1}\) on the top branch and applying relation \textbf{(15)}\({}^\dagger\):
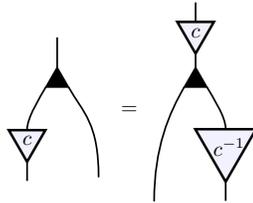
  \begin{center}
    \scalebox{0.80}{
   \begin{tikzpicture}[thick]
   \node[coplus] (cop1) at (0.5,1) {};
   \node[multiply] (c1) at (0,0) {\(c\)};
   \node (eq) at (1.7,0.5) {\(=\)};
   \node[multiply] (c2) at (2.8,1.8) {\(c\)};
   \node[coplus] (cop2) at (2.8,1) {};
   \node[multiply] (cinv) at (3.3,-0.1) {\(c^{-1}\!\!\)};

   \draw (cop1.io) -- +(0,0.5) (cop1.right out) .. controls +(300:0.5) and +(90:0.9) .. +(0.5,-1.5)
   (cop1.left out) .. controls +(240:0.2) and +(90:0.2) .. (c1.90) (c1.io) -- +(0,-0.3)
   (c2) -- +(0,0.5) (c2.io) -- (cop2.io) (cop2.left out) .. controls +(240:0.5) and +(90:0.9) .. +(-0.5,-1.9)
   (cop2.right out) .. controls +(300:0.2) and +(90:0.2) .. (cinv.90) (cinv.io) -- +(0,-0.3);
   \end{tikzpicture}
    }
  \end{center}
This allows one of the outputs of the coaddition to connect directly to a \(+\) cluster.
         \begin{itemize}[leftmargin=1em]
         \item \Define{If both branches go to different \(+\) clusters}, Frobenius relations
\textbf{(21)--(22)} slide the \(+^\dagger\) down the \(+\) cluster on one side until it gets to the
end of that cluster.
  \begin{center}
    \scalebox{0.80}
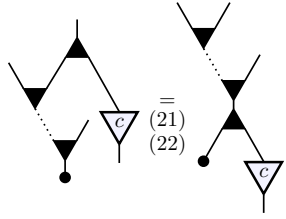
{
   \begin{tikzpicture}[thick]
   \node[plus] (P1a) {};
   \node[plus] (P1b) at (-0.5,0.866) {};
   \node[coplus] (cop1) at (0.2,1.61) {};
   \node[multiply] (c1) at (0.9,0.41) {\(c\)};
   \node[zero] (out1) at (0,-0.5) {};
   \node (eq) at (1.7,0.75) {\(=\)};
   \node (rel1) at (1.7,0.45) {(21)};
   \node (rel2) at (1.7,0.05) {(22)};

   \draw (P1a.io) -- (out1) (c1.io) -- +(0,-0.3) (cop1.io) -- +(0,0.3) (P1a.right in) -- +(60:0.4)
   (cop1.right out) -- (c1.90) (P1b.right in) -- (cop1.left out) (P1b.left in) -- +(120:0.5);
   \draw[dotted] (P1a.left in) -- (P1b.io);

   \node[plus] (P2a) at (2.8,1) {};
   \node[plus] (P2b) at (2.3,1.866) {};
   \node[coplus] (cop2) at (2.8,0.42) {};
   \node[multiply] (c2) at (3.3,-0.42) {\(c\)};
   \node[zero] (out2) at (2.3,-0.23) {};

   \draw (P2b.left in) -- +(120:0.5) (c2.io) -- +(0,-0.3) (P2a.io) -- (cop2.io) (cop2.left out) -- (out2)
   (cop2.right out) -- (c2.90) (P2b.right in) -- +(60:0.5) (P2a.right in) -- +(60:0.5);
   \draw[dotted] (P2a.left in) -- (P2b.io);
   \end{tikzpicture}
    }
  \end{center}
The only morphisms added to the \(\Vectk\) block that are not from \(\Vectk\) were the coaddition
and the cozero.  Since these reduce to an identity morphism string by relation
\textbf{(1)}\({}^\dagger\), the \(\Vectk\) block is truly a \(\Vectk\) block again.
         \item \Define{If both branches go to the same \(+\) cluster}, relation \textbf{(3)} and the
Frobenius relation \textbf{(21)} take both branches to the same addition.
  \begin{center}
    \scalebox{0.80}
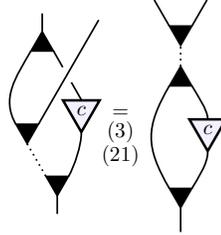
{
   \begin{tikzpicture}[thick]
   \node[coplus] (cop1) at (0,1)              {};
   \node[plus]  (add1a) at (-0.25,-0.433)     {};
   \node[plus]  (add1b) at (0.25,-1.3)        {};
   \node[multiply] (c1) at (0.65,-0.104) {\(c\)};
   \node           (eq) at (1.3,-0.15)   {\(=\)};
   \node         (rel1) at (1.3,-0.45)     {(3)};
   \node         (rel2) at (1.3,-0.85)    {(21)};

   \draw (add1a.left in) .. controls +(120:0.5) and +(240:0.5) .. (cop1.left out)
   (cop1.io) -- +(0,0.3) (cop1.right out) -- (c1.90) (add1b.io) -- +(0,-0.3)
   (c1.io) .. controls +(270:0.2) and +(60:0.2) .. (add1b.right in);
   \node[hole] (cross) at (0.42,0.5) {};
   \draw (add1a.right in) -- +(60:2);
   \draw[dotted] (add1a.io) -- (add1b.left in);

   \node[plus] (topadd) at (2.3,1.2) {};
   \node[coplus] (cop2) at (2.3,0.42) {};
   \node[multiply] (c2) at (2.75,-0.42) {\(c\)};
   \node[plus] (botadd) at (2.3,-1.5) {};

   \draw (topadd.left in) -- +(120:0.5) (topadd.right in) -- +(60:0.5) (botadd.io) -- +(0,-0.4)
   (cop2.right out) .. controls +(300:0.15) and +(90:0.15) .. (c2.90)
   (c2.io) .. controls +(270:0.2) and +(60:0.2) .. (botadd.right in)
   (botadd.left in) .. controls +(120:0.8) and +(240:0.8) .. (cop2.left out);
   \draw[dotted] (cop2.io) -- (topadd.io);
   \end{tikzpicture}
    }
  \end{center}
Depending on whether the remaining multiplication is by \(1\), either relation \textbf{(25)} reduces
the coaddition and the given addition to an identity string or relation \textbf{(D8)} applies.  In
the former case we are done, and in the latter case relations \textbf{(D7)} and
\textbf{(10)}\({}^\dagger\) remove the \(!^\dagger\) introduced by applying relation (D8).
  \begin{center}
    \scalebox{0.80}
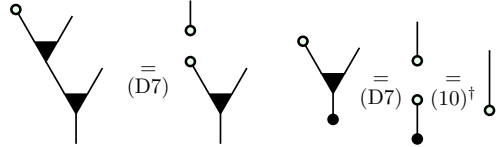
{
   \begin{tikzpicture}[thick]
   \node[bang] (cob1) at (0,1) {};
   \node[plus] (sum1a) at (0.5,0.35) {};
   \node[plus] (sum1b) at (1,-0.516) {};
   \node (eq) at (2.2,0) {\(=\)};
   \node (rel) at (2.2,-0.3) {(D7)};

   \draw (cob1) -- (sum1a.left in) (sum1a.right in) -- +(60:0.5) (sum1a.io) -- (sum1b.left in)
   (sum1b.right in) -- +(60:0.5) (sum1b.io) -- +(0,-0.5);

   \node[bang] (cob2) at (2.9,0.134) {};
   \node[bang] (bang) at (2.9,0.65) {};
   \node[plus] (sum2) at (3.4,-0.516) {};

   \draw (bang) -- +(0,0.5) (cob2) -- (sum2.left in) (sum2.right in) -- +(60:0.5)
   (sum2.io) -- +(0,-0.5);
   \end{tikzpicture}
\qquad
   \begin{tikzpicture}[thick]
   \node[bang] (cob1) at (0,0.65) {};
   \node[plus] (sum1) at (0.5,0) {};
   \node[zero] (coz1) at (0.5,-0.65) {};
   \node (eq1) at (1.3,0) {\(=\)};
   \node (rel1) at (1.3,-0.3) {(D7)};

   \draw (cob1) -- (sum1.left in) (sum1.right in) -- +(60:0.5) (sum1.io) -- (coz1);

   \node[bang] (cob2) at (1.9,-0.325) {};
   \node[bang] (bang) at (1.9,0.325) {};
   \node[zero] (coz2) at (1.9,-0.975) {};
   \node (eq2) at (2.5,0) {\(=\)};
   \node (rel2) at (2.5,-0.3) {(10)\({}^\dagger\)};

   \draw (bang) -- +(0,0.65) (cob2) -- (coz2);

   \node[bang] (cob3) at (3.1,-0.5) {};
   \draw (cob3) -- +(0,1);
   \end{tikzpicture}
    }
  \end{center}
          \end{itemize}
        \end{itemize}
     \item \Define{\(\cup \of S\) and \(S \of \cap\)}\\
Composing with a cup below \(S\) is equivalent to composing with cap above \(T\), since \(\cap =
\cup^\dagger\).  Using relation \textbf{(D10)}\({}^\dagger\), this cap can be replaced by multiplication by
\(-1\), coaddition, and zero.  By the arguments above, \(-1\), \(+^\dagger\), and \(0\) can each be
absorbed into the \(\Vectk\) block.
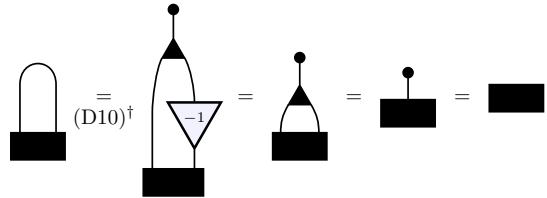
  \begin{center}
    \scalebox{0.80}{
   \begin{tikzpicture}[thick]
   \filldraw[fill=black,draw=black] (-0.4,-0.8) rectangle (0.5,-0.35);
   \draw[out=90,in=90,relative,looseness=2] (-0.25,-0.45) -- ++(0,0.9) to +(0.6,0) -- +(0,-0.9);
   \node (eq1) at (1.15,0.225) {\(=\)};
   \node (rel) at (1.15,-0.1) {(D10)\({}^\dagger\)};

   \filldraw[fill=black,draw=black] (1.8,-1.4) rectangle (2.8,-0.95);
   \node[multiply] (neg) at (2.65,-0.1) {\(\scriptstyle{-1}\)};
   \node[coplus] (cops) at (2.3,1) {};
   \node[zero] (zero1) at (2.3,1.7) {};
   \node (eq2) at (3.5,0.225) {\(=\)};

   \draw (zero1) -- (cops.io) (neg.io) -- (2.65,-1.05)
   (cops.right out) .. controls +(300:0.2) and +(90:0.2) .. (neg.90)
   (cops.left out) .. controls +(240:0.2) and +(90:0.2) .. (1.95,0) -- (1.95,-1.05);

   \filldraw[fill=black,draw=black] (3.95,-0.8) rectangle (4.85,-0.35);
   \node[coplus] (robbers) at (4.4,0.225) {};
   \node[zero] (zero2) at (4.4,0.9) {};
   \node (eq3) at (5.3,0.225) {\(=\)};

   \draw (zero2) -- (robbers.io)
   (robbers.left out) .. controls +(240:0.2) and +(90:0.2) .. +(255:0.5)
   (robbers.right out) .. controls +(300:0.2) and +(90:0.2) .. +(285:0.5);

   \filldraw[fill=black,draw=black] (5.75,0.2) rectangle (6.65,-0.25);
   \node[zero] (zero3) at (6.2,0.65) {};
   \draw (zero3) -- +(0,-0.5);
   \node (eq4) at (7.1,0.225) {\(=\)};
   \filldraw[fill=black,draw=black] (7.55,0.45) rectangle (8.45,0);
   \end{tikzpicture}
    }
  \end{center}
The compositions with zero and multiplication by \(-1\) expand the \(\Vectk\) block, thus have no
effect on whether the diagram can be written in prestandard form.
     \item \Define{\(! \of S\)}\\
When composing \(!^\dagger\) above \(T\), two possibilities arise, depending on whether there is a
layer of \(\Delta\)s in the \(\Vectk\) block.  If there is such a layer, relation \textbf{(30)} combines the
\(!^\dagger\) with a \(\Delta\), making a cap on top of \(T\).  As we have just seen, this can be
rewritten in prestandard form.
  \begin{center}
    \scalebox{0.80}{
   \begin{tikzpicture}[thick]
   \filldraw[fill=black,draw=black] (-0.45,0.45) rectangle (0.45,0);
   \node[bang] (cob1) at (0,0.9) {};
   \node (eq1) at (0.9,0.225) {\(=\)};
   \node[delta] (dub) at (1.5,0.225) {};
   \node[bang] (cob2) at (1.5,0.9) {};
   \filldraw[fill=black,draw=black] (1.05,-0.8) rectangle (1.95,-0.35);
   \node (eq2) at (2.5,0.225) {\(=\)};
   \node (rel) at (2.5,-0.075) {(30)};
   \filldraw[fill=black,draw=black] (2.85,-0.8) rectangle (3.75,-0.35);

   \draw[out=90,in=90,relative,looseness=2] (3,-0.45) -- ++(0,0.9) to +(0.6,0) -- +(0,-0.9)
   (cob1) -- +(0,-0.5) (cob2) -- (dub.io)
   (dub.left out) .. controls +(240:0.2) and +(90:0.2) .. +(255:0.5)
   (dub.right out) .. controls +(300:0.2) and +(90:0.2) .. +(285:0.5);
   \end{tikzpicture}
    }
  \end{center}
If no layer of \(\Delta\)s exists, relations \textbf{(31)}\({}^\dagger\) and
\textbf{(18)}\({}^\dagger\) pass the codeletion through a nonzero multiplication.  Then relations
\textbf{(D7)} and \textbf{(10)}\({}^\dagger\) can be used to remove \(!^\dagger\), as we have
already seen.  This leaves only the basic morphisms of \(\Vectk\) within the \(\Vectk\) block.
  \begin{center}
    \scalebox{0.80}{
   \begin{tikzpicture}[thick]
   \node[bang] (cob1) at (0,1) {};
   \node[multiply] (c) at (0,0) {\(c\)};
   \node (eq1) at (0.8,0) {\(=\)};
   \node (rel1) at (0.8,-0.3) {(31)\({}^\dagger\)};

   \draw (cob1) -- (c) -- +(0,-1);

   \node[bang] (cob2) at (1.8,1) {};
   \node[upmultiply] (cinv) at (1.8,0) {\(\!c^{-1}\!\)};
   \node (eq2) at (2.8,0) {\(=\)};
   \node (rel2) at (2.8,-0.3) {(18)\({}^\dagger\)};

   \draw (cob2) -- (cinv) -- +(0,-1);

   \node[bang] (cob3) at (3.4,0.5) {};
   \draw (cob3) -- +(0,-1);
   \end{tikzpicture}
    }
  \end{center}
If the multiplication is \(c=0\), relation \textbf{(14)} converts \(c=0\) to \(0 \of !\), allowing
relation \textbf{(28)} to remove the \(!^\dagger\), with the same conclusion.
  \begin{center}
    \scalebox{0.80}{
   \begin{tikzpicture}[thick]
   \node[bang] (cob1) at (0.3,1) {};
   \node[multiply] (c0) at (0.3,0) {\(0\)};
   \node (eq1) at (1,0) {\(=\)};
   \node (rel1) at (1,-0.3) {(14)};

   \draw (cob1) -- (c0) -- +(0,-1);

   \node[bang] (cob2) at (1.6,0.925) {};
   \node[bang] (bang) at (1.6,0.325) {};
   \node[zero] (zero1) at (1.6,-0.325) {};
   \node (eq2) at (2.2,0) {\(=\)};
   \node (rel2) at (2.2,-0.3) {(28)};

   \draw (cob2) -- (bang) (zero1) -- +(0,-0.65);

   \node[zero] (zero2) at (2.8,0.5) {};
   \draw (zero2) -- +(0,-1);
   \end{tikzpicture}
    }
  \end{center}
     \item \Define{\(c \of S\) for \(c=0\)}\\
Composing with multiplication by \(c = 0\) below \(S\) is equivalent to composing with codeletion,
followed by tensoring with zero.  Codeletion is the \(! \of S\) case, and zero can be written in a
prestandard form, so this reduces to tensoring two diagrams that are in prestandard form.
  \begin{center}
    \scalebox{0.80}{
   \begin{tikzpicture}[thick]
   \filldraw[fill=black,draw=black] (-0.45,0.45) rectangle (0.45,0);
   \node[multiply] (c) at (1.05,0) {\(0\)};

   \draw[out=-90,in=-90,relative,looseness=2]
   (c.io) -- (1.05,-1.35) (c.90) -- (1.05,0.45) to (0.15,0.45);

   \node (eq) at (1.8,0) {\(=\)};
   \node (rel) at (1.8,-0.3) {(14)};
   \filldraw[fill=black,draw=black] (2.3,0.45) rectangle (3.2,0);
   \node[zero] (ins2) at (3.8,-0.3) {};
   \node[bang] (del2) at (3.8,0.3) {};

   \draw[out=-90,in=-90,relative,looseness=2]
   (ins2) -- (3.8,-1.35) (del2) -- (3.8,0.45) to +(-0.9,0);
   \end{tikzpicture}
    }
  \end{center}
     \end{itemize}
Finally, we need to show the prestandard forms can be rewritten in standard form.  We need to show
what elementary row operations look like in terms of string diagrams.  We also need to show for an
arbitrary prestandard string diagram \(S\) with \(\Vectk\) block \(T\) that if \(T\) is replaced
with \(T'\), the diagram where an elementary row operation has been performed on \(T\), the 
resulting diagram \(S'\) can be built from \(S\) using relations \textbf{(1)--(31)}.

Because the \(i\)th output of a \(\Vectk\) diagram is a linear combinations of the inputs, with the
coefficients coming from the \(i\)th row of its matrix, rows of the matrix correspond to outputs of
the \(\Vectk\) block.  Because of this, the row operation subdiagrams in \(S'\) will have
\(0^\dagger\)s immediately beneath them.  Showing \(S'\) can be built from \(S\) reduces to showing
composition of row operations with \(0^\dagger\)s builds the same number of \(0^\dagger\)s.
     \begin{itemize}[leftmargin=1em]
     \item Add a multiple \(c\) of one row to another row:\\
If we want to add a multiple of the \(\beta\) row to the \(\alpha\) row, we need a map \((y_\alpha,
y_\beta) \mapsto (y_\alpha + c y_\beta, y_\beta)\).  By the naturality of the braiding in a
symmetric monoidal category, we can ignore any intermediate outputs:
  
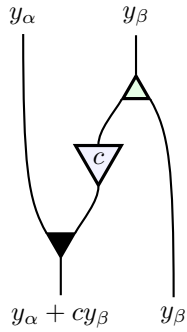
\begin{center}
   \begin{tikzpicture}[thick]
   \node [plus] (plus) at (0,0) {};
   \node [multiply] (c) at (0.5,1.1) {\(c\)};
   \node [delta] (dub) at (1,2) {};
   \node (in1) at (-0.5,3) {\(y_\alpha\)};
   \node (in2) at (1,3) {\(y_\beta\)};
   \node (out1) at (0,-1) {\(y_\alpha + c y_\beta\)};
   \node (out2) at (1.5,-1) {\(y_\beta\)};

   \draw
   (in1) .. controls +(270:2) and +(120:0.5) .. (plus.left in)
   (plus.io) -- (out1) (dub.io) -- (in2)
   (out2) .. controls +(90:2) and +(300:0.5) .. (dub.right out)
   (dub.left out) .. controls +(240:0.2) and +(90:0.2) .. (c.90)
   (c.io) .. controls +(270:0.2) and +(60:0.2) .. (plus.right in);
   \end{tikzpicture}
  \end{center}
When two cozeros are composed on the bottom of this diagram, the result is two cozeros:
  \begin{center}
    \scalebox{0.80}
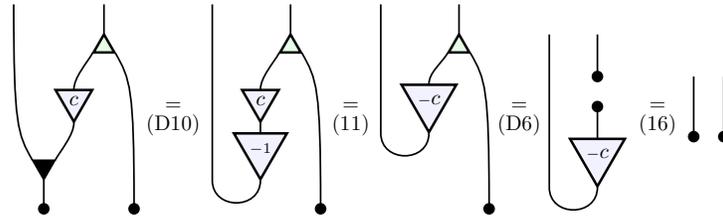
{
   \begin{tikzpicture}[thick]
   \node [plus] (plus) at (-0.3,0) {};
   \node [multiply] (c) at (0.2,1.1) {\(c\)};
   \node [delta] (dub) at (0.7,2) {};
   \node [coordinate] (in1) at (-0.8,2.7) {};
   \node [coordinate] (in2) at (0.7,2.7) {};
   \node [zero] (out1) at (-0.3,-0.7) {};
   \node [zero] (out2) at (1.2,-0.7) {};

   \draw
   (in1) .. controls +(270:2) and +(120:0.5) .. (plus.left in)
   (plus.io) -- (out1) (dub.io) -- (in2)
   (out2) .. controls +(90:2) and +(300:0.5) .. (dub.right out)
   (dub.left out) .. controls +(240:0.2) and +(90:0.2) .. (c.90)
   (c.io) .. controls +(270:0.2) and +(60:0.2) .. (plus.right in);
   \node (eq1) at (1.85,1) {\(=\)};
   \node (rel1) at (1.85,0.7) {(D10)};
   \node [multiply] (c2) at (3.3,1.1) {\(c\)};
   \node [multiply] (neg2) at (3.3,0.3) {\(\scriptstyle{-1}\)};
   \node [delta] (dub2) at (3.8,2) {};
   \node [zero] (coz2) at (4.3,-0.7) {};

   \draw
   (dub2) -- (3.8,2.7) (c2) -- (neg2)
   (dub2.left out) .. controls +(240:0.2) and +(90:0.2) .. (c2.90)
   (coz2) .. controls +(90:2) and +(300:0.5) .. (dub2.right out)
   (neg2.io) .. controls +(270:0.5) and +(270:0.5) .. +(-0.8,0) -- (2.5,2.7);
   \node (eq2) at (4.8,1) {\(=\)};
   \node (rel2) at (4.8,0.7) {(11)};
   \node [multiply] (c3) at (6.1,1.1) {\(\scriptstyle{-}\)\(c\)};
   \node [delta] (dub3) at (6.6,2) {};
   \node [zero] (coz3) at (7.1,-0.7) {};

   \draw
   (dub3) -- (6.6,2.7)
   (dub3.left out) .. controls +(240:0.2) and +(90:0.2) .. (c3.90)
   (coz3) .. controls +(90:2) and +(300:0.5) .. (dub3.right out)
   (c3.io) .. controls +(270:0.5) and +(270:0.5) .. +(-0.8,0) -- (5.3,2.7);
   \node (eq3) at (7.6,1) {\(=\)};
   \node (rel3) at (7.6,0.7) {(D6)};
   \node [multiply] (c4) at (8.9,0.2) {\(\scriptstyle{-}\)\(c\)};
   \node [zero] (zero) at (8.9,1) {};
   \node [zero] (coz4) at (8.9,1.5) {};

   \draw
   (zero) -- (c4) (coz4) -- (8.9,2.2)
   (c4.io) .. controls +(270:0.5) and +(270:0.5) .. +(-0.8,0) -- (8.1,2.2);
   \node (eq4) at (9.9,1) {\(=\)};
   \node (rel4) at (9.9,0.7) {(16)};
   \node [zero] (coz5a) at (10.5,0.5) {};
   \node [zero] (coz5b) at (11,0.5) {};

   \draw (coz5a) -- (10.5,1.5) (coz5b) -- (11,1.5);
   \end{tikzpicture}
    }
  \end{center}
     \item Swap rows:\\
If we want to swap the \(\beta\) row with the \(\alpha\) row, we need a map \((y_\alpha, y_\beta)
\mapsto (y_\beta, y_\alpha)\), which is the braiding of two outputs.  Again, intermediate outputs
may be ignored:
  
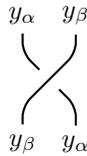
\begin{center}
   \begin{tikzpicture}[thick,node distance=0.5cm]
   \node (fstart) {\(y_\alpha\)};
   \node [coordinate] (ftop) [below of=fstart] {};
   \node (center) [below right of=ftop] {};
   \node [coordinate] (fout) [below right of=center] {};
   \node (fend) [below of=fout] {\(y_\alpha\)};
   \node [coordinate] (gtop) [above right of=center] {};
   \node (gstart) [above of=gtop] {\(y_\beta\)};
   \node [coordinate] (gout) [below left of=center] {};
   \node (gend) [below of=gout] {\(y_\beta\)};

   \draw [rounded corners] (fstart) -- (ftop) -- (center) --
   (fout) -- (fend) (gstart) -- (gtop) -- (gout) -- (gend);
   \end{tikzpicture}
  \end{center}
When two cozeros are composed at the bottom of this diagram, the cut strings untwist by the
naturality of the braiding:
  \begin{center}
    \scalebox{0.80}
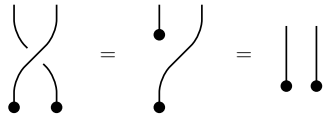
{
   \begin{tikzpicture}[thick,node distance=0.5cm]
   \node [coordinate] (fstart) {};
   \node [coordinate] (ftop) [below of=fstart] {};
   \node (center) [below right of=ftop] {};
   \node [coordinate] (fout) [below right of=center] {};
   \node [zero] (fend) [below of=fout] {};
   \node [coordinate] (gtop) [above right of=center] {};
   \node [coordinate] (gstart) [above of=gtop] {};
   \node [coordinate] (gout) [below left of=center] {};
   \node [zero] (gend) [below of=gout] {};

   \draw [rounded corners=0.25cm] (fstart) -- (ftop) -- (center) --
   (fout) -- (fend) (gstart) -- (gtop) -- (gout) -- (gend);

   \node (eq1) [below right of=gtop, shift={(0.5,0)}] {\(=\)};
   \node [zero] (ftop1) [above right of=eq1, shift={(0.5,0)}] {};
   \node [coordinate] (fstart1) [above of=ftop1] {};
   \node (center1) [below right of=ftop1] {};
   \node [coordinate] (gtop1) [above right of=center1] {};
   \node [coordinate] (gstart1) [above of=gtop1] {};
   \node [coordinate] (gout1) [below left of=center1] {};
   \node [zero] (gend1) [below of=gout1] {};

   \draw [rounded corners=0.25cm] (fstart1) -- (ftop1)
   (gstart1) -- (gtop1) -- (gout1) -- (gend1);

   \node (eq2) [below right of=gtop1, shift={(0.35,0)}] {\(=\)};
   \node [zero] (coz1) [right of=eq2, shift={(0.2,-0.5)}] {};
   \node [zero] (coz2) [right of=coz1] {};

   \draw (coz1) -- +(0,1) (coz2) -- +(0,1); 
   \end{tikzpicture}
    }
  \end{center}
     \item Multiply a row by \(c \neq 0\):\\
The third row operation is multiplying an arbitrary row by a unit, but since \(k\) is a field, that
means any \(c \neq 0\).  This is just the multiplication map on one of the outputs:
  
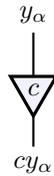
\begin{center}
   \begin{tikzpicture}[thick]
   \node [multiply] (c) at (0,0) {\(c\)};
   \node (in) at (0,1) {\(y_\alpha\)};
   \node (out) at (0,-1) {\(c y_\alpha\)};

   \draw (in) -- (c) -- (out);
   \end{tikzpicture}
  \end{center}
Because \(c\) is a unit, \(c^{-1} \in k\), so the multiplication by \(c\) can be replaced by the
adjoint of multiplication by \(c^{-1}\).
  \begin{center}
    \scalebox{0.80}
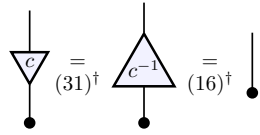
{
   \begin{tikzpicture}[thick]
   \node [multiply] (c) at (0,0) {\(c\)};
   \node (in) at (0,1) {};
   \node [zero] (out) at (0,-1) {};

   \draw (in) -- (c) -- (out);

   \node (eq1) at (0.8,0) {\(=\)};
   \node (rel1) at (0.8,-0.35) {(31)\({}^\dagger\)};
   \node [zero] (coz1) at (1.9,-1) {};
   \node [upmultiply] (c1) at (1.9,-0.1) {\(\!c^{-1}\!\)};

   \draw (coz1) -- (c1) -- (1.9,1);

   \node (eq2) at (3,0) {\(=\)};
   \node (rel2) at (3,-0.35) {(16)\({}^\dagger\)};
   \node [zero] (coz2) at (3.7,-0.5) {};

   \draw (coz2) -- +(0,1);
   \end{tikzpicture}
    }
  \end{center}
     \end{itemize}
\end{proof}

\section{An example}
\label{example}

A famous example in control theory is the `inverted pendulum': an upside-down 
pendulum on a cart \cite{Friedland}.  The pendulum naturally tends to fall over, but we
can stabilize it by setting up a feedback loop where we observe its position and move 
the cart back and forth in a suitable way based on this observation.  Without introducing 
this feedback loop, let us see how signal-flow diagrams can be used to describe the pendulum
and the cart.  We shall see that the diagram for a system made of parts is built 
from the diagrams for the parts, not merely by composing and tensoring, but also with the
help of duplication and coduplication, which give additional ways to set variables equal to one another.

Suppose the cart has mass \(M\) and can only move back and forth in one direction, so its
position is described by a function \(x(t)\).  If it is acted on by a total force 
\(F_{\mathrm{net}}(t)\) then Newton's second law says 
\[
F_{\mathrm{net}}(t) = M \ddot{x}(t)  .
\]
We can thus write a signal-flow diagram with the force as input and the cart's position as output:
  
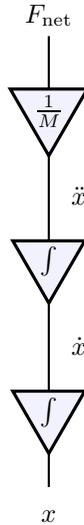
\begin{center}
   \begin{tikzpicture}[thick]
   \node[coordinate] (q) [label={[shift={(0,-0.6)}]\(x\)}] {};
   \node [integral] (diff) [above of=q] {\(\int\)};
   \node (v) [above of=diff, label={[shift={(0.4,-0.5)}]\(\dot x\)}] {};
   \node [integral] (dot) [above of=v] {\(\int\)};
   \node (a) [above of=dot, label={[shift={(0.4,-0.5)}]\(\ddot{x}\)}] {};
   \node [multiply] (M) [above of=a] {\(\frac{1}{M}\)};
   \node[coordinate] (F) [above of=m, label={[shift={(0,0)}]\(F_{\mathrm{net}}\)}] {};

   \draw (F) -- (M) -- (dot) -- (diff) -- (q);
   \end{tikzpicture}
  \end{center}

The inverted pendulum is a rod of length \(\ell\) with a mass \(m\)  at its end, mounted 
on the cart and only able to swing back and forth in one direction, parallel to the cart's movement.  If its angle from vertical, \(\theta(t)\), is small, then its equation of motion 
is approximately linear:
\[
   \ell \ddot{\theta}(t) = g \theta(t) - \ddot{x}(t) 
\]
where \(g\) is the gravitational constant.  We can turn this equation
into a signal-flow diagram with \(\ddot{x}\) as input and \(\theta\) as output:
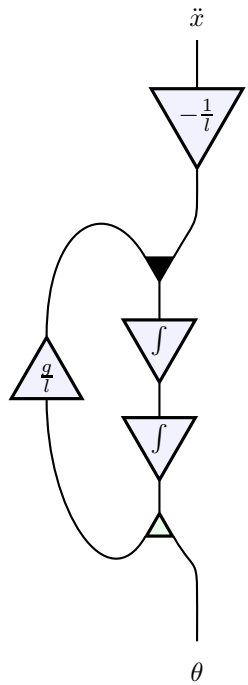
\begin{center}
\begin{tikzpicture}[thick]
   \node [multiply] (linverse) at (0,0) {\(-\frac{1}{l}\)};
   \node [plus] (adder) at (-0.5,-2) {};
   \node [integral] (int1) at (-0.5,-3) {\(\int\)};
   \node [upmultiply] (goverl) at (-2,-3.5) {\(\frac{g}{l}\)};
   \node [integral] (int2) at (-0.5,-4.3) {\(\int\)};
   \node [delta] (split) at (-0.5,-5.5) {};
   \node at (0 em,1.3) {\(\ddot{x}\)}; 
   \node at (0 em,-7.4) {\(\theta\)}; 

   \draw (linverse) -- (0,1)
         (linverse.io) .. controls +(270:0.8) and +(60:0.7) .. (adder.right in)
         (adder.io) -- (int1) -- (int2) -- (split)
         (goverl.io) .. controls +(90:1.5) and +(120:1) .. (adder.left in)
         (split.right out) .. controls +(300:0.7) and +(90:1.3) .. (0,-7)
         (split.left out) .. controls +(240:1) and +(270:1.5) .. (goverl.270);
\end{tikzpicture}
\end{center}
Note that this already includes a kind of feedback loop, since the pendulum's angle affects
the force on the pendulum.  

Finally, there is an equation describing the total force on the cart:
\[ F_{\mathrm{net}}(t)  = F(t) - m g \theta(t)  \]
where \(F(t)\) is an externally applied force and \(-mg\theta(t)\) is the force due to the 
pendulum.  It will be useful to express this as follows:
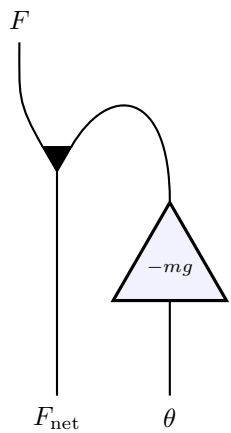
\begin{center}
\begin{tikzpicture}[thick]
   \node [plus] (differ) at (0,0) {};
   \node [upmultiply] (mg) at (1.5,-1.5) {\(\scriptstyle{-mg}\)};
   \node at (0,-3.5) {\(F_{\mathrm{net}}\)}; 
   \node at (1.5,-3.5) {\(\theta\)}; 
   \node at (-0.5,1.8) {\(F\)}; 

   \draw (mg.io) .. controls +(90:1.5) and +(60:1) .. (differ.right in)
         (differ.left in) .. controls +(120:.7) and +(270:.7) .. (-0.5,1.5)
         (differ.io) -- (0,-3.2) 
         (mg) -- (1.5,-3.2);
\end{tikzpicture}
\end{center}
Here we are treating \(\theta\) as an output rather than an input, with the help of a cap.

The three signal-flow diagrams above describe the following linear relations:
\begin{eqnarray}
x &=& \int \int \frac{1}{M}  F_{\mathrm{net}}   \label{eq.1}  \\
\theta &=& \int \int 
\left(\frac{g}{l} \,\theta - \frac{1}{l}\, \ddot{x}\right)  \label{eq.2} \\
 F_{\mathrm{net}} + mg \theta &=&  F \label{eq.3}
\end{eqnarray}
where we treat \eqref{eq.1} as a relation with \(F_{\mathrm{net}}\) as input and
\(x\) as output, \eqref{eq.2} as a relation with \(\ddot{x}\) as input and \(\theta\)
as output, and \eqref{eq.3} as a relation with \(F\) as input and \((F_{\mathrm{net}}, 
\theta)\) as output.   

To understand how the external force affects the position of the cart and the angle of the pendulum,
we wish to combine all three diagrams to form a signal-flow diagram that has the external force
\(F\) as input and the pair \((x, \theta)\) as output.  This is not just a simple matter of
composing and tensoring the three diagrams.  We can take \(F_{\mathrm{net}}\), which is an output of
\eqref{eq.3}, and use it as an input for \eqref{eq.1}.  But we also need to duplicate \(\ddot{x}\),
which appears as an intermediate variable in \eqref{eq.1} since \(\ddot{x} = \frac{1}{M}
F_{\mathrm{net}}\), and use it as an input for \eqref{eq.2}.  Finally, we need to take the variable
\(\theta\), which appears as an output of both \eqref{eq.2} and \eqref{eq.3}, and identify the two
copies of this variable using coduplication.  Following traditional engineering practice, we shall
write coduplication in terms of duplication and a cup, as follows:
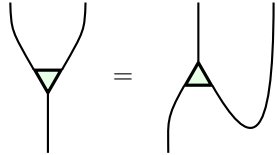
\begin{center}
\begin{tikzpicture}[thick]
\node [codelta] (split) at (-2,0) {};
\node [] at (-1,0) {\(=\)};
 \node [delta] (theta) at (0,0) {};
\draw     (theta.left out) .. controls +(240:0.5) and +(90:0.4) .. (-0.4,-1)
         (theta.right out) .. controls +(300:1) and +(270:2) .. (1,1)
         (theta.io) -- (0,1)
         (split.left in) .. controls +(120:0.6) .. (-2.5,1)
         (split.right in) .. controls +(60:0.6) .. (-1.5,1)
         (split.io) -- (-2,-1);
\end{tikzpicture}
\end{center}

The result is this signal-flow diagram:
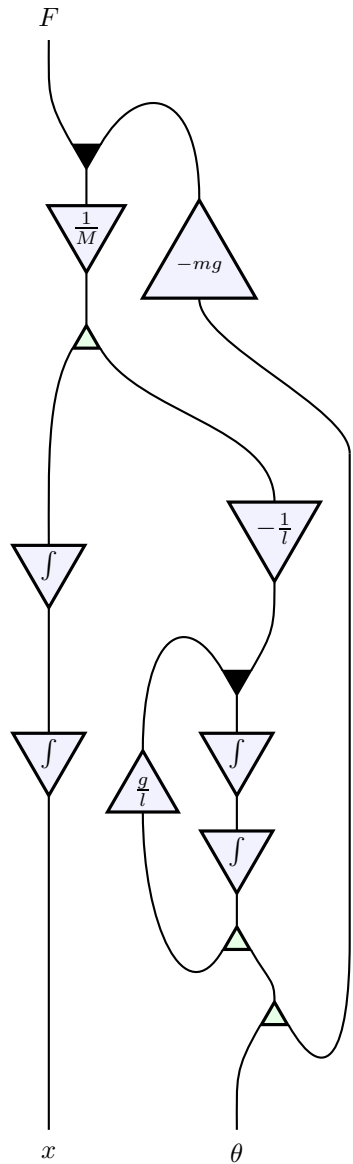
\vfill \eject 
\begin{center}
\begin{tikzpicture}[thick]
   \node [plus] (differ) at (0,0) {};
   \node [upmultiply] (mg) at (1.5,-1.5) {\(\scriptstyle{-mg}\)};
   \node [delta] (split) at (0,-2.5) {};
   \node [multiply] (Minv) at (0,-1) {\(\frac{1}{M}\)};
   \node at (-0.5,1.8) {\(F\)}; 

   \draw (mg.io) .. controls +(90:1.5) and +(60:1) .. (differ.right in)
         (differ.left in) .. controls +(120:.7) and +(270:.7) .. (-0.5,1.5)
         (differ.io) -- (Minv.90)
         (Minv.io) -- (split.io)
         (mg.270) .. controls +(270:0.6) and +(90:0.6) .. (3.5,-4);

   \node [integral] (int1) at (-0.5,-5.5) {\(\int\)};
   \node [integral] (int2) at (-0.5,-8) {\(\int\)};
   \node at (-0.5 cm,-13.3) {\(x\)}; 

   \draw (split.left out) .. controls +(240:0.7) and +(90:1) .. (int1.90)
         (int1.io) -- (int2.90)
         (int2.io) -- (-0.5,-13);

   \node [multiply] (linverse) at (2.5,-5) {\(-\frac{1}{l}\)};
   \node [plus] (adder) at (2,-7) {};
   \node [integral] (int3) at (2,-8) {\(\int\)};
   \node [upmultiply] (goverl) at (0.75,-8.5) {\(\frac{g}{l}\)};
   \node [integral] (int4) at (2,-9.3) {\(\int\)};
   \node [delta] (split2) at (2,-10.5) {};
   \node [delta] (theta) at (2.5,-11.5) {};
   \node at (2 cm,-13.3) {\(\theta\)}; 

   \draw (split.right out) .. controls +(300:1) and +(90:1) .. (linverse.90) 
         (linverse.io) .. controls +(270:0.6) and +(60:0.6) .. (adder.right in)
         (adder.io) -- (int3) -- (int4) -- (split2)
         (goverl.io) .. controls +(90:1.5) and +(120:1) .. (adder.left in)
         (split2.right out) .. controls +(300:0.5) and +(90:0.5) .. (theta)
         (theta.left out) .. controls +(240:0.7) and +(90:0.7) .. (2,-13)
         (theta.right out) .. controls +(300:1) and +(270:1.5) .. (3.5,-10.5)
         (3.5,-10.5) -- (3.5,-4)
         (split2.left out) .. controls +(240:1) and +(270:1.5) .. (goverl.270);
\end{tikzpicture}
\end{center}

This is not the signal-flow diagram for the inverted pendulum that one sees in
Friedland's textbook on control theory \cite{Friedland}.  We leave it as an exercise to the reader to 
rewrite the above diagram using the rules given in this paper, obtaining Friedland's diagram:

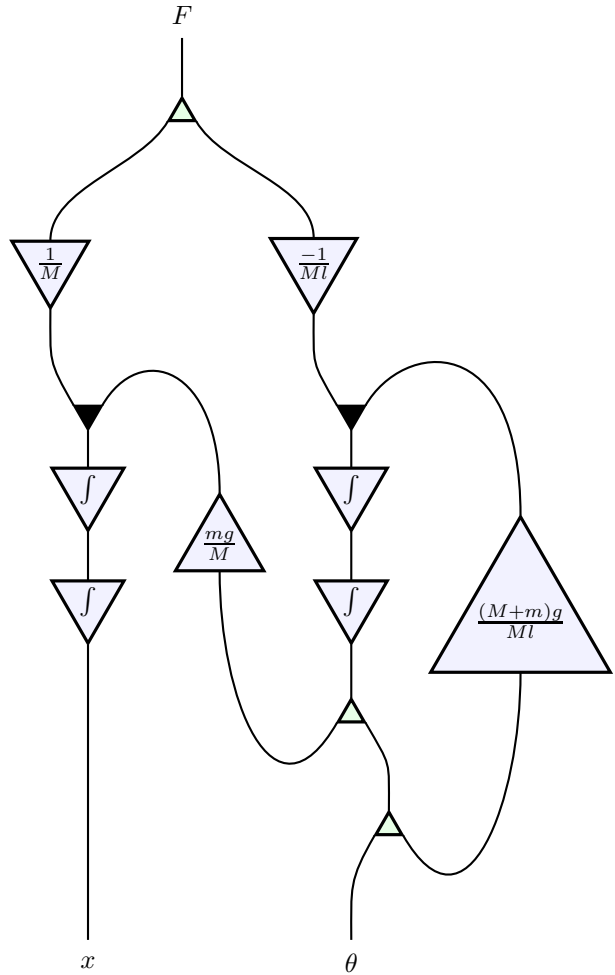
\vfill \eject  

\begin{center}
\begin{tikzpicture}[thick]
   \node [delta] (F) at (0,10) {};
   \node at (0,11.3) {\(F\)}; 
   \node [multiply] (Minv) at (-1.75,8) {\(\frac{1}{M}\)};
   \node [multiply] (Mlinv) at (1.75,8) {\(\frac{-1}{Ml}\)};
   \node [plus] (xsum) at (-1.25,6) {};
   \node [plus] (thsum) at (2.25,6) {};
   \node [integral] (i1) at (-1.25,5) {\(\int\)};
   \node [integral] (i2) at (-1.25,3.5) {\(\int\)};
   \node at (-1.25,-1.3) {\(x\)}; 
   \node [integral] (i3) at (2.25,5) {\(\int\)};
   \node [integral] (i4) at (2.25,3.5) {\(\int\)};
   \node [upmultiply] (mgM) at (0.5,4.25) {\(\frac{mg}{M}\)};
   \node [upmultiply] (mess) at (4.5,3.25) {\(\!\!\frac{(M+m)g}{Ml}\!\!\)};
   \node [delta] (theta1) at (2.25,2) {};
   \node [delta] (theta2) at (2.75,0.5) {};
   \node at (2.25,-1.3) {\(\theta\)}; 

   \draw (F) -- (0,11)
         (F.left out) .. controls +(240:0.7) and +(90:0.7) .. (Minv.90)
         (F.right out) .. controls +(300:0.7) and +(90:0.7) .. (Mlinv.90)
         (Minv.io) .. controls +(270:0.7) and +(120:0.7) .. (xsum.left in)
         (Mlinv.io) .. controls +(270:0.7) and +(120:0.7) .. (thsum.left in)
         (mgM.io) .. controls +(90:1.7) and +(60:1) .. (xsum.right in)
         (mess.io) .. controls +(90:2.5) and +(60:1) .. (thsum.right in)
         (xsum) -- (i1) -- (i2) -- (-1.25,-1)
         (thsum) -- (i3) -- (i4) -- (theta1)
         (mgM.270) .. controls +(270:1.7) and +(240:1.7) .. (theta1.left out)
         (mess.270) .. controls +(270:1.7) and +(300:1.7) .. (theta2.right out)
         (theta2.io) .. controls +(90:0.7) and +(300:0.7) .. (theta1.right out)
         (theta2.left out) .. controls +(240:0.7) and +(90:0.7) .. (2.25,-1)
;
\end{tikzpicture}
\end{center}
As a start, one can use Theorem \ref{presrk} to prove that it is indeed possible to do this rewriting.  
To do this, simply check that both signal-flow diagrams define the same linear relation.  The proof of
the theorem gives a method to actually do the rewriting---but not necessarily the fastest method.

\section{Conclusions}
\label{conclusions}

We conclude with some remarks aimed at setting our work in context.
In particular, we would like to compare it to some other recent papers.  On April 30th, 2014, after
most of this paper was written, Soboci\'nski told the first author about some closely related 
papers that he wrote with Bonchi and Zanasi \cite{BSZ1,BSZ2}.  These provide
interesting characterizations of symmetric monoidal categories equivalent to \(\Vectk\) and
\(\Relk\).  Later, while this paper was being refereed, Wadsley and Woods \cite{WW} generalized the first of
these results to the case where \(k\) is any commutative rig.  We discuss Wadsley and Woods'
work first, since doing so makes the exposition simpler.

A particularly tractable sort of symmetric monoidal category is a PROP: that is, a strict symmetric
monoidal category where the objects are natural numbers and the tensor product of objects is given
by ordinary addition.  The symmetric monoidal category \(\Vectk\) is equivalent to the PROP
\(\Mat(k)\), where a morphism \(f \maps m \to n\) is an \(n \times m\) matrix with entries in
\(k\), composition of morphisms is given by matrix multiplication, and the tensor product of
morphisms is the direct sum of matrices.

Wadsley and Woods gave an elegant description of the algebras of \(\Mat(k)\).  Suppose \(C\) is a
PROP and \(D\) is a strict symmetric monoidal category.  Then the \Define{category of algebras} of
\(C\) in \(D\) is the category of strict symmetric monoidal functors \(F \maps C \to D\) and
natural transformations between these.  If for every choice of \(D\) the category of algebras of
\(C\) in \(D\) is equivalent to the category of algebraic structures of some kind in \(D\), we say
\(C\) is the PROP for structures of that kind. 

In this language, Wadsley and Woods proved that \(\Mat(k)\) is the PROP for `bicommutative
bimonoids over \(k\)'.  To understand this, first note that for any bicommutative bimonoid \(A\) in
\(D\), the bimonoid endomorphisms of \(A\) can be added and composed, giving a rig \(\End(A)\).  A
bicommutative bimonoid \Define{over \(k\)} in \(D\) is one equipped with a rig homomorphism
\(\Phi_A \maps k \to \End(A)\).  Bicommutative bimonoids over \(k\) form a category where a
morphism \(f \maps A \to B\) is a bimonoid homomorphism compatible with this extra structure,
meaning that for each \(c \in k\) the square
\[ 
\xymatrix{ A  \ar[dd]_f \ar[rr]^{\Phi_A(c)} && A \ar[dd]^f \\  \\
B  \ar[rr]_{\Phi_B(c)} && B
 } 
\] 
commutes.  Wadsley and Woods proved that this category is equivalent to the category of algebras of
\(\Mat(k)\) in \(D\).   

This result amounts to a succinct restatement of Theorem \ref{presvk}, though technically the
result is a bit different, and the style of proof much more so.  The fact that an algebra of
\(\Mat(k)\) is a bicommutative bimonoid is equivalent to our relations \textbf{(1)--(10)}.  The
fact that \(\Phi_A(c)\) is a bimonoid homomorphism for all \(c \in k\) is equivalent to relations
\textbf{(15)--(18)}, and the fact that \(\Phi\) is a rig homomorphism is equivalent to relations
\textbf{(11)--(14)}.  

Even better, Wadsley and Woods showed that \(\Mat(k)\) is the PROP for bicommutative bimonoids over
\(k\) whenever \(k\) is a commutative rig.  Subtraction and division are not required to define the
PROP \(\Mat(k)\), nor are they relevant to the definition of bicommutative bimonoids over \(k\).
Working with commutative rigs is not just generalization for the sake of generalization: it
clarifies some interesting facts.

For example, the commutative rig of natural numbers gives a PROP \(\Mat(\N)\).  This is equivalent
to the symmetric monoidal category where morphisms are isomorphism classes of spans of finite sets,
with disjoint union as the tensor product.  Lack \cite[Ex.\ 5.4]{Lack} had already shown that this
is the PROP for bicommutative bimonoids.  But this also follows from the result of Wadsley and
Woods, since every bicommutative bimonoid \(A\) is automatically equipped with a unique rig
homomorphism \(\Phi_A \maps \N \to \End(A)\).

Similarly, the commutative rig of booleans \(\B = \{F,T\}\), with `or' as addition and `and' as
multiplication, gives a PROP \(\Mat(\B)\).  This is equivalent to the symmetric monoidal category
where morphisms are relations between finite sets, with disjoint union as the tensor product.
Mimram \cite[Thm.\ 16]{Mimram} had already shown this is the PROP for \Define{special}
bicommutative bimonoids, meaning those where comultiplication followed by multiplication is the
identity:
\begin{center}
   \begin{tikzpicture}[thick]
   \node [plus] (sum) at (0.4,-0.5) {};
   \node [delta] (cosum) at (0.4,0.5) {};
   \node [coordinate] (in) at (0.4,1) {};
   \node [coordinate] (out) at (0.4,-1) {};
   \node (eq) at (1.3,0) {\(=\)};
   \node [coordinate] (top) at (2,1) {};
   \node [coordinate] (bottom) at (2,-1) {};

   \path (sum.left in) edge[bend left=30] (cosum.left out)
   (sum.right in) edge[bend right=30] (cosum.right out);
   \draw (top) -- (bottom)
   (sum.io) -- (out)
   (cosum.io) -- (in);
   \end{tikzpicture}
  \end{center}
But again, this follows from the general result of Wadsley and Woods.

Finally, taking the commutative ring of integers \(\Z\), Wadsley and Woods showed that \(\Mat(\Z)\)
is the PROP for bicommutative Hopf monoids.  The key here is that scalar multiplication by \(-1\)
obeys the axioms for an antipode, namely:
  \begin{center}
   \begin{tikzpicture}[thick]
  \node [delta] (cosum) at (0,1.8) {};
   \node [multiply] (times) at (-0.34,0.97) {\tiny \(-1\)};
   \node [plus] (sum) at (0,0) {};

   \draw
   (cosum.io) -- +(0,0.3) (sum.io) -- +(0,-0.3)
   (cosum.right out) .. controls +(300:0.5) and +(60:0.5) .. (sum.right in)
   (cosum.left out) .. controls +(240:0.15) and +(90:0.15) .. (times.90)
   (times.io) .. controls +(270:0.15) and +(120:0.15) .. (sum.left in);

   \node (eq) at (1.2,0.9) {\(=\)};
   \node [bang] (bang) at (2.1,1.3) {};
   \node [zero] (cobang) at (2.1,0.5) {};

   \draw (bang) -- +(0,1.02) (cobang) -- +(0,-1.02);

   \node (eq) at (3,0.9) {\(=\)};
   \node [delta] (cosum) at (4.2,1.8) {};
   \node [multiply] (times) at (4.54,0.97) {\tiny \(-1\)};
   \node [plus] (sum) at (4.2,0) {};

   \draw
   (cosum.io) -- +(0,0.3) (sum.io) -- +(0,-0.3)
   (cosum.left out) .. controls +(240:0.5) and +(120:0.5) .. (sum.left in)
   (cosum.right out) .. controls +(300:0.15) and +(90:0.15) .. (times.90)
   (times.io) .. controls +(270:0.15) and +(60:0.15) .. (sum.right in);
  
   \end{tikzpicture}
   \end{center}
More generally, whenever \(k\) is a commutative ring, the presence of \(-1 \in k\) guarantees that
a bimonoid over \(k\) is automatically a Hopf monoid over \(k\).  So, when \(k\) is a commutative
ring, Wadsley and Woods' result implies that \(\Mat(k)\) is the PROP for Hopf monoids over \(k\).  

Earlier, Bonchi, Soboci\'nski and Zanasi gave an elegant and very different proof that \(\Mat(R)\)
is the PROP for Hopf monoids over \(R\) when \(R\) is a principal ideal domain
\cite[Prop.\ 3.7]{BSZ1}.  The advantage of their argument is that they build up the PROP for Hopf
monoids over \(R\) from smaller pieces, using some ideas developed by \cite{Lack}.  

These authors also described a PROP that is equivalent to \(\Relk\) as a symmetric monoidal
category whenever \(k\) is a field.  In this PROP, which they call \(\SV_k\), a morphism \(f \maps
m \to n\) is a linear relation from \(k^m\) to \(k^n\).  They proved that \(\SV_k\) is a pushout in
the category of PROPs and strict symmetric monoidal functors:
\[ 
\xymatrix{ \Mat(R) + \Mat(R)^{\mathrm{op}} \ar[dd] \ar[rr] && \mathrm{Span}(\Mat(R)) \ar[dd] \\  \\
\mathrm{Cospan}(\Mat(R))  \ar[rr] && \SV_k
 } 
\] 

This pushout square requires a bit of explanation.
Here \(R\) is any principal ideal domain whose field of fractions is \(k\).  For example, we 
could take \(R = k\), though Bonchi, Soboci\'nski and Zanasi  are more interested in the example
where  \(R = \R[s]\) and \(k = \R(s)\).  A morphism in \(\mathrm{Span}(\Mat(R))\) is an isomorphism 
class of spans in \(\Mat(R)\).  There is a covariant functor 
\[      \begin{array}{ccc} \Mat(R) &\to& \mathrm{Span}(\Mat(R))   \\ 
                                    m \stackrel{f}{\to} n &\mapsto & m \stackrel{1}{\leftarrow} m \stackrel{f}{\to} n \end{array} \]
and also a contravariant functor
\[      \begin{array}{ccc} \Mat(R) &\to& \mathrm{Span}(\Mat(R))   \\ 
                                    m \stackrel{f}{\to} n &\mapsto & n \stackrel{f}{\leftarrow} m \stackrel{1}{\to} m. \end{array} \]
Putting these together we get the functor from \(\Mat(R) + \Mat(R)^{\mathrm{op}}\) to
\(\mathrm{Span}(\Mat(R))  \) that gives the top edge of the square.  Similarly, a morphism in
\(\mathrm{Cospan}(\Mat(R))\) is an isomorphism class of cospans in \(\Mat(R)\), and we have both a
covariant functor 
\[      \begin{array}{ccc} \Mat(R) &\to& \mathrm{Cospan}(\Mat(R))   \\
                                     m \stackrel{f}{\to} n &\mapsto & m \stackrel{f}{\rightarrow} n \stackrel{1}{\leftarrow} n \end{array} \]
and a contravariant functor
\[      \begin{array}{ccc} \Mat(R) &\to& \mathrm{Cospan}(\Mat(R)) \\ 
                                    m \stackrel{f}{\to} n &\mapsto & n \stackrel{1}{\rightarrow} n \stackrel{f}{\leftarrow} m. \end{array} \]
Putting these together we get the functor from \( \Mat(R) + \Mat(R)^{\mathrm{op}} \) to
\( \mathrm{Cospan}(\Mat(R)) \) that gives the left edge of the square. 

Bonchi, Soboci\'nski and Zanasi analyze this pushout square in detail, giving explicit
presentations for each of the PROPs involved, all based on their presentation of \(\Mat(R)\).  The
upshot is a presentation of \(\SV_k\) which is very similar to our presentation of the equivalent
symmetric monoidal category \(\Relk\).   Their methods allow them to avoid many, though not all, of
the lengthy arguments that involve putting morphisms in `normal form'.


\subsection*{Acknowledgements}

We thank Jamie Vicary for pointing out the relevance of the ZX calculus when the first author gave
a talk on this material at Oxford in February 2014 \cite{B}.  Discussions with Brendan Fong have
also been very useful.  

\bibliographystyle{plain}

\end{document}